\documentclass[11pt,leqno,a4paper]{amsart}
\usepackage{xypic}
\usepackage{epsfig}
\usepackage{tikz}
\usepackage{graphicx}
\usepackage{amsthm}
\usepackage{amssymb}
\usepackage{caption}
\usepackage{subcaption}
\usepackage{color, enumerate}
\usepackage{hyperref}
\usepackage{amsmath,mathrsfs,amsfonts,verbatim,enumitem,leftidx}
\newtheorem{theorem}{Theorem}[section]

\newtheorem{proposition}[theorem]{Proposition}

\newtheorem{lemma}[theorem]{Lemma}
\theoremstyle{definition}
\newtheorem{definition}[theorem]{Definition}
\newtheorem{example}[theorem]{Example}
\newtheorem{remark}[theorem]{Remark}

\numberwithin{equation}{section}

\providecommand{\G}{\mathcal{G}}
\providecommand{\F}{\mathcal{F}}

\providecommand{\D}{\mathcal{D}}

\providecommand{\B}{\mathcal{B}}
\providecommand{\A}{\mathcal{A}}
\providecommand{\T}{T}

\providecommand{\N}{\mathcal{B}}

\providecommand{\NN}{K}

\providecommand{\i}{\mathfrak{i}}

\begin{document}

\title{Continued fraction algorithm for Sturmian colorings of trees}

\date{September 20, 2016}
\author{Dong Han Kim}
\address{Department of Mathematics Education, Dongguk University--Seoul, 30 Pildong-ro 1-gil, Jung-gu, Seoul, 04620
} \email{kim2010@dongguk.edu}
\author{Seonhee Lim}
\address{Department of Mathematical Sciences, Seoul National University, Kwanak-ro 1, Kwanak-gu, Seoul
08826} \email{slim@snu.ac.kr}

\begin{abstract}  
Factor complexity $b_n(\phi)$ for a vertex coloring $\phi$ of a regular tree is the number of classes of $n$-balls up to color-preserving automorphisms. Sturmian colorings are colorings of minimal unbounded factor complexity $b_n(\phi) = n+2$. 

In this article, we prove an induction algorithm for Sturmian colorings using colored balls in a way analogous to the continued fraction algorithm for Sturmian words. Furthermore, we characterize Sturmian colorings in terms of the data appearing in the induction algorithm. 
\end{abstract}
\maketitle
\section{Introduction}
Factor complexity (also called subword complexity or block complexity) $b_n(u)$ of an (one-sided) infinite word $u$, which counts the number of distinct $n$-subwords, has been studied for a long time in symbolic dynamics  \cite{Fo}. 
A classical theorem of Hedlund and Morse says that Sturmian sequences, which are sequences of minimal unbounded complexity $b_n(u)=n+1$, correspond to irrational rotations \cite{HM}. 

Factor complexity was generalized from sequences to vertex colorings of a regular tree in \cite{KL1} and \cite{KL2}. Fo a graph $X$, let us denote its vertex set by $VX$ and its set of oriented edges by $EX$.
%\begin{definition} 
For a vertex coloring $\phi : V\T \to \A$  of a $d$-regular tree $\T$, let $\mathbf{B}_n(\phi)$ the set of classes of $n$-balls appearing in $T$ colored by $\phi$ up to coloring-preserving isomorphisms of $n$-balls. 
The \emph{factor complexity  $b_n(\phi)$ (called subword complexity in \cite{KL1}) } is the cardinality of the set $\mathbf{B}_n(\phi)$. 
%Let us denote the tree $T$ colored by $\phi$ by $T_\phi$.

%\end{definition}
A coloring $\phi$ is called \emph{periodic} if $b_n(\phi)$ is bounded. 
A coloring $\phi_g$ associated to an automorphism $g$ of a uniform tree $T$ was constructed in \cite{LMZ} and \cite{ALN} so that $g$ is a commensurator element of a cocompact lattice $\Gamma \subset G=Aut(T)$ if and only if $\phi_g$ is a periodic coloring. (See \cite{LMZ},  \cite{BM}, \cite{Serre}, Chapter 6 of \cite{BL} for commensurators.) 
One can associate a vertex coloring to any automorphism of a tree, thus classifying vertex colorings of trees may give us a tool to classify automorphisms of a tree.

\begin{figure}
\includegraphics[width=5cm]{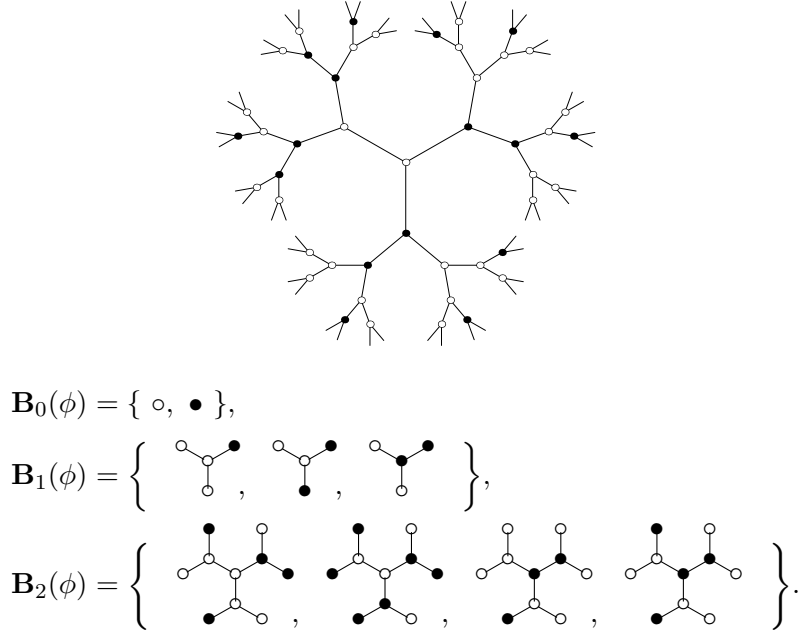}
%\begin{minipage}
\begin{align*}
{\mathbf B}_0(\phi) &= \{ \ \circ, \ \bullet \ \}, \\
{\mathbf B}_1(\phi) &= \bigg \{ \quad 
\raisebox{-.2\height}{\begin{tikzpicture}[scale=.4]
  \tikzstyle{every node}=[inner sep=-1.5pt]
  \node (0) at (0,0) {$\circ$};
  \node (1) at (.87,.5) {$\bullet$};
  \node (2) at (-.87,.5) {$\circ$};
  \node (3) at (0,-1) {$\circ$};
  \path[-] (0) edge (1) (0) edge (2) (0) edge (3);
\end{tikzpicture},} \quad
\raisebox{-.2\height}{\begin{tikzpicture}[scale=.4]
  \tikzstyle{every node}=[inner sep=-1.5pt]
  \node (0) at (0,0) {$\circ$};
  \node (1) at (.87,.5) {$\bullet$};
  \node (2) at (-.87,.5) {$\circ$};
  \node (3) at (0,-1) {$\bullet$};
  \path[-] (0) edge (1) (0) edge (2) (0) edge (3);
\end{tikzpicture},} \quad
\raisebox{-.2\height}{\begin{tikzpicture}[scale=.4]
  \tikzstyle{every node}=[inner sep=-1.5pt]
  \node (0) at (0,0) {$\bullet$};
  \node (1) at (.87,.5) {$\bullet$};
  \node (2) at (-.87,.5) {$\circ$};
  \node (3) at (0,-1) {$\circ$};
  \path[-] (0) edge (1) (0) edge (2) (0) edge (3);
\end{tikzpicture}} \quad  \bigg \}, \\
\mathbf B_2(\phi) &= \Bigg \{ \quad 
\raisebox{-.3\height}{\begin{tikzpicture}[scale=.4]
  \tikzstyle{every node}=[inner sep=-1.5pt]
  \node (0) at (0,0) {$\circ$};
  \node (1) at (.87,.5) {$\bullet$};
  \node (2) at (-.87,.5) {$\circ$};
  \node (3) at (0,-1) {$\circ$};
  \node (4) at (1.73,0) {$\bullet$};
  \node (5) at (.87,1.5) {$\circ$};
  \node (6) at (-.87,1.5) {$\bullet$};
  \node (7) at (-1.73,0) {$\circ$};
  \node (8) at (-.87,-1.5) {$\bullet$};
  \node (9) at (.87,-1.5) {$\circ$};
  \path[-] (0) edge (1) (0) edge (2) (0) edge (3) (1) edge (4) (1) edge (5) (2) edge (6) (2) edge (7) (3) edge (8) (3) edge (9);
\end{tikzpicture},} \quad
\raisebox{-.3\height}{\begin{tikzpicture}[scale=.4]
  \tikzstyle{every node}=[inner sep=-1.5pt]
  \node (0) at (0,0) {$\circ$};
  \node (1) at (.87,.5) {$\bullet$};
  \node (2) at (-.87,.5) {$\circ$};
  \node (3) at (0,-1) {$\bullet$};
  \node (4) at (1.73,0) {$\bullet$};
  \node (5) at (.87,1.5) {$\circ$};
  \node (6) at (-.87,1.5) {$\bullet$};
  \node (7) at (-1.73,0) {$\bullet$};
  \node (8) at (-.87,-1.5) {$\bullet$};
  \node (9) at (.87,-1.5) {$\circ$};
  \path[-] (0) edge (1) (0) edge (2) (0) edge (3) (1) edge (4) (1) edge (5) (2) edge (6) (2) edge (7) (3) edge (8) (3) edge (9);
\end{tikzpicture},} \quad
\raisebox{-.3\height}{\begin{tikzpicture}[scale=.4]
  \tikzstyle{every node}=[inner sep=-1.5pt]
  \node (0) at (0,0) {$\bullet$};
  \node (1) at (.87,.5) {$\bullet$};
  \node (2) at (-.87,.5) {$\circ$};
  \node (3) at (0,-1) {$\circ$};
  \node (4) at (1.73,0) {$\circ$};
  \node (5) at (.87,1.5) {$\circ$};
  \node (6) at (-.87,1.5) {$\circ$};
  \node (7) at (-1.73,0) {$\circ$};
  \node (8) at (-.87,-1.5) {$\bullet$};
  \node (9) at (.87,-1.5) {$\circ$};
  \path[-] (0) edge (1) (0) edge (2) (0) edge (3) (1) edge (4) (1) edge (5) (2) edge (6) (2) edge (7) (3) edge (8) (3) edge (9);
\end{tikzpicture},} \quad
\raisebox{-.3\height}{\begin{tikzpicture}[scale=.4]
  \tikzstyle{every node}=[inner sep=-1.5pt]
  \node (0) at (0,0) {$\bullet$};
  \node (1) at (.87,.5) {$\bullet$};
  \node (2) at (-.87,.5) {$\circ$};
  \node (3) at (0,-1) {$\circ$};
  \node (4) at (1.73,0) {$\circ$};
  \node (5) at (.87,1.5) {$\circ$};
  \node (6) at (-.87,1.5) {$\bullet$};
  \node (7) at (-1.73,0) {$\circ$};
  \node (8) at (-.87,-1.5) {$\bullet$};
  \node (9) at (.87,-1.5) {$\circ$};
  \path[-] (0) edge (1) (0) edge (2) (0) edge (3) (1) edge (4) (1) edge (5) (2) edge (6) (2) edge (7) (3) edge (8) (3) edge (9);
\end{tikzpicture}} \quad  \Bigg \}. 
\end{align*}%\end{minipage}
\caption{An example of a Sturmian tree}
\label{sturmfig}
\end{figure}
A coloring $\phi$ is called \emph{Sturmian} if it has minimal unbounded factor complexity $b_n(\phi)$, which is $n+2$. 
See Figure~\ref{sturmfig} for an example of a Sturmian coloring and the sets $\mathbf{B}_n(\phi)$ for $n=0,1,2$.

A fundamental feature of Sturmian colorings is that for each $n$, there is a unique class of $n$-balls, called the \emph{special $n$-ball} and denoted by $S_n$, which is the restriction (to the concentric $n$-ball) of two distinct classes of $(n+1)$-balls denoted by $A_{n+1}, B_{n+1}$. We denote by $C_n$ the central $n$-ball of $S_{n+1}$. For $\phi$ in Figure 1, $S_0$ is the white vertex, $S_1$ is the last ball in $\mathbf B_1(\phi)$.
%\begin{definition}

Let $G$ be the group of $\phi$-preserving automorphisms of $\T$, which is a subgroup of the automorphism group $\mathrm{Aut}(\T)$. Define $X= X_\phi: = G \backslash \T$ to be the graph obtained from quotienting $\T$ by $G$. We call $X_\phi$ \emph{the quotient graph of} $(T,\phi)$.
%\end{definition}
%Note that $G$ is not necessarily a discrete subgroup of $\mathrm{Aut}(\T)$, but one can often find a subgroup $\Gamma \subset G$ discrete in $\mathrm{Aut}(\T)$ such that $X = \Gamma \backslash \T.$ 

The quotient graph $X=X_\phi$ has a structure of edge-indexed graph $(X,i)$, i.e. a graph with index $i : EX \to \mathbb{N}$ on the set of \emph{oriented} edges. It is in fact the edge-indexed graph associated to the graph of groups $(X, \mathrm{Stab}_{G}(\cdot))$ which is the graph $X$ with stabilizers $\mathrm{Stab}_{G}(x)$ attached to each $x \in VX \cup EX$. (See Chapter 2 of \cite{BL} for details.) 

The tree $\T$ is the universal cover of the edge-indexed graph $(X,i)$ (see Appendix of \cite{BL}). The graph $X$ is again colored by $\phi$ and the original coloring is the lift of $\phi$ from $X$ to $T$. In a previous work, we defined Stumian colorings of bounded type and characterized the quotient graph $X$ as follows.

%\begin{definition}(bounded type) 

For $v \in V\T$, we define the \emph{type set of $v$} as the set of $m \in \mathbb{Z}_{\geq 0}$ such that the colored $m$-ball centered at $v$ is special. 
A Sturmian coloring $\phi$ on tree $\T$ is called \emph{of bounded type} if 
the type set of each vertex is finite. 
%\end{definition} 
If one vertex is of bounded type, then all vertices are of bounded type (\cite{KL1}). 
Throughout the paper, by a \emph{cycle}, we mean a cycle of length $>1$. A cycle of length 1 will be called a \emph{loop}.
\begin{theorem}[\cite{KL1} Theorem 3.4] \label{thm:1.1}
Any Sturmian coloring $\phi$ is a lift of a coloring on a graph $X$ which is an infinite ray (with loops possibly attached)
\begin{center}
\begin{tikzpicture}[every loop/.style={}]
  \tikzstyle{every node}=[inner sep=-1pt]
  \node (5) at (-1,0) {$\bullet$};
  \node (6) at  (0,0) {$\bullet$};
  \node (7) at  (1,0) {$\bullet$};
  \node (8) at  (2,0) {$\bullet$};
  \node (9) at  (3,0) {$\bullet$};
  \node (10) at (4,0) {$\bullet$};
  \node (11) at (5,0) {$\cdots$};

\tikzstyle{every loop}=   [-, shorten >=.5pt]

  \path[-] 
	(5)  edge  (6)
	(6)  edge (7)
	(7)  edge (8)
	(8)  edge (9)
	(9)  edge  (10)
	(10) edge (11);
  \path[dotted] 
		 (5) edge [loop left] (5)
		 (6) edge [loop above] (6)
		 (7) edge [loop above] (7)
		 (8) edge [loop above] (8)
		 (9) edge [loop above] (9)
		 (10) edge [loop above] (10);  
\end{tikzpicture}
\end{center}
or a biinfinite line (with loops possibly attached)
\begin{center}
\begin{tikzpicture}[every loop/.style={}]
  \tikzstyle{every node}=[inner sep=-1pt]
  \node (0) at (-6,0) {$\cdots$};
  \node (1) at (-5,0) {$\bullet$} ;
  \node (2) at (-4,0) {$\bullet$} ;
  \node (3) at (-3,0) {$\bullet$};
  \node (4) at (-2,0) {$\bullet$};
  \node (5) at (-1,0) {$\bullet$};
  \node (6) at  (0,0) {$\bullet$};
  \node (7) at  (1,0) {$\bullet$};
  \node (8) at  (2,0) {$\bullet$};
  \node (9) at  (3,0) {$\bullet$};
  \node (10) at (4,0) {$\bullet$};
  \node (11) at (5,0) {$\cdots$};
  
  \tikzstyle{every loop}=   [-, shorten >=.5pt]

  \path[-] 
		 (0)  edge (1)
		 (1)  edge  (2)
		 (2)  edge (3)
		 (3)  edge  (4)
		 (4)  edge (5)
		 (5)  edge (6)
		 (6)  edge  (7)
		 (7)  edge  (8)
		 (8)  edge  (9)
		 (9)  edge  (10)
		 (10) edge (11) ;
  \path[dotted]
		 (1) edge [loop above] (1)
		 (2) edge [loop above] (2)
		 (3) edge [loop above] (3)
		 (4) edge [loop above] (4)
		 (5) edge [loop above] (5)
		 (6) edge [loop above] (6)
		 (7) edge [loop above] (7)
		 (8) edge [loop above] (8)
		 (9) edge [loop above] (9) 
		 (10) edge [loop above] (10) ; 
\end{tikzpicture}
\end{center}
Moreover, if the coloring $\phi$ is of bounded type, then the graph $X$ is the former. {If $x_i$ denotes the $i$-th vertex from the left, then there exists $m$ such that the maximum of the type set of $x_i$ is $m+i, \forall i \geq 0$.}
\end{theorem}

The edge-indexed graph $(X, i)$ of the first graph in the theorem above can be the edge-indexed graph associated to lattice of Nagao type (see Chapter 10 of \cite{BL} for such lattices).
%We further characterized the quotient graph for certain Sturmian colorings called Sturmian colorings of bounded type.

In this article, we show an induction algorithm for Sturmian colorings that characterizes Sturmian colorings completely, analogous to a continued fraction algorithm for Sturmian words. 
\subsection{Induction algorithm}

%It was shown in \cite{KL1} that the quotient graph $X$ of any Sturmian coloring is a line or a ray with loops possibly attached.
 %  for every vertex $x$ in $X$, there are at most two distinct vertices adjacent to $x$ apart from $x$ itself. 

\subsubsection{Sturmian words} 

Let us recall the correspondence between Sturmian words and irrational rotations. 
Fix $0 < \theta < 1$ and $0 \le \rho < 1.$  
%Let us identify $\mathbb{R}/\mathbb{Z}$ with $ [0, 1).$ Divide $[0,1)$ into two intervals $[0, 1-\theta), [1-\theta, 1)$ and label them as 0,1, respectively.  
%Let $s_n$ be the label of the interval to which $\rho + n \theta \pmod 1$ belongs. 
%One can check that
%$$s_n = \lfloor (n+1)  \theta + \rho \rfloor - \lfloor n \theta + \rho \rfloor .$$
%Then $\mathbf{s}=\cdots s_{-1} s_0 s_1 s_2 \cdots$ is a Sturmian word. \kdh{See e.g. \cite[Theorem 2.1.13]{Loth}.}
%\begin{figure}
%\includegraphics{Sturmian_tree_graph-2.mps}
%\caption{Strumian word by the irrational rotation}
%\label{Strum_rotation}
%\end{figure}
Consider the orbit $ \{ \rho + n \theta \pmod 1 \}_{n \in \mathbb{N}} $ of a rotation with irrational slope $\theta$ and intercept $\rho$. Partition $[0,1)$ into $[0, 1-\theta)\sqcup[1-\theta, 1)$ and give index $0,1$ to the two sets.
%We call $\theta$ the slope and $\rho$ the intercept of the irrational rotation. 
%Conversely, let $\mathbf{s}=s_1 s_2 \cdots$ be a Sturmian word.
%Then there exist $\theta$ and $\rho$ such that
%$$s_n = \lfloor (n+1)  \theta + \rho \rfloor - \lfloor n \theta + \rho \rfloor \text{ or } \lceil  (n+1)  \theta + \rho \rceil - \lceil n \theta + \rho \rceil.$$

An infinite word $\mathbf{s}=s_1 s_2 \cdots$ is a Sturmian word if and only if 
$$s_n = \lfloor (n+1)  \theta + \rho \rfloor - \lfloor n \theta + \rho \rfloor \text{ or } \lceil  (n+1)  \theta + \rho \rceil - \lceil n \theta + \rho \rceil,$$
for some $\theta$ and $\rho$. 
Note that in the first case, $s_n$ is the index of the set where the orbit belongs for the partition above and the second case corresponds to another similar partition.
See for example \cite[Section 2.1.2]{Loth}.
%The latter formula corresponds to the partition $(0,1]=(0, 1-\theta] \cup (1-\theta, 1]$ in place of $[0, 1-\theta)\cup [1-\theta, 1)$.
%Two Sturmian words have the same factor set if their slopes are same.

We choose two sequences of finite words $(u_n)$ and $(v_n)$ with alphabets $\{0,1\}$, satisfying $u_{-1}=0, v_{-1}=1$ and
\begin{equation}\label{eqn:induction} (u_{n+1},  v_{n+1}) = R(u_n,v_n) := (u_n,  u_n v_n) \quad
\mathrm{or}
\quad (u_{n+1}, v_{n+1})  = L(u_n,v_n) := ( v_n u_n,  v_n).\end{equation}
Rauzy showed that both sequences $u_n, v_n$ have the same limit which is a characteristic Sturmian word (i.e. Sturmian word with $\rho =0$), and conversely any characteristic Sturmian word is the limit of two such sequences \cite{R85}.

Two Sturmian words have the same factors if and only if they have the same slope \cite[Proposition 2.1.18]{Loth}.
For each $\theta$, the characteristic Sturmian word ${\mathbf s}_\theta$ is constructed by the partial quotients $a_k$ of the continued fraction expansion of the slope
$\theta = \frac{1}{a_1 + \dfrac{1}{a_2 + \ddots}}:$
%Let $s_{-1} = 1$, $s_0=0$, $s_n = s_{n-1}^{a_n}s_{n-2}$. 
%The characteristic Sturmian word of slope $\theta$ is 
$${\mathbf s}_\theta = \lim_{n \to \infty} u_n = \lim_{n \to \infty} v_n $$ 
where
\begin{equation}\label{ukvk}
\lim_{n \to \infty} (u_n , v_n) = \lim_{k \to \infty} R^{a_{2k+1}} \cdots R^{a_3} L^{a_2} R^{a_1 -1} (u_{-1}, v_{-1}).
\end{equation}
(see e.g. \cite[Proposition 2.2.24]{Loth}). 
 
There is an induction algorithm, closely related to continued fraction algorithm, for Sturmian words using Rauzy graphs (also called factor graphs) \cite{R82} (see also \cite{Ca}):
for a given infinite word $u$, the Rauzy graph $\mathfrak{G}_n$ is a finite oriented graph whose vertices are distinct $n$-words. There is an oriented edge from $u$ to $v$ if there are two letters $x,y$ such that $uy=xv$ is an $(n+1)$-word. 
For Sturmian words, the Rauzy graph $\mathfrak{G}_n$ is always a union of two cycles with an intersection which is either a vertex (case (i)) or a line segment (case (ii)). The $\mathfrak{G}_n$ belongs to the case (i) infinitely often. (See Arnoux and Rauzy's work \cite{AR} for transitions between case (i) and case (ii).)
Let $n_k$ be the sequence of positive integers such that $\mathfrak{G}_{n_k}$ belongs to case (i).
The two cycles in $\mathfrak{G}_{n_k}$ correspond to the finite words $(u_k, v_k)$.
Thus, the Rauzy graph $\mathfrak{G}_{n_k}$ evolves by the formula of \eqref{ukvk}. 

\begin{figure}
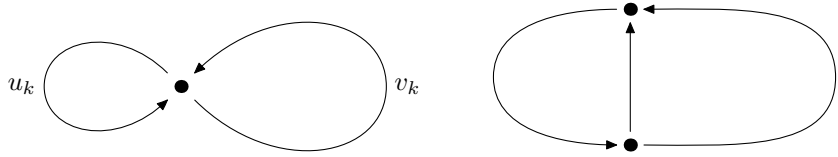

\includegraphics{Sturmian_tree_graph-4.mps}
\qquad 
\includegraphics{Sturmian_tree_graph-5.mps}
\caption{Rauzy graphs of Strumian word, case (i) and case (ii)}
\label{rauzygraph}
\end{figure}

\subsubsection{Sturmian colorings}
{Now let us describe our induction algorithm for Sturmian colorings. 
%The main idea to characterize Sturmian colorings is to introduce new colorings induced from the original colorings, by considering classes of $n$-balls (up to color-preserving isomorphisms of balls). 
For a given Sturmian coloring $\phi$, let $\G_n$ be the graph whose vertices are elements of $\mathbf{B}_n(\phi)$. There is an edge between two classes $D,E$ of $n$-balls if there exist $n$-balls of center $x,y$ in the class $D,E$, respectively, such that $d(x,y)=1$. The graph $\G_n$ is an analogue of Rauzy graphs $\mathfrak{G}_n$.}

Except at the vertex $S_n$, the graph $\G_n$ is locally isomorphic as colored graphs to the quotient graph $X_\phi$ since for any class $D \neq S_n$ of $n$-balls, there is a unique class of $(n+1)$-ball containing $D$ concentrically. For the special class $S_n$, there are two classes of $(n+1)$-balls $A_{n+1}, B_{n+1}$ containing $S_n$. 

%We say that two vertices $S_n,E$ of $\G_n$ are \emph{adjacent in $A$ ($B$, resp.)} if there are two $n$-balls of class $S_n,E$,  respectively, in $A_{n+1}$ ($B_{n+1}$, resp.) whose centers are of distance one.
We define graphs $\G^A_n, \G^B_n$ as follows: the vertices of $\G^A_n, \G^B_n$ are vertices of $\G_n$ that are connected to $S_n$ with edges defined by adjacency in $A_{n+1}, B_{n+1}$, respectively. The graphs $\G^A_n$, $\G^B_n$ both contain $C_n$ (Lemma~\ref{2.11KL}).

For $n$ such that $C_n$ is the end vertex of both $\G^A_n$ and $\G^B_n$, 
we define the concatenation $\G^A_n\bowtie \G^B_n$ of two graphs $\G^A_n, \G^B_n$. (See Definition~\ref{def:sum} for details.)

 If $C_n \neq S_n$, we define the vertex set to be $V(\G^A_n\bowtie \G^B_n)=V\G^A_n\cup V\G^B_n$ where the end vertices $C_n$ in $\G^A_n$ and $C_n$ in $\G^B_n$ are identified. The edge set is $E(\G^A_n\bowtie \G^B_n)=E\G^A_n\cup E\G^B_n, $ where the loops at $C_n$ in $\G^A_n$ and $\G^B_n$ are identified in $\G^A_n\bowtie \G^B_n.$
 
 If $C_n = S_n$, then define $V(\G^A_n\bowtie \G^B_n)=V\G^A_n\sqcup V\G^B_n$ where $C_n=S_n$ in $\G^A_n$ and $\G^B_n$ remain two distinct vertices in $\G^A_n\bowtie \G^B_n$.
 Define $E(\G^A_n\bowtie \G^B_n)=E\G^A_n\sqcup E \G^B_n \cup \{e\}$, where $e$ is the edge between two end vertices $C_n$ in $\G^A_n$ and $\G^B_n$.
%\end{enumerate}

Our first main theorem is the following induction algorithm, an analog of continued fraction algorithm described in \eqref{eqn:induction} and \eqref{ukvk}.

\begin{theorem}\label{thm:main} Let $\phi$ be a Sturmian coloring.
\begin{enumerate}
\item
If $\phi$ is such that $\G_n$ does not have any cycle for all $n$, then
there exists $\NN \in [0,\infty]$ and a sequence $(n_k)_{k \ge 0}$ such that 
$n_k =k$ for $0 \le k \le \NN$ and 
\begin{align*}
&\G^A_{n} \cong \G^A_{n-1}, \  \G^B_{n} \cong \G^A_{n-1} \bowtie \G^B_{n-1}, &\text{if} \ 0 \le  n < K, \\ 
\G^A_{n} \cong \G^A_{n-1} {\bowtie} \; \G^B_{n-1}, \  & \G^B_{n} \cong \G^A_{n-1}  \bowtie \G^B_{n-1} \ \text{ or } \ \G^A_{n} \cong \G^A_{n-1}   \bowtie \G^B_{n-1}, \  \G^B_{n} \cong  \G^B_{n-1}, &\text{if} \ n= K, \\ 
&\G^A_{n} \cong \G^A_{n-1}, \;\; \G^B_{n} \cong \G^B_{n-1},  &\text{if} \;\;n \neq n_k, n > K, \\ 
\G^A_{n} \cong \G^A_{n-1}  \bowtie \G^B_{n-1}, \; & \G^B_{n} \cong \G^B_{n-1} \ \text{ or } \ \G^A_{n} \cong \G^A_{n-1}, \  \G^B_{n} \cong \G^A_{n-1}  \bowtie  \G^B_{n-1},  &\text{if} \ n= n_k, \ n > \NN. 
\end{align*}
\item If $\phi$ is such that $\G_n$ has a cycle, for some $n$, then $\phi$ is of bounded type. The coloring $\phi$ is of bounded type if and only if either $\G^A_n$ or $\G^B_n$ eventually stabilizes. 
\end{enumerate}
\end{theorem}
Part (1) is proved in Theorem~\ref{thm:4} and part (2) is proved in Proposition~\ref{prop:circular} and Theorem~\ref{thm:bounded1}.
Note that the graphs $\G^A_n, \G^B_n$ has canonical structure of edge-indexed graphs (see Definition~\ref{def:gnab}). Theorem~\ref{thm:4} is more precise than part (1) of the above theorem in the sense that the concatenation we define in Section~\ref{section:wocycle} include the vertices $V,V'$ where the concatenation occurs and edge-indices $\mathbf{i}_k$ of edges around $V,V'.$

\subsection{Direct limit as the inverse process} In the second part of the article, we construct an inverse process of induction algorithm described in Theorem~\ref{thm:main}.

We first define the $\alpha_k$-admissible sequences of indices $\mathbf{i}_k$: $\alpha_k$ is an arbitrary sequence of $A$ and $B$ and $\mathbf{i}_k$ is defined so that the universal cover of the concatenated edge-indexed graph has degree $d$ (see Definition~\ref{def:4.1}). 
We then define $\alpha_k$-admissible $\beta_k$ in \eqref{eqn:4.2}, which is a sequence of $A,B$ closely related to $\alpha_k$. Two sequences $\beta_k, \beta'_k$ are \emph{equivalent} if they are eventually equal. Our main result in the second part of the article is that there is a one-to-one correspondence between Sturmian colorings and equivalence classes of admissible sequences:
$$ \{ (X, \phi) \text{ Sturmian colorings }  \} \ \overset{\Phi}{\underset{\Psi}{\rightleftarrows}}\  \{ (\alpha_k, \mathbf{i}_k, [\beta_k]): \mathbf{i}_k , \beta_k \text{ are } \alpha_k\text{-admissible } \},$$
where $[\beta_k]$ is the equivalence class of $\beta_k$.

\subsubsection{ $\Psi$ : a direct limit} For any given sequence $\alpha_k$ and $\alpha_k$-admissible sequence of indices $\mathbf{i}_k$, we define $\F^A_k, \F^B_k$ recursively using concatenations with edge-indices $\mathbf{i}_k$.
For an $\alpha_k$-admissible $\beta_k$, we first show that the direct limit of $\F^{\beta_k}_{k}$ is the quotient graph $(X,\phi)$ of a Sturmian coloring $\phi$ (Theorem~\ref{thm:6}), i.e. $\Psi$ is well-defined. 
%Note that $n_k$ are determined by $(\alpha_k,\mathbf{i}_k)$.
If $\phi$ is of bounded type, the sequence $\beta_k$ has a further restriction: $\beta_k$ is eventually constant (see Proposition~\ref{prop:circular}).

\subsubsection{$\Phi$ : the data of vertices and edge-indices in the induction algorithm.}\label{ak} For a given Sturmian coloring $\phi$, we define $\alpha_k, \mathbf{i}_k$ and $\beta_k=\beta_k(t)$ for each $t\in VX_\phi$: 
$(\alpha_k)$ is the sequence of letters defined by $\alpha_k =A$ if $|V\G_{n_k}^A| \ge |V\G_{n_k}^B|$ and $\alpha_k =B$ otherwise. The vector $\mathbf{i}_k$ is the new indices that appear in the concatenation: for example, $\mathbf{i}_k = i$ and $(i,j)$ respectively, in the two figures below (see Definition~\ref{def:4.8}).
%it is of the form $i$, $(i,j)$ or $(i, i', j)$ where $i,i',j$ are all integers between $0$ and $d$. 

\begin{center}
\begin{tikzpicture}[every loop/.style={}]
  \tikzstyle{every node}=[inner sep=-1pt]
  \node (10) at (8,0) {$\cdots$ \ };
  \node (11) at (9,0) {$\circledcirc$} ;
  \node (12) at (10,0) {$\bullet$} node [below=5pt] at (10,0) {$S_{n+1}$} node [below = 25 pt] at (10,0) {$\G^A_n \bowtie \G^B_n$ };
  \node (13) at (11,0) {$\circledast$}; 
  \node (14) at (12,0) {\ $\cdots$};
  \path[-] 
	(10) edge (11)
	(13) edge (14)
	(11) edge node [above=4pt] {\quad $i$} (12)
	(12) edge node [above=4pt] {\quad $m-i$ \quad} (13);
  \path[-] (12) edge [loop above] node [above=3pt,left=4pt] {$\ell$} (12);
\end{tikzpicture} 
\qquad 
\begin{tikzpicture}[every loop/.style={}]
  \tikzstyle{every node}=[inner sep=-1pt]
  \node (10) at (8,0) {$\cdots$\ };
  \node (11) at (9,0) {$\circledcirc$};
  \node (12) at (10,0) {$\bullet$} node [below=5pt] at (10,0) {$A_{n+1}$} node [below = 25 pt] at (10.5,0) {$\G^A_n \bowtie \G^B_n$ };

  \node (13) at (11,0) {$\bullet$} node [below=5pt] at (11,0) {$B_{n+1}$};
  \node (14) at (12,0) {$\circledast$};
  \node (15) at (13,0) {\ $\cdots$};
  \path[-] 
	(10) edge (11)
	(14) edge (15)
	(12) edge node [above=4pt] {$i \quad j$} (13)
	(11) edge node [above=4pt] {\quad $m_1$ \;\;} (12)
	(13) edge node [above=4pt] {$m_2$ \quad} (14);
  \path[-] (12) edge [loop above] node [above=5pt] {$\ell_1-i$} (12);
  \path[-] (13) edge [loop above] node [above=5pt] {$\ell_2-j$} (13);
\end{tikzpicture} 
\end{center}

We show that $\Phi$ is well-defined i.e.  $\mathbf{i}_k$, $\beta_k$ are $\alpha_k$-admissible and distinct $t,t' \in VX_\phi$ gives eventually equal sequences $\beta_k(t), \beta_k(t')$.
The next theorem shows that $\Psi \circ \Phi = \mathrm{Id}.$
% and $\Phi \circ \Psi = \mathrm{Id}.$
\begin{theorem}\label{thm1.3}
Let $\phi$ be a Sturmian coloring. 
\begin{enumerate} 

\item For a sequence $\alpha_k \in \{A, B\}$, if $\mathbf{i}_k$ and $\beta_k$ are $\alpha_k$-admissible, then the direct limit of $\F^{\beta_k}_k$ is a Sturmian coloring. 
Furthermore, the sequence induced from the direct limit $\F^{\beta_k}_k$  is $(\alpha_k, \mathbf{i}_k, \beta'_k)$ such that $\beta_k$ and $\beta'_k$ are eventually equal.
\item If $\phi$ is such that $\G_n$ does not have any cycle, for all $n$,
then there are $(\alpha_k)$, $(\alpha_k)$-admissible indices $(\mathbf{i}_k)$ and an $\alpha_k$-admissible sequence $\beta_k$ such that the direct limit $\underset{k \to \infty}{\lim} \G^{\beta_k}_{n_k}$ of $\G^{\beta_k}_{n_k}$ is the original Sturmian coloring $(X, \phi)$.

\item If $\phi$ is such that $\G_n$ has a cycle, then it is a coloring of bounded type.
If $\phi$ is of bounded type, then $(X, \phi)=\varinjlim \G^*_n$ 
where $*=A$ or $B$.
\end{enumerate}
\end{theorem}

Part (1) is proved in Theorem~\ref{thm:6} and part (2) is proved in Theorem~\ref{thm:6 }. Part (3) is proved in Theorem~\ref{thm:bounded1} (2).
The colorings have the same classes of balls if and only if the sequences $(\alpha_k, \mathbf{i}_k)$ are the same.
In this regard, the sequences $(\alpha_k, \mathbf{i}_k)$ correspond to partial quotients of the slope of a Sturmian word and the sequence $\beta_k$ corresponds to the intercept of a Sturmian word. 
 
% \begin{remark} If we consider two-sided infinite words, there are words of factor complexity $(n+1)$ but which are usually excluded from Sturmian words : they are eventually periodic i.e. 
%$(s_n)_{n \geq N}, (s_n)_{n \leq N}$ are periodic for some $N \in \mathbb{N}$,  e.g. $\cdots00000111111\cdots, \cdots00000100000\cdots.$
%If $\phi$ is of bounded type, then by part (2) of the above theorem, $(X, \phi)$ is the limit of $\G^A_n$'s or $\G^B_n$'s, say $\G^A_n$'s. Theorem~\ref{thm:bounded2} shows that the Sturmian colorings of bounded type behave similarly to eventually periodic bi-infinite words. 
%We further show that $\G^B_n$ also plays a role in the sense that the coloring is a countable union of periodic colorings, where the periodic coloring is determined by $\G^B_n$. See Theorem \ref{thm:bounded2} for details.
%\end{remark}

The article is organized as follows. 
In Section~\ref{section:lemmas}, we gather preliminary facts about colored balls, define the graph $\G_n$ and the edge-indexed graphs $\G^A_n$, $\G^B_n$. 
In Section~\ref{section:wocycle}, we study colorings for which $\G_n$ does not have a cycle for all $n$ and show induction algorithm (Theorem~\ref{thm:main}). 
In Section~\ref{section:4}, we define $\alpha_k$-admissible sequences $\mathbf{i}_k$ and $\beta_k$ and show Theorem~\ref{thm1.3} (1), (2).
In Section~\ref{sec:5}, cyclic Sturmian colorings are treated and we investigate Sturmian colorings of bounded type further to show Theorem~\ref{thm1.3} (3). 

\section{Graphs of colored balls}\label{section:lemmas}

Let $\T$ be a $d$-regular tree, i.e. the degree of each vertex is $d$.
Let $V\T, E\T$ be the set of vertices and the set of oriented edges of $\T$, respectively. The group $Aut(T)$ of automorphisms of $\T$ is a locally compact topological group with compact-open topology.
Consider the path metric $d$ on $\T$ with edge length all equal to $1$. 
The (closed) \emph{$n$-ball around $x$} is defined by $\N_n(x) = \{ y \in \T : d(x,y) \leq n \}$.

Throughout the paper, $\phi :V\T \to \A$ is a Sturmian coloring, i.e., a coloring of factor complexity $b_n(\phi) = n+2$.  
Since $b(0)=2$, $ \mathcal A$ has two elements.
Set $\mathcal A = \{ a, b\}$.  
%Also, the quotient graph $X$ defined in the introduction will be called the quotient graph associated to $(T, \phi)$.
\subsection{Preliminary : basic properties of Sturmian colorings} In this subsection, we recall preliminary facts from \cite{KL1} and prove  basic properties of Sturmian colorings, mostly about various adjacencies of $n$-balls in $(T,\phi)$.  
\begin{definition}\label{def:subball} For a Sturmian coloring $\phi $ on $T$, denote the colored tree by  $\T_\phi$. 
\begin{enumerate}
\item
Two vertices $x$ and $y$ are called \textit{congruent} if there exists a color-preserving automorphism of $\T_\phi$ sending $x$ to $y$.

\item Two $n$-balls $\N_n(x), \N_n(y)$ are called \textit{equivalent} if there exists a color-preserving isomorphism of $n$-balls between them. Such an equivalence class is called \textit{a class of $n$-balls} and is denoted it either by an $n$-ball with brackets, e.g. $[\N_n(x)]$, or by capital alphabets, e.g. $A_n, B_n, C_n, D, E$.
\item For two classes $D, E$ of balls, $D$ is \emph{(always) adjacent to} $E$ if for any $n$-balls $\N_n(x)$ in the class $D$, there exists an $n$-ball $\N_m(y)$ in the class $E$ such that $d(x,y)=1$. Two classes $D,E$ are \emph{(always) adjacent} if $D$ is adjacent to $E$ and vice versa.
\item Two classes $D,E$ of balls \emph{can be adjacent} if $D=[\N_n(x)], E=[\N_m(y)]$ for \emph{some} balls $\N_n(x), \N_m(y)$ in $T_\phi$ such that $d(x,y)=1$. 
\item A class of $n$-balls is called \textit{admissible} if it appears in $\T_\phi$.
Let $\mathbf{B}_n(\phi)$ be the set of admissible classes of $n$-balls. 
As defined in the introduction, $b_n(\phi) = | \mathbf{B}_n(\phi)|$. \end{enumerate} 
\end{definition}
 We will omit  the word ``always" in part (3) when there is no confusion. 
\begin{remark}Remark that if two classes $C, D$ of $n$-balls are both not special, then various kinds of adjacencies in part (3) and (4) above are equivalent. The only subtle situation is when one of them is $S_n$: for a class $D$, it is possible that $S_n, D$ can be adjacent, $D$ is always adjacent to $S_n$ and $S_n$ is not always adjacent to $D$.\end{remark}

Recall that $S_n$ is the unique ball contained in two distinct classes of $(n+1)$-balls $A_{n+1}, B_{n+1}$, and $C_n$ is the central $n$-ball of $S_{n+1}$.
%\begin{definition} Let $\phi$ be a Sturmian coloring, or more generally a coloring of subword complexity $b(n)= n+k$ for some $k \in \mathbb{N}$.
%\begin{enumerate}
%\item Since $b(n+1)-b(n) =1$, for each $n \geq 0 $, there is a unique class of $n$-balls, denoted by $S_n$, which is contained concentrically 
%in two distinct classes of $(n+1)$-balls, denoted by $A_{n+1}, B_{n+1}$. 
%We call $S_n$ the \emph{special (colored) $n$-ball} for $n \geq 0$. 
%\item The central class of $n$-balls of the special $(n+1)$-ball $S_{n+1}$ will be denoted by $C_n$. 
%\end{enumerate}
%\end{definition}
For a class of $n$-balls $B=[B_n(x)]$, denote the class of $[B_{n-1}(x)]$ by $\underline{B}$ and call it \emph{the restriction of} $B$.
One of the most basic properties of a Sturmian coloring is that for any non-special class of $n$-balls $B$, there is a unique class of $(n+1)$-balls containing $B$ concentrically, which we denote by $\overline{B}$ and call \emph{the extension of} $B$.
For the notational simplicity, we denote the empty ball by $S_{-1} = A_{-1} = C_{-1}$. Note that $B_{-1}$ is not defined. 

%We will call a class of $n$-balls an \emph{$n$-ball} when there is no ambiguity, in particular for $A_n, B_n, C_n$ and $S_n$.

\begin{lemma}\label{2.11KL}
For a Sturmian coloring $\phi$, without loss of generality, we assume that $S_0=A_0=[a]$ and $B_0=[b]$.
%If $m \in \Lambda_x$ for some $m \ge 1$, $x \in V\T$, then 
%there exists $y \in V\T$ such that $d(x,y) = 1$ and $m-1 \in \Lambda_y$. 
\begin{enumerate}
\item 
We can choose $\{A_n\},  \{B_n\}$ so that $A_{n+1}, B_{n+1}$ are always adjacent to $A_n, B_n$, respectively. 
%\seon{are adjacent should be "can be adjacent"?}
Moreover, $A_{n+1}, B_{n+1}$ are uniquely determined if we impose the condition that $A_{n+1}$ contains balls of class $A_n$ more than $B_{n+1}$ does.
%Let $[\N_{m+1}(z)] = [\N_{m+1}(z')]$ and $[\N_{m+2}(z)] \ne [\N_{m+2}(z')]$.
%Then, there exist $w$, $w'$ such that 
%$[\N_{m}(w)] = [\N_{m}(w')]$, $[\N_{m+1}(w)] \ne [\N_{m+1}(w')]$ and $d(w,z) = d(w',z') = 1$.
%\item 
%Each special class of $n$-balls $S_{n}$ is adjacent to an $n$-ball $C_n$. %Moreover, if $C_n \neq S_n$, then $C_n$ extends to $S_{n+1}$ only.       
\item
For each $n$-ball $\N_n(x)$, %the vertices in any 1-ball are in at most three distinct type sets. 
 the $n$-balls adjacent to $\N_n(x)$ belong to at most two classes of $n$-balls apart from $[\N_n(x)]$. 
Thus for any class of $n$-balls $D \neq S_n$, there are at most two classes of $n$-balls adjacent to $D$.
\item
If $A_{n} \ne S_{n}$ ($B_{n} \ne S_{n}$), then $A_{n}$ ($B_{n}$, respectively) is always adjacent to $S_{n}$.

\item The classes $S_n, C_n$ are always adjacent.

\end{enumerate}
\end{lemma}

\begin{proof}
\begin{enumerate}
\item Lemma 2.11 of \cite{KL1} for $m=1$ says that %\seon{$A_{n+1}, A_n$ and $B_{n+1}, B_n$ can be adjacent.} 
there are two balls $C,D$ such that
$A_{n+1}, B_{n+1}$ are always adjacent to $C,D$, respectively and $\{ C, D \}=\{A_n, B_n\}$. 
Thus we may choose $A_n, B_n$ so that $A_{n+1}, B_{n+1}$ are always adjacent to $A_n, B_n$, respectively.
The numbers of adjacent balls of class $A_n$ and $B_n$ to each $A_{n+1}$ or $B_{n+1}$  are constants since the balls of $A_n$, $B_n$ are contained in $A_{n+1}$ or $B_{n+1}$.
%there exist $x,y,x',y'$ such that $[\N_{n+1}(x)]=A_{n+1}, [\N_{n+1}(y)]=B_{n+1}$, $[\N_{n}(x')]=A_n, [\N_{n}(y')]=B_n$ and $d(x,x')=1, d(y,y')=1$.  
To show the uniqueness, denote the number of $A_n, B_n$ adjacent to $A_{n+1}$, $B_{n+1}$ by $i(A_{n+1}, A_n), i(A_{n+1}, B_n), i(B_{n+1}, A_n), i(B_{n+1}, B_n),$ respectively. Remark that
$$ i(A_{n+1}, A_{n}) + i(A_{n+1}, B_{n}) = i(B_{n+1}, A_{n})+ i(B_{n+1}, A_{n}),$$
since they are both equal to the number of $S_{n-1}$ adjacent to $S_n$. If $i_A(A_n, A_{n+1})=i_B(A_n, A_{n+1})$, then $A_{n+1}=B_{n+1}$ since all the $n$-balls distinct from $S_{n-1}$ adjacent to $S_n$ have unique extensions to $(n+1)$-balls.

\item It is clear that the centers of distinct classes of $n$-balls are not congruent, and there are at most two congruence classes of vertices adjacent to any given vertex by Theorem~\ref{thm:1.1}. Any $D \neq S_n$ has the same set of classes of colored balls in its 1-neighborhood, thus there are at most two classes that can be adjacent to $D$.
\item By part (1), % $A_n, A_{n+1}$ can be adjacent, thus
$A_n$ can be adjacent to $\underline{A_{n+1}}=S_n$. Thus if $A_n \neq S_n$, $A_n$ is always adjacent to $S_n$.
%Let $x,z \in V\T$ be the centers of balls colored by $A_n, A_{n+1}$, respectively.
%By part (1), there exists $w \in V\T$ adjacent to $z$ which is the center of $A_n$.
%Since $A_n$, being non-special, has a unique extension, we have $[\N_{n+1}(w)] =[\N_{n+1}(x)] = \overline{A_n}.$ Therefore a color-preserving isomorphism $f$ from $\N_n(w)$ to $\N_n(x)$ extends to $\N_{n+1}(w)$ to $\N_{n+1}(x)$. 
%Let $y = f(z)$, then $d(x,y)= d(w, z)=1$ and $[\N_n(y)]=[\N_n(z)]=S_n$. 
%Finally, (4) for $n$-balls is a direct consequence of (3) for $(n+1)$-balls.
\item 
If $S_n \ne C_n$, then $A_{n+1} \neq S_{n+1}$ and $B_{n+1} \neq S_{n+1}$, thus by part (3), both $A_{n+1}$ and $B_{n+1}$ are always adjacent to $S_{n+1}$, thus $S_n$ is always adjacent to $C_n$. 

If $S_n = C_n$, we may assume that $A_{n+1} = S_{n+1} \neq B_{n+1}$.
By part (1), $A_{n+2}$ and $B_{n+2}$ are always adjacent to $A_{n+1}$ and $B_{n+1}$ respectively, thus $S_{n+1} = A_{n+1}$ is always adjacent to either $A_{n+1}$ or $B_{n+1}$.
On the other hand, by part (3) $B_{n+1}$ is always adjacent to $S_{n+1}=A_{n+1}$.
Therefore, $S_n$ is always adjacent to $\underline{A_{n+1}} =\underline{B_{n+1}}=S_n = C_n$.

%there exists a neighboring vertex $y \in V\T$ of $x$ which is a center of $A_{n+1}$ or  $B_{n+1}$. 
%Since $A_{n+1} = S_{n+1}$ implies $S_n = C_n$, $y$ is a center of $C_n$.} 
\end{enumerate}
\end{proof}

\begin{lemma}\label{newlem} Let $\phi$ be a Sturmian coloring.
\begin{enumerate}
\item
Suppose that $D$ is not equal to any of $A_n$, $B_n$, $S_n$. If $S_n, D$ can be adjacent, then $S_n, D$ are always adjacent. 
\item If $D$ in part (1) satisfies $D \neq C_n$, then $S_n \ne C_n$.

\item Apart from $S_n$ itself, there are at most three classes of $n$-balls which can be adjacent to $S_n$.
\end{enumerate}
\end{lemma}

\begin{proof}
(1) We only need to show that $S_n$ is always adjacent to $D$. 
Since $S_n$ can be adjacent to $D$, it follows that $S_n$ contains $\underline{D}$, 
i.e. $S_n$ is always adjacent to $\underline{D}$. Note that $\underline{D} \neq S_{n-1}$ since $D \neq A_n, B_n$.
Therefore $\underline{D}$ is uniquely extended to $D$ and $S_n$ is always adjacent to $D$.
%Let $x,y \in V\T$ be adjacent vertices which are centers of $n$-balls colored by $S_n$, $D$, respectively.
%Let $x'$ be the center of another $n$-ball colored by $S_n$
%and $f$ be the color preserving isomorphism from $n$-ball of $x$ to $n$-ball of $x'$. 
%Then $y'=f(y)$ is the center of $\underline D$ with $d(x',y') =1$, thus the centre of $D$.
%Therefore, every ball colored by $S_n$ is adjacent to a ball colored by $D$.

(2) Suppose $D \neq C_n$. If $S_n = C_n$, then either $A_{n+1} = S_{n+1}$ or $B_{n+1} = S_{n+1}$.
By Lemma~\ref{2.11KL} (3), it follows that $B_{n+1} (\ne S_{n+1})$ is adjacent to $S_{n+1} = A_{n+1}$ or
$A_{n+1} (\ne S_{n+1})$ is adjacent to $S_{n+1} = B_{n+1}$.
Therefore, $A_{n+1}$, $B_{n+1}$ can be adjacent.
Since $A_{n+1}$, $B_{n+1}$ are always adjacent to $D$ by part (1) and to $A_n$, $B_n$, respectively, by Lemma~\ref{2.11KL} (1),
it follows that $A_{n+1}$ and $B_{n+1}$ can be adjacent to $B_{n+1}$, $\overline{A_n}$, $\overline{D}$
and to $A_{n+1}$, $\overline{B_n}$, $\overline{D}$ respectively.
Since either $A_n \ne S_n$ or $B_n \ne S_n$, either $B_{n+1}$, $\overline{A_n}$, $\overline{D}$ or $A_{n+1}$, $\overline{B_n}$, $\overline{D}$ are distinct, which contradicts Lemma~\ref{2.11KL} (2).

(3) By Lemma~\ref{2.11KL} (2), apart from $S_n$, there are at most four classes of $n$-balls adjacent to $S_n$. 
If there exists $D$ distinct from $A_n, B_n, S_n$ and adjacent to $S_n$, by part (1), $D, S_n$ are always adjacent, thus two of the four classes are $D$.  
\end{proof}

\subsection{Graph $\G_n$ of colored balls} In this subsection, we establish basic properties of the graph $\G_n$.
Recall that $V\G_n = \B_\phi(n)$ and two classes of $n$-balls are adjacent in $\G_n$, i.e. they are connected by an edge in $E\G_n$ if they can be adjacent. Recall also that a loop is an edge whose initial and terminal vertices are the same, and a cycle is of length larger than 1.

%As before, we let $S_n$ to be the special ball, $A_n, B_n$ to be the two distinct extensions of $S_{n-1}$, and $C_n$ to be the central $n$-ball of $S_{n+1}$. 
%For any non-special $n$-ball $C$, we denote the unique extension of $C$ to $(n+1)$-ball by $\overline{C}$.

\begin{lemma}\label{lem:merged}
\begin{enumerate}
\item Suppose that $D$ is distinct from $A_n, B_n, C_n, S_n$. If $S_n,D$ are adjacent in $\G_n$, then $D$ is a vertex of degree 1 in $\G_n$ and there is a cycle in $\G_{n+1}$.

\item If $B_{n+1}$ (resp. $A_{n+1}$) is adjacent to $A_n$ (resp. $B_n$) in $\G_n$, and $A_n \neq S_n$, $C_n$, then there is a cycle in $\G_{n+1}$.
\end{enumerate}
\end{lemma}

\begin{proof}
(1) By Lemma~\ref{newlem} (2), we have $S_n \neq C_n$. Since $S_n \ne D$ as well, 
both $A_{n+1}$ and $B_{n+1}$ are distinct from $S_{n+1}$ and $\overline D$. 
Lemma~\ref{2.11KL} (3) implies that both $A_{n+1}$ and $B_{n+1}$ are adjacent to $S_{n+1}$.
By Lemma~\ref{newlem} (1), $S_n$ is always adjacent to $D$,
thus both $A_{n+1}$ and $B_{n+1}$ can be adjacent to $\overline D$. 
Therefore, the path with vertices $[S_{n+1}A_{n+1}\overline D B_{n+1} S_{n+1}]$ is a cycle in $\G_{n+1}$. The restriction of the path to $n$-balls is a line segment $[C_n S_n D]$ with $D$ a vertex of degree 1. 

(2) By assumption, $\overline{A_n}$ is adjacent to $B_{n+1}$ in $\G_n$. By Lemma~\ref{2.11KL} (1), $\overline{A_n}$ is adjacent to $A_{n+1}$ in $\G_n$.
%Since $A_n \neq C_n$, by Lemma~\ref{2.11KL} (1) and by assumption, the nonspecial $(n+1)$-ball $\overline{A_n}$ is adjacent to $A_{n+1}$ and $B_{n+1}$. 
Since $A_n \neq S_n$, $\overline{A_n}$ differs from both $A_{n+1}$ and $B_{n+1}$.
%Thus by Lemma~\ref{lem:KL}, there is no other $(n+1)$-ball  besides $A_{n+1}, B_{n+1}$ adjacent to $\overline{A_n}$. It implies that there is no $n$-ball besides $S_n$ in the neighbour of $D$.
%Therefore, the branch starting from $A_n$ is of length 1.
If $S_n \neq C_n$, then %$S_n$ is always adjacent to $C_n$ by Lemma~\ref{2.11KL} (2), the 1-neighborhoods of $S_n$ are one of $C_n , S_n, A_n$.
%Since $S_n \ne C_n$, we have 
$S_{n+1}$ differs from both $A_{n+1}$ and $B_{n+1}$.
Thus, by Lemma~\ref{2.11KL} (3), the path with vertices $[S_{n+1} A_{n+1} \overline{A_n}B_{n+1}S_{n+1}]$ is a cycle in $\G_{n+1}$.
If $S_n = C_n$, then $S_{n+1}$ is one of $A_{n+1}, B_{n+1}$, 
By Lemma~\ref{2.11KL} (3), the path $[S_{n+1}\overline{A_n} B_{n+1} S_{n+1}]$ or $[S_{n+1}\overline{A_n}A_{n+1}S_{n+1}]$ is a cycle in $\G_{n+1}$.
\end{proof}

Recall that $S_{-1}=C_{-1}=A_{-1}= \emptyset$ and $S_0=A_0$.
\begin{lemma}\label{lem:n0} 
There exists $0 \leq \NN \leq \infty$ such that $S_n = A_n = C_n$ if and only if $-1 \le n < \NN$. Therefore, either $A_K =S_K$ or $B_K=S_K.$ 

%(2) If $A_{n+1} = S_{n+1}$ and $B_{n+1} = C_{n+1}$, then $S_n = A_n = C_n$.
\end{lemma}

%
%(2) If $A_{n+1} = S_{n+1}$ and $B_{n+1} = C_{n+1}$, then $S_n = A_n = C_n$.
%(2) Clearly $S_n = C_n$ from $A_{n+1} = S_{n+1}$. 
%
\begin{proof}     
Suppose that $S_{n+1} = A_{n+1} = C_{n+1}$ and $n \ge 0$. 
Then by Lemma~\ref{newlem} (2), $S_{n+1}$ is adjacent to only $B_{n+1}$ and itself.
By Lemma~\ref{2.11KL} (2), $B_{n+1}$ has exactly one more adjacent vertex $D$ in $\G_{n+1}$ since $\G_{n+1}$ is connected and $|V\G_{n+1}| \ge 3$. 
Then $S_n = \underline{B_{n+1}}=\underline{A_{n+1}} = \underline{S_{n+1}} = C_n$ and $S_n$ is adjacent to only $\underline{D}$ and itself.
Since $S_n$ can be adjacent to $A_n$ and $B_n$ by Lemma~\ref{2.11KL} (1), 
we have $A_n = S_n=C_n$ and $B_n = \underline{D}$.
The case $B_n=S_n=C_n$ and $A_n = \underline{D}$ does not occur since otherwise $A_{n+1}$ would be adjacent to $\underline{D}$ which is a contradiction to the fact that $S_{n+1}$ is adjacent to only $B_{n+1}$ and itself.
If  $S_0 = A_0 \ne C_0$, then put $\NN =0$. 

Since $S_{K-1}=C_{K-1}$, either $A_K =S_K$ or $B_K=S_K.$ 
\end{proof}
\subsection{Edge-indexed Graphs $\G_n^A, \G_n^B$ of colored balls}\label{subsec:GAB} Although the vertices of $\G_n$ are equivalence classes of vertices of $X$, $\G_n$ do not resemble the graph $X$ even locally, because of the special ball $S_n$, which can be extended to two ways in $T_\phi$. In this section, we define two edge-indexed graphs 
$\G^A_n$, $\G^B_n$ so that $X$ locally looks like either $\G^A_n$ or $\G^B_n$.   

\begin{definition} Define the indices $i, i_A, i_B : \bigcup_n (V \G_n \times V \G_n) \to \mathbb{N}$ as follows.
\begin{enumerate}
\item[(1)] If an $n$-ball $X_n$ is not special, let 
$ i(X_n, Y_n) $ be the number of $n$-balls colored by $Y_n$ adjacent to $X_n$. It is independent of the position of $X_n$ in the tree.
\item[(2)] Define $i_A(S_n, Y_n)$, $i_B(S_n, Y_n)$ to be the number of $Y_n$ adjacent to $S_n$ in $A_{n+1}, B_{n+1}$, respectively. For simplicity, for $X_n \ne S_{n}$ denote $i_A (X_n,Y_n) = i_B (X_n,Y_n) = i(X_n,Y_n)$.  

\end{enumerate}
\end{definition}

In particular, $ i(X_n, Y_n)=0 $ if there is no edge between $X_n$ and $Y_n$. 
\begin{remark}
Remark that $i, i_A, i_B$ are not reflexive in general, however the positivity of $i$ is reflexive in the following sense: 
\begin{enumerate}
\item If $X,Y$ are not special, then  $i(X, Y)>0$ if and only if $i(Y, X)>0$. 
\item If $X \neq S_n$, then $i(X, S_n)>0$ if and only if $i_A(S_n, X)+i_B(S_n, X)>0$.
\end{enumerate}
\end{remark}
\begin{definition}\label{def:gnab}
Define $\G^A_n$ ($\G^B_n$) to be the edge-indexed oriented graph whose vertices are those which are connected by a path from $S_n$ with edges of positive index for $i_A$ ($i_B$, respectively). 
Their oriented edge set is the set of oriented edges between vertices in $\G^A_n$ ($\G^B_n$, respectively) with positive index for $i_A$ ($i_B$, respectively), endowed with the index $i_A$ ($i_B$, respectively). 
\end{definition}
For $n = -1$,  define $\G^A_{-1} = \G^B_{-1}$ to be the edge-indexed graph consisting of the vertex of empty ball $S_{-1} = C_{-1}$
 and a loop on it indexed by $d$, the degree of $\T$.

 It is clear that the set of vertices $V\G^A_n \cup V\G^B_n$ is the set of classes of $n$-balls $\mathbf B_\phi (n)$.

Note that the 1-neighborhood of a given vertex in $\G^A_n$ and $\G^B_n$ are identical except at $S_n$, since any non-special $n$-ball has a unique extension to $(n+1)$-ball. 

%\begin{lemma}
%There are graph injections from $\G^A_{n}$, $\G^B_{n}$ to $\G^A_{n+1}$, $\G^B_{n+1}$ respectively. 
%\end{lemma}
%
%\begin{proof}
%By Lemma~\ref{2.11KL},
%we have $\overline{U} \in V\G^A_{n+1}$ for $U \in V\G^A_n -\{ S_n\}$ and $A_{n+1} \in V\G^A_{n+1}$.
%Thus there is the cannonical graph injection from $\G^A_{n}$ to $\G^A_{n+1}$ given by $U \mapsto \overline U$.
%\end{proof}

The following lemma is immediate by definition. The index $i_B$ enjoys properties similar to those of $i_A$.
\begin{lemma}\label{lem:index}
Let $V \neq S_n$. Then for each $n \ge -1$
\begin{enumerate}
\item For an $n$-ball $U \neq S_n, C_n$ we have
$$ i(U,V) = i(\overline{U}, \overline{V}), \qquad i(U, S_n) = i(\overline{U}, A_{n+1})+i(\overline{U}, B_{n+1}).$$
\item If $C_n \ne S_n$, then 
\begin{align*} i(C_n, V) &= i_A (S_{n+1}, \overline{V}),  &i(C_n, S_n) &= i_A (S_{n+1}, A_{n+1}) + i_A (S_{n+1}, B_{n+1}), \\
i_A(S_n, V)&= i(A_{n+1}, \overline{V}), & i_A(S_n, S_n) &= i(A_{n+1}, A_{n+1}) +i(A_{n+1},B_{n+1}).\end{align*}
%Similar properties hold for $i_B$.

\item If $C_n = S_n$, say $S_{n+1} = A_{n+1}$, then %for $n \ge -1$
\begin{align*}
i_A(S_n, V) &= i_A(A_{n+1}, \overline{V}) &
i_A(S_n, S_n) &= i_A(A_{n+1}, A_{n+1}) +i_A(A_{n+1},B_{n+1}) \\
&= i_B(A_{n+1}, \overline{V}) &
&= i_B(A_{n+1}, A_{n+1}) +i_B(A_{n+1},B_{n+1}) \\
i_B(S_n, V) &= i(B_{n+1}, \overline{V}) &  
i_B(S_n, S_n) &= i(B_{n+1}, A_{n+1}) +i(B_{n+1},B_{n+1}).
\end{align*}
%If $B_{n+1} = S_{n+1}$ (thus, $C_n = S_n$), then we have
%\begin{align*}
%i_A(S_n, V) &= i(A_{n+1}, \overline{V}) &  
%i_A(S_n, S_n) &= i(A_{n+1},A_{n+1}) +i(A_{n+1},B_{n+1}) \\
%i_B(S_n, V) &= i_A(B_{n+1},\overline{V}) &
%i_B(S_n, S_n) &= i_A(B_{n+1}, A_{n+1}) +i_A(B_{n+1},B_{n+1}) \\
%&= i_B(B_{n+1}, \overline{V}) &
%&= i_B(B_{n+1}, A_{n+1}) +i_B(B_{n+1},B_{n+1}).
%\end{align*}
\end{enumerate}
\end{lemma}

\begin{definition}
We say that a Sturmian coloring is cyclic if there is a cycle in $\G_n$ for some $n$.
We say that a Sturmian coloring is acyclic if it is not cyclic.
\end{definition}

\begin{example}
For the coloring $\phi$ in Figure~\ref{sturmfig}, we have the sequence of $\G_n$ as follows.
$$
A_0=S_0 =\circ,  \quad B_0=C_0= \bullet; $$
$$
\begin{tikzpicture}[every loop/.style={}]
  \tikzstyle{every node}=[inner sep=-1pt]
  
  \node (0) at (0,0) {$\circ$} node [left=30pt] at (0,0) {$\G_0$ :} node [below=4pt] at (0,0) {$S_0$};  
  \node (1) at (1.5,0) {$\bullet$} node [below=4pt] at (1.5,0) {$B_0 $};

  \path[-] 
	(0)  edge (1);
  
  \node (10) at (0,1) {$\circ$} node [left=30pt] at (0,1) {$\G^A_0$ :} node [below=4pt] at (0,1) {$S_0$};    
  \node (11) at (1.5,1) {$\bullet$} node [below=4pt] at (1.5,1) {$B_0$};
 
  \node (16) at (0,-1) {$\circ$} node [left=30pt] at (0,-1) {$\G^B_0$ :} node [below=4pt] at (0,-1) {$S_0$};  
  \node (17) at (1.5,-1) {$\bullet$} node [below=4pt] at (1.5,-1) {$B_0$};
 
  \path[-] 
	(10)  edge node [above=4pt] {1 \qquad 2} (11)
	(16)  edge node [above=4pt] {2 \qquad 2} (17);
  \path[-]
	  	(10) edge [loop left] node [above=8pt,right=4pt] {2} (10)
	  	(11) edge [loop right] node [above=8pt, left=4pt] {1} (11)
	  	(16) edge [loop left] node [above=8pt,right=4pt] {1} (16)
	  	(17) edge [loop right] node [above=8pt, left=4pt] {1} (17);
\end{tikzpicture}  $$
$$
A_1 = 
\raisebox{-.2\height}{\begin{tikzpicture}[scale=.6]
  \tikzstyle{every node}=[inner sep=-1pt]
  \node (0) at (0,0) {$\circ$};
  \node (1) at (.87,.5) {$\bullet$};
  \node (2) at (-.87,.5) {$\circ$};
  \node (3) at (0,-1) {$\circ$};
  \path[-] (0) edge (1) (0) edge (2) (0) edge (3);
\end{tikzpicture},} \quad
B_1 = C_1 =  \raisebox{-.2\height}{\begin{tikzpicture}[scale=.6]
  \tikzstyle{every node}=[inner sep=-1pt]
  \node (0) at (0,0) {$\circ$};
  \node (1) at (.87,.5) {$\bullet$};
  \node (2) at (-.87,.5) {$\circ$};
  \node (3) at (0,-1) {$\bullet$};
  \path[-] (0) edge (1) (0) edge (2) (0) edge (3);
\end{tikzpicture},} 
\quad S_1 = \overline{B_0}=
\raisebox{-.2\height}{\begin{tikzpicture}[scale=.6]
  \tikzstyle{every node}=[inner sep=-1pt]
  \node (0) at (0,0) {$\bullet$};
  \node (1) at (.87,.5) {$\bullet$};
  \node (2) at (-.87,.5) {$\circ$};
  \node (3) at (0,-1) {$\circ$};
  \path[-] (0) edge (1) (0) edge (2) (0) edge (3);
\end{tikzpicture};}$$
$$
\begin{tikzpicture}[every loop/.style={}]
  \tikzstyle{every node}=[inner sep=-1pt]
  
  \node (0) at (0,0) {$\circ$} node [left=20pt] at (0,0) {$\G_1$ :} node [below=4pt] at (0,0) {$A_1$};  
  \node (1) at (1.5,0) {$\bullet$} node [below=4pt] at (1.5,0) {$S_1$};
  \node (2) at (3,0) {$\circ$} node [below=4pt] at (3,0) {$B_1$};

  \path[-] 
	(0)  edge (1)
	(1)  edge (2);
  
  \node (10) at (0,1.3) {$\circ$} node [left=20pt] at (0,1.3) {$\G^A_1$ :} node [below=4pt] at (0,1.3) {$A_1$};    
  \node (11) at (1.5,1.3) {$\bullet$} node [below=4pt] at (1.5,1.3) {$S_1$};
  \node (12) at (3,1.3) {$\circ$} node [below=4pt] at (3,1.3) {$B_1$};

  \node (16) at (1.5,-1.3) {$\bullet$} node [left=20pt] at (0,-1.3) {$\G^B_1$ :} node [below=4pt] at (1.5,-1.3) {$S_1$};  
  \node (17) at (3,-1.3) {$\circ$} node [below=4pt] at (3,-1.3) {$B_1$};

  \path[-] 
	(10)  edge node [above=4pt] {1 \qquad 1} (11)
	(11)  edge node [above=4pt] {1 \qquad 2} (12)
	(16)  edge node [above=4pt] {2 \qquad 2} (17);
  \path[-]
	  	(10) edge [loop left] node [above=8pt,right=4pt] {2} (10)
	  	(11) edge [loop above] (11)
	  	(12) edge [loop right] node [above=8pt, left=4pt] {1} (12)
	  	(16) edge [loop left] node [above=8pt,right=4pt] {1} (16)
	  	(17) edge [loop right] node [above=8pt, left=4pt] {1} (17);
\end{tikzpicture} 
$$

$$
\overline{A_1} = \raisebox{-.3\height}{\begin{tikzpicture}[scale=.4]
  \tikzstyle{every node}=[inner sep=-1pt]
  \node (0) at (0,0) {$\circ$};
  \node (1) at (.87,.5) {$\bullet$};
  \node (2) at (-.87,.5) {$\circ$};
  \node (3) at (0,-1) {$\circ$};
  \node (4) at (1.73,0) {$\bullet$};
  \node (5) at (.87,1.5) {$\circ$};
  \node (6) at (-.87,1.5) {$\bullet$};
  \node (7) at (-1.73,0) {$\circ$};
  \node (8) at (-.87,-1.5) {$\bullet$};
  \node (9) at (.87,-1.5) {$\circ$};
  \path[-] (0) edge (1) (0) edge (2) (0) edge (3) (1) edge (4) (1) edge (5) (2) edge (6) (2) edge (7) (3) edge (8) (3) edge (9);
\end{tikzpicture},} \quad
\overline{B_1} = S_2 = \raisebox{-.3\height}{\begin{tikzpicture}[scale=.4]
  \tikzstyle{every node}=[inner sep=-1pt]
  \node (0) at (0,0) {$\circ$};
  \node (1) at (.87,.5) {$\bullet$};
  \node (2) at (-.87,.5) {$\circ$};
  \node (3) at (0,-1) {$\bullet$};
  \node (4) at (1.73,0) {$\bullet$};
  \node (5) at (.87,1.5) {$\circ$};
  \node (6) at (-.87,1.5) {$\bullet$};
  \node (7) at (-1.73,0) {$\bullet$};
  \node (8) at (-.87,-1.5) {$\bullet$};
  \node (9) at (.87,-1.5) {$\circ$};
  \path[-] (0) edge (1) (0) edge (2) (0) edge (3) (1) edge (4) (1) edge (5) (2) edge (6) (2) edge (7) (3) edge (8) (3) edge (9);
\end{tikzpicture},} \quad
A_2 = \raisebox{-.3\height}{\begin{tikzpicture}[scale=.4]
  \tikzstyle{every node}=[inner sep=-1pt]
  \node (0) at (0,0) {$\bullet$};
  \node (1) at (.87,.5) {$\bullet$};
  \node (2) at (-.87,.5) {$\circ$};
  \node (3) at (0,-1) {$\circ$};
  \node (4) at (1.73,0) {$\circ$};
  \node (5) at (.87,1.5) {$\circ$};
  \node (6) at (-.87,1.5) {$\circ$};
  \node (7) at (-1.73,0) {$\circ$};
  \node (8) at (-.87,-1.5) {$\bullet$};
  \node (9) at (.87,-1.5) {$\circ$};
  \path[-] (0) edge (1) (0) edge (2) (0) edge (3) (1) edge (4) (1) edge (5) (2) edge (6) (2) edge (7) (3) edge (8) (3) edge (9);
\end{tikzpicture},} \quad
B_2 = \raisebox{-.3\height}{\begin{tikzpicture}[scale=.4]
  \tikzstyle{every node}=[inner sep=-1pt]
  \node (0) at (0,0) {$\bullet$};
  \node (1) at (.87,.5) {$\bullet$};
  \node (2) at (-.87,.5) {$\circ$};
  \node (3) at (0,-1) {$\circ$};
  \node (4) at (1.73,0) {$\circ$};
  \node (5) at (.87,1.5) {$\circ$};
  \node (6) at (-.87,1.5) {$\bullet$};
  \node (7) at (-1.73,0) {$\circ$};
  \node (8) at (-.87,-1.5) {$\bullet$};
  \node (9) at (.87,-1.5) {$\circ$};
  \path[-] (0) edge (1) (0) edge (2) (0) edge (3) (1) edge (4) (1) edge (5) (2) edge (6) (2) edge (7) (3) edge (8) (3) edge (9);
\end{tikzpicture}.} $$
$$
\begin{tikzpicture}[every loop/.style={}]
  \tikzstyle{every node}=[inner sep=-1pt]
  
  \node (0) at (0,.3) {$\circ$} node [left=20pt] at (0,0) {$\G_2$ :} node [below=4pt] at (0,.3) {$\overline{A_1}$};  
  \node (1) at (1.5,.3) {$\bullet$} node [below=4pt] at (1.5,.3) {$A_2$};
  \node (2) at (1.5,-.3) {$\bullet$} node [below=4pt] at (1.5,-.3) {$B_2$};
  \node (3) at (3,0) {$\circ$} node [below=4pt] at (3,0) {$S_2$};

  \path[-] 
	(0)  edge (1)
	(1)  edge (3)
	(2)  edge (3);
  
  \node (10) at (0,1.3) {$\circ$} node [left=20pt] at (0,1.3) {$\G^A_2$ :} node [below=4pt] at (0,1.3) {$\overline{A_1}$};    
  \node (11) at (1.5,1.3) {$\bullet$} node [below=4pt] at (1.5,1.3) {$A_2$};
  \node (12) at (3,1.3) {$\circ$} node [below=4pt] at (3,1.3) {$S_2$};

  \node (16) at (1.5,-1.3) {$\bullet$} node [left=20pt] at (0,-1.3) {$\G^B_2$ :} node [below=4pt] at (1.5,-1.3) {$B_2$};  
  \node (17) at (3,-1.3) {$\circ$} node [below=4pt] at (3,-1.3) {$S_2$};

  \path[-] 
	(10)  edge node [above=4pt] {1 \qquad 1} (11)
	(11)  edge node [above=4pt] {1 \qquad 2} (12)
	(16)  edge node [above=4pt] {2 \qquad 2} (17);
  \path[-]
	  	(10) edge [loop left] node [above=8pt,right=4pt] {2} (10)
	  	(11) edge [loop above] (11)
	  	(12) edge [loop right] node [above=8pt, left=4pt] {1} (12)
	  	(16) edge [loop left] node [above=8pt,right=4pt] {1} (16)
	  	(17) edge [loop right] node [above=8pt, left=4pt] {1} (17);
\end{tikzpicture} 
$$

Note that $S_2 = C_2$ and $A_3 = S_3$. Let us omit the balls from $n=4$.

\bigskip
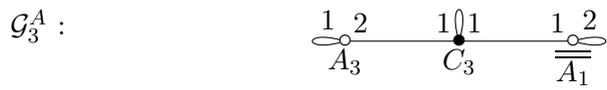
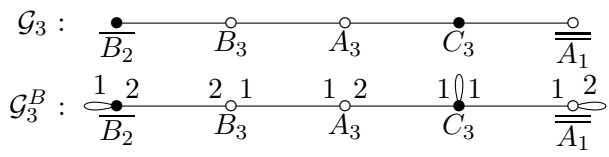
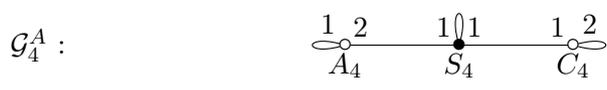
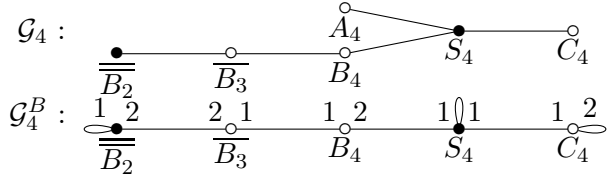
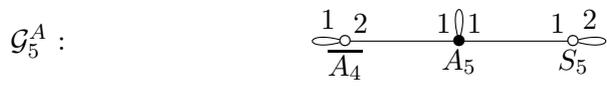
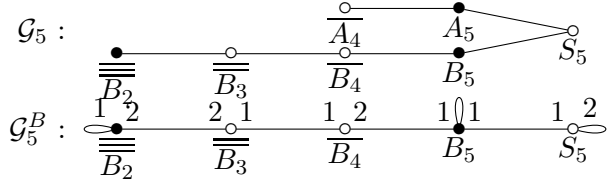
\begin{center}
\begin{tikzpicture}[every loop/.style={}]
  \tikzstyle{every node}=[inner sep=-1pt]
  
  \node (0) at (0,0) {$\bullet$} node [left=20pt] at (0,0) {$\G_3$ :} node [below=4pt] at (0,0) {$\overline{B_2}$};  
  \node (1) at (1.5,0) {$\circ$} node [below=4pt] at (1.5,0) {$B_3$};
  \node (2) at (3,0) {$\circ$} node [below=4pt] at (3,0) {$A_3$};
  \node (3) at (4.5,0) {$\bullet$} node [below=4pt] at (4.5,0) {$C_3$};
  \node (4) at (6,0) {$\circ$} node [below=4pt] at (6,0) {$\overline{\overline{A_1}}$};

  \path[-] 
	(0)  edge (1)
	(1)  edge (2)
	(2)  edge (3)
	(3)  edge (4);
	  
  \node (10) at (3,1.1) {$\circ$} node [left=20pt] at (0,1.3) {$\G^A_3$ :} node [below=4pt] at (3,1.1) {$A_3$};    
  \node (11) at (4.5,1.1) {$\bullet$} node [below=4pt] at (4.5,1.1) {$C_3$};
  \node (12) at (6,1.1) {$\circ$} node [below=4pt] at (6,1.1) {$\overline{\overline{A_1}}$};

  \node (16) at (0,-1.1) {$\bullet$} node [left=20pt] at (0,-1.1) {$\G^B_3$ :} node [below=4pt] at (0,-1.1) {$\overline{B_2}$};  
  \node (17) at (1.5,-1.1) {$\circ$} node [below=4pt] at (1.5,-1.1) {$B_3$};
  \node (18) at (3,-1.1) {$\circ$} node [below=4pt] at (3,-1.1) {$A_3$};
  \node (19) at (4.5,-1.1) {$\bullet$} node [below=4pt] at (4.5,-1.1) {$C_3$};
  \node (20) at (6,-1.1) {$\circ$} node [below=4pt] at (6,-1.1) {$\overline{\overline{A_1}}$};

  \path[-] 
	(10)  edge node [above=4pt] {2 \qquad 1} (11)
	(11)  edge node [above=4pt] {1 \qquad 1} (12)
	(16)  edge node [above=4pt] {2 \qquad 2} (17)
	(17)  edge node [above=4pt] {1 \qquad 1} (18)
	(18)  edge node [above=4pt] {2 \qquad 1} (19)
	(19)  edge node [above=4pt] {1 \qquad 1} (20);
  \path[-]
	  	(10) edge [loop left] node [above=8pt,right=4pt] {1} (10)
	  	(11) edge [loop above] (11)
	  	(12) edge [loop right] node [above=8pt, left=4pt] {2} (12)
	  	(16) edge [loop left] node [above=8pt,right=4pt] {1} (16)
	  	(19) edge [loop above] (19)
	  	(20) edge [loop right] node [above=8pt, left=4pt] {2} (20);
\end{tikzpicture}\\

\bigskip

\begin{tikzpicture}[every loop/.style={}]
  \tikzstyle{every node}=[inner sep=-1pt]
  
  \node (0) at (0,-.3) {$\bullet$} node [left=20pt] at (0,0) {$\G_4$ :} node [below=4pt] at (0,-.3) {$\overline{\overline{B_2}}$};  
  \node (1) at (1.5,-.3) {$\circ$} node [below=4pt] at (1.5,-.3) {$\overline{B_3}$};
  \node (2) at (3,-.3) {$\circ$} node [below=4pt] at (3,-.3) {$B_4$};
  \node (3) at (4.5,0) {$\bullet$} node [below=4pt] at (4.5,0) {$S_4$};
  \node (4) at (6,0) {$\circ$} node [below=4pt] at (6,0) {$C_4$};
  \node (5) at (3,.3) {$\circ$} node [below=4pt] at (3,.3) {$A_4$};

  \path[-] 
	(0)  edge (1)
	(1)  edge (2)
	(2)  edge (3)
	(5)  edge (3)
	(3)  edge (4);
	  
  \node (10) at (3,1.1) {$\circ$} node [left=20pt] at (0,1.1) {$\G^A_4$ :} node [below=4pt] at (3,1.1) {$A_4$};    
  \node (11) at (4.5,1.1) {$\bullet$} node [below=4pt] at (4.5,1.1) {$S_4$};
  \node (12) at (6,1.1) {$\circ$} node [below=4pt] at (6,1.1) {$C_4$};

  \node (16) at (0,-1.3) {$\bullet$} node [left=20pt] at (0,-1.1) {$\G^B_4$ :} node [below=4pt] at (0,-1.3) {$\overline{\overline{B_2}}$};  
  \node (17) at (1.5,-1.3) {$\circ$} node [below=4pt] at (1.5,-1.3) {$\overline{B_3}$};
  \node (18) at (3,-1.3) {$\circ$} node [below=4pt] at (3,-1.3) {$B_4$};
  \node (19) at (4.5,-1.3) {$\bullet$} node [below=4pt] at (4.5,-1.3) {$S_4$};
  \node (20) at (6,-1.3) {$\circ$} node [below=4pt] at (6,-1.3) {$C_4$};

  \path[-] 
	(10)  edge node [above=4pt] {2 \qquad 1} (11)
	(11)  edge node [above=4pt] {1 \qquad 1} (12)
	(16)  edge node [above=4pt] {2 \qquad 2} (17)
	(17)  edge node [above=4pt] {1 \qquad 1} (18)
	(18)  edge node [above=4pt] {2 \qquad 1} (19)
	(19)  edge node [above=4pt] {1 \qquad 1} (20);
  \path[-]
	  	(10) edge [loop left] node [above=8pt,right=4pt] {1} (10)
	  	(11) edge [loop above] (11)
	  	(12) edge [loop right] node [above=8pt, left=4pt] {2} (12)
	  	(16) edge [loop left] node [above=8pt,right=4pt] {1} (16)
	  	(19) edge [loop above] (19)
	  	(20) edge [loop right] node [above=8pt, left=4pt] {2} (20);
\end{tikzpicture} \\
\bigskip
\begin{tikzpicture}[every loop/.style={}]
  \tikzstyle{every node}=[inner sep=-1pt]
  
  \node (0) at (0,-.3) {$\bullet$} node [left=20pt] at (0,0) {$\G_5$ :} node [below=4pt] at (0,-.3) {$\overline{\overline{\overline{B_2}}}$};  
  \node (1) at (1.5,-.3) {$\circ$} node [below=4pt] at (1.5,-.3) {$\overline{\overline{B_3}}$};
  \node (2) at (3,-.3) {$\circ$} node [below=4pt] at (3,-.3) {$\overline{B_4}$};
  \node (3) at (4.5,-.3) {$\bullet$} node [below=4pt] at (4.5,-.3) {$B_5$};
  \node (4) at (6,0) {$\circ$} node [below=4pt] at (6,0) {$S_5$};
  \node (5) at (3,.3) {$\circ$} node [below=4pt] at (3,.3) {$\overline{A_4}$};
  \node (6) at (4.5,.3) {$\bullet$} node [below=4pt] at (4.5,.3) {$A_5$};

  \path[-] 
	(0)  edge (1)
	(1)  edge (2)
	(2)  edge (3)
	(5)  edge (6)
	(6)  edge (4)
	(3)  edge (4);
	  
  \node (10) at (3,1.1) {$\circ$} node [left=20pt] at (0,1.1) {$\G^A_5$ :} node [below=4pt] at (3,1.1) {$\overline{A_4}$};    
  \node (11) at (4.5,1.1) {$\bullet$} node [below=4pt] at (4.5,1.1) {$A_5$};
  \node (12) at (6,1.1) {$\circ$} node [below=4pt] at (6,1.1) {$S_5$};

  \node (16) at (0,-1.3) {$\bullet$} node [left=20pt] at (0,-1.3) {$\G^B_5$ :} node [below=4pt] at (0,-1.3) {$\overline{\overline{\overline{B_2}}}$};  
  \node (17) at (1.5,-1.3) {$\circ$} node [below=4pt] at (1.5,-1.3) {$\overline{\overline{B_3}}$};
  \node (18) at (3,-1.3) {$\circ$} node [below=4pt] at (3,-1.3) {$\overline{B_4}$};
  \node (19) at (4.5,-1.3) {$\bullet$} node [below=4pt] at (4.5,-1.3) {$B_5$};
  \node (20) at (6,-1.3) {$\circ$} node [below=4pt] at (6,-1.3) {$S_5$};

  \path[-] 
	(10)  edge node [above=4pt] {2 \qquad 1} (11)
	(11)  edge node [above=4pt] {1 \qquad 1} (12)
	(16)  edge node [above=4pt] {2 \qquad 2} (17)
	(17)  edge node [above=4pt] {1 \qquad 1} (18)
	(18)  edge node [above=4pt] {2 \qquad 1} (19)
	(19)  edge node [above=4pt] {1 \qquad 1} (20);
  \path[-]
	  	(10) edge [loop left] node [above=8pt,right=4pt] {1} (10)
	  	(11) edge [loop above] (11)
	  	(12) edge [loop right] node [above=8pt, left=4pt] {2} (12)
	  	(16) edge [loop left] node [above=8pt,right=4pt] {1} (16)
	  	(19) edge [loop above] (19)
	  	(20) edge [loop right] node [above=8pt, left=4pt] {2} (20);
\end{tikzpicture} 
\end{center}

\end{example}

\section{Induction algorithm for acyclic Sturmian colorings}\label{section:wocycle}

Throughout Section~\ref{section:wocycle} and Section~\ref{section:4}, we assume that $\phi$ is acyclic, i.e. $\G_n$ does not have any cycle of length larger than 1 for all $n$.

\subsection{Concatenation of $\G_n^A, \G_n^B$}\label{subsec:3.2}

The following lemma gives some special condition which ensures that the index $i$ for $\G_{n+1}$ is zero for some balls.

\begin{lemma}\label{lem:conf} 
The following properties hold. (Similar properties hold for $i_B$.)
\begin{enumerate}
\item For $D \neq A_n, B_n, S_n, C_n$, $i(D,S_n) = 0.$
 If $A_n \ne S_n, C_n$, then $i(A_n, S_n) = i(\overline{A_n}, A_{n+1})$.

%\begin{equation*}
%%i(U, V) = i(\overline{U}, \overline{V}), \qquad 
%i(U, S_n) = \begin{cases}
%0 ,  &\text{ if } U \ne A_n, B_n,   \\
%i(\overline{A_n}, A_{n+1}), &\text{ if } U= A_n, \\
%i(\overline{B_n}, B_{n+1}), &\text{ if } U= B_n.
%\end{cases}
%\end{equation*}

\item If $C_n \ne S_n$, then $i_A(S_n, S_n) = i(A_{n+1}, A_{n+1})$. 
%and $i_B(S_n, S_n) = i(B_{n+1},B_{n+1}).$

If furthermore $A_{n+1} \ne C_{n+1},$ 
%(resp. $B_{n+1} \ne C_{n+1}$), 
then $i(C_n, S_n) = i_B (S_{n+1}, B_{n+1}).$ 
%(resp. $\ i(C_n, S_n) = i_A (S_{n+1}, A_{n+1})).$

\item
If $C_n = S_n$, say $A_{n+1} = S_{n+1}$ and $B_{n+1} \ne C_{n+1}$, 
% (resp. $B_{n+1} = S_{n+1}$, $A_{n+1} \ne C_{n+1}$), 
then 
$i_A(S_n, S_n) = i_A(A_{n+1}, A_{n+1}).$
%\quad  (\text{resp. }\ i_B(S_n, S_n) = i_B(B_{n+1}, B_{n+1}).
\end{enumerate}
\end{lemma}

\begin{proof}
(1) It follows from Lemma~\ref{lem:merged} (1) that  $i(D,S_n) = 0$  for $D \ne A_n, B_n, C_n, S_n$.
If $A_n \ne S_n, C_n$, 
then by Lemma~\ref{lem:merged} (2) $i(\overline{A_n}, B_{n+1}) = 0$.
Lemma~\ref{lem:index} (1) implies that $i(A_n, S_{n}) =  i(\overline{A_n}, A_{n+1})$.
% and $i(B_n, S_{n}) =  i(\overline{B_n}, B_{n+1})$.

(2) $C_n \ne S_n$ implies that $S_{n+1} \ne A_{n+1}$, $S_{n+1} \ne B_{n+1}$.
If $A_{n+1}$ and $B_{n+1}$ are adjacent, then $[S_{n+1} A_{n+1}  B_{n+1}  S_{n+1}]$ is a cycle in $\G_{n+1}$,
thus $i(A_{n+1}, B_{n+1}) = i(B_{n+1}, A_{n+1}) =0$. %Apply Lemma~\ref{lem:index} again.
{Moreover,} if $A_{n+1} \ne C_{n+1}$, then by Lemma~\ref{lem:merged} (2) we deduce that $i_B (S_{n+1}, A_{n+1}) = 0$.
%The case of $B_{n+1} \ne C_{n+1}$ is analogous.
By Lemma~\ref{lem:index}, the second assertion follows.

(3) If $A_{n+1} = S_{n+1}$, $B_{n+1} \ne C_{n+1}$, then since $B_{n+1} \ne A_{n+1} = S_{n+1}$
it follows from Lemma~\ref{lem:merged} {(2)} that $i_A(S_{n+1}, B_{n+1}) = 0$.
Using Lemma~\ref{lem:index}, we complete the proof.
\end{proof}

\begin{definition} \label{def:sum}[Concatenation]    %%d I changed the definition
Let $\G_1, \G_2$ be two edge-indexed graphs.
%which have one end colored by an $n$-ball $V$ {which} is $C_n \neq S_n$ or $S_n=C_n$. 
%\seon{The condition with $C_n, S_n$ is bad for the inverse process}
Let $m_i,\ell_i$ ($i=1,2$) be the indices of the edge and the loop coming out of $V$ in $\G_i$. 
When $m_1=m_2$ and $\ell_1=\ell_2$, we write them as $m$ and $\ell$.

\begin{enumerate}
\item For $1 \le i < m$, the $(i)$-concatenation $\G_1 \overset{V,V'}{\underset{i}{\bowtie}} \G_2$ of $\G_1$ and $\G_2$ at $(V,V')$ is the edge-indexed graph defined as follows : the vertex set is $V \G_1 \cup V \G_2$, where only the vertices $V \in \G_1$ and $V' \in \G_2$ are identified (all the other vertices of $\G_1$ and $\G_2$ are distinct in $\G_1 \overset{V, V'}{\underset{i}{\bowtie}} \G_2$). The oriented edge set is $E\G_1 \cup E\G_2$, where the loop at $V$ in $\G$ and the loop at $V$ in $\G_2$ are identified. 
The edge-index is $i, m-i$ for the non-loop edge at $V$ in $\G_1$, $\G_2$, respectively, and $\ell$ for the loop at $V$. Note that $\ell$ can be zero.
%$(\G \setminus \{ V \}) \cup ( \G' \setminus \{ V \} ) \cup \{ V \}$,
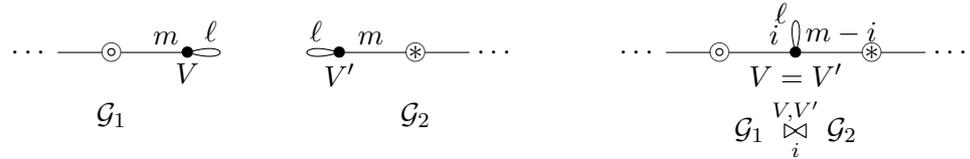
\begin{center}
\begin{tikzpicture}[every loop/.style={}]
  \tikzstyle{every node}=[inner sep=-1pt]
  \node (0) at (0,0) {$\cdots$ \ };
  \node (1) at (1,0) {$\circledcirc$}node [below = 20 pt] at (1,0) {$\G_1$};
  \node (2) at (2,0) {$\bullet$} node [below=5pt] at (2,0) {$V$};
  \path[-] 
	(0) edge (1)
	(1) edge node [above=4pt] {\quad $m$} (2);
  \path[-] (2) edge [loop right] node [above=8pt,left=2pt] {$\ell$} (2);

  \node (6) at (4,0) {$\bullet$} node [below=5pt] at (4,0) {$V'$} ;
  \node (7) at (5,0) {$\circledast$} node [below = 20 pt] at (5,0) {$\G_2$};
  \node (8) at (6,0) {\ $\cdots$};
  \path[-] 
	(7) edge (8)
	(6) edge node [above=4pt] {$m$ \quad} (7);
  \path[-] (6) edge [loop left] node [above=8pt,right=2pt] {$\ell$} (6);

  \node (10) at (8,0) {$\cdots$ \ };
  \node (11) at (9,0) {$\circledcirc$} ;
  \node (12) at (10,0) {$\bullet$} node [below=5pt] at (10,0) {$V=V'$} node [below = 18 pt] at (10,0) {$\G_1 \overset{V,V'} {\underset{i}  {\bowtie}} \G_2$ };
  \node (13) at (11,0) {$\circledast$}; 
  \node (14) at (12,0) {\ $\cdots$};
  \path[-] 
	(10) edge (11)
	(13) edge (14)
	(11) edge node [above=4pt] {\quad $i$} (12)
	(12) edge node [above=4pt] {\quad $m-i$ \quad} (13);
  \path[-] (12) edge [loop above] node [above=3pt,left=4pt] {$\ell$} (12);
\end{tikzpicture} 
\end{center}

\item For $1 \leq i \leq \ell_1, 1 \leq j \le \ell_2$, the $(i,j)$-concatenation $\G_1  \underset{i,j}  {\bowtie} \G_2$ of $\G_1$ and $\G_2$ at $(V, V')$ is the edge-indexed graph with vertex set $V\G_1 \sqcup V\G_2$ (the vertices $V$ in $\G_1$ and $\G_2$ are distinct in the concatenation) and the oriented edge set $E\G_1 \sqcup E \G_2 \sqcup \{e, \overline{e}\}$ with one new pair of oriented edges $e, \overline{e}$ between $V$ in $\G_1$ and $V$ in $\G_2$. 
The edge-index is $\ell_1-i$, $\ell_2-j$ for loops from $\G_1$ and $\G_2$, respectively, and $i(V, V')=i, i(V', V)=j$ for $e$.
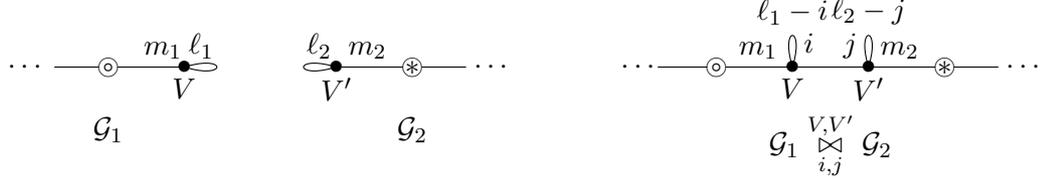
\begin{center}
\begin{tikzpicture}[every loop/.style={}]
  \tikzstyle{every node}=[inner sep=-1pt]
  \node (0) at (0,0) {$\cdots$ \ };
  \node (1) at (1,0) {$\circledcirc$} node [below = 20 pt] at (1,0) {$\G_1$};
  \node (2) at (2,0) {$\bullet$} node [below=5pt] at (2,0) {$V$};
  \path[-] 
	(0) edge (1)
	(1) edge node [above=4pt] {\quad $m_1$} (2);
  \path[-] (2) edge [loop right] node [above=8pt,left=2pt] {$\ell_1$} (2);

  \node (6) at (4,0) {$\bullet$} node [below=5pt] at (4,0) {$V'$} ;
  \node (7) at (5,0) {$\circledast$} node [below = 20 pt] at (5,0) {$\G_2$};
  \node (8) at (6,0) {\ $\cdots$};
  \path[-] 
	(7) edge (8)
	(6) edge node [above=4pt] {$m_2$ \quad} (7);
  \path[-] (6) edge [loop left] node [above=8pt,right=2pt] {$\ell_2$} (6);

  \node (10) at (8,0) {$\cdots$\ };
  \node (11) at (9,0) {$\circledcirc$};
  \node (12) at (10,0) {$\bullet$} node [below=5pt] at (10,0) {$V$} node [below = 18 pt] at (10.5,0) {$\G_1 \overset{V,V'}{\underset{i,j}{\bowtie}} \G_2$ };

  \node (13) at (11,0) {$\bullet$} node [below=5pt] at (11,0) {$V'$};
  \node (14) at (12,0) {$\circledast$};
  \node (15) at (13,0) {\ $\cdots$};
  \path[-] 
	(10) edge (11)
	(14) edge (15)
	(12) edge node [above=4pt] {$i \quad j$} (13)
	(11) edge node [above=4pt] {\quad $m_1$ \;\;} (12)
	(13) edge node [above=4pt] {$m_2$ \quad} (14);
  \path[-] (12) edge [loop above] node [above=5pt] {$\ell_1-i$} (12);
  \path[-] (13) edge [loop above] node [above=5pt] {$\ell_2-j$} (13);
\end{tikzpicture} 
\end{center}
\end{enumerate}
\end{definition}
Let us call $V,V'$ the \emph{joining vertices} of the concatenation.
%For simplicity, let us denote $\G_2$ by $\G_1 \underset{0,j}{\bowtie}  \G_2$ .
We will denote by $\G_1 \bowtie \G_2$ either an $i$-concatenation or an $(i,j)$-concatenation of $\G_1$ and $\G_2$. In this section, when we omit the vertices $V, V'$, all the concatenations of $\G^A_n$, $\G^B_n$ are at $V=C_n, V'=C_n$.

\begin{theorem}\label{thm:4p}
Acyclic Sturmian colorings enjoy the following properties. 
\begin{enumerate}
\item If $A_{n+1} \ne S_{n+1}$ and $A_{n+1} \ne C_{n+1}$, then
$\G^B_{n+1} \cong \G^B_{n}.$

\item If $S_n \neq C_n$ and $A_{n+1} = C_{n+1}$, then for some $i$,
$\G^B_{n+1} \cong \G^A_{n} \underset{i}  {\bowtie}  \G^B_{n}.$
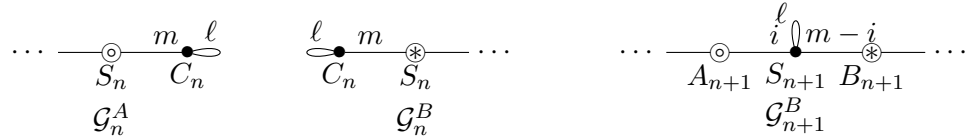
\begin{center}
\begin{tikzpicture}[every loop/.style={}]
  \tikzstyle{every node}=[inner sep=-1pt]
  \node (0) at (0,0) {$\cdots$ \ };
  \node (1) at (1,0) {$\circledcirc$} node [below=5.5pt] at (1,0) {$S_n$} node [below = 20 pt] at (1,0) {$\G^A_n$};
  \node (2) at (2,0) {$\bullet$} node [below=5pt] at (2,0) {$C_n$};
  \path[-] 
	(0) edge (1)
	(1) edge node [above=4pt] {\quad $m$} (2);
  \path[-] (2) edge [loop right] node [above=8pt,left=2pt] {$\ell$} (2);

  \node (6) at (4,0) {$\bullet$} node [below=5pt] at (4,0) {$C_n$} ;
  \node (7) at (5,0) {$\circledast$} node [below=5.5pt] at (5,0) {$S_n$} node [below = 20 pt] at (5,0) {$\G^B_n$};
  \node (8) at (6,0) {\ $\cdots$};
  \path[-] 
	(7) edge (8)
	(6) edge node [above=4pt] {$m$ \quad} (7);
  \path[-] (6) edge [loop left] node [above=8pt,right=2pt] {$\ell$} (6);

  \node (10) at (8,0) {$\cdots$ \ };
  \node (11) at (9,0) {$\circledcirc$} node [below=5.5pt] at (9,0) {$A_{n+1}$}  ;
  \node (12) at (10,0) {$\bullet$} node [below=5pt] at (10,0) {$S_{n+1}$} node [below = 18 pt] at (10,0) {$\G^B_{n+1}$ };
  \node (13) at (11,0) {$\circledast$} node [below=5.5pt] at (11,0) {$B_{n+1}$} ; 
  \node (14) at (12,0) {\ $\cdots$};
  \path[-] 
	(10) edge (11)
	(13) edge (14)
	(11) edge node [above=4pt] {\quad $i$} (12)
	(12) edge node [above=4pt] {\quad $m-i$ \quad} (13);
  \path[-] (12) edge [loop above] node [above=3pt,left=4pt] {$\ell$} (12);
\end{tikzpicture} 
\end{center}
\item If (i) $A_{n+1} = S_{n+1}$ or (ii) $S_n = C_n$ and $A_{n+1} = C_{n+1}$, then for some $i,j$,
$\G^B_{n+1} \cong \G^A_{n} \underset{i,j}  {\bowtie}  \G^B_{n}.$
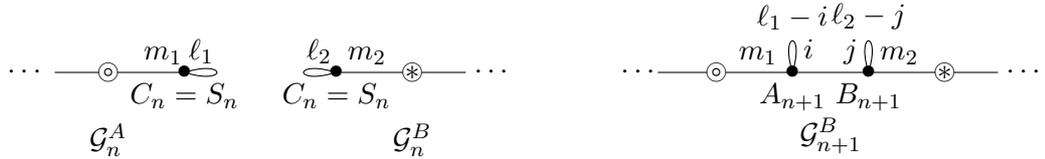
\begin{center}
\begin{tikzpicture}[every loop/.style={}]
  \tikzstyle{every node}=[inner sep=-1pt]
  \node (0) at (0,0) {$\cdots$ \ };
  \node (1) at (1,0) {$\circledcirc$} node [below = 20 pt] at (1,0) {$\G^A_n$};
  \node (2) at (2,0) {$\bullet$} node [below=5pt] at (2,0) {$C_n=S_n$};
  \path[-] 
	(0) edge (1)
	(1) edge node [above=4pt] {\quad $m_1$} (2);
  \path[-] (2) edge [loop right] node [above=8pt,left=2pt] {$\ell_1$} (2);

  \node (6) at (4,0) {$\bullet$} node [below=5pt] at (4,0) {$C_n=S_n$} ;
  \node (7) at (5,0) {$\circledast$} node [below = 20 pt] at (5,0) {$\G^B_n$};
  \node (8) at (6,0) {\ $\cdots$};
  \path[-] 
	(7) edge (8)
	(6) edge node [above=4pt] {$m_2$ \quad} (7);
  \path[-] (6) edge [loop left] node [above=8pt,right=2pt] {$\ell_2$} (6);

  \node (10) at (8,0) {$\cdots$\ };
  \node (11) at (9,0) {$\circledcirc$};
  \node (12) at (10,0) {$\bullet$} node [below=5pt] at (10,0) {$A_{n+1}$} node [below = 18 pt] at (10.5,0) {$\G^B_{n+1}$ };

  \node (13) at (11,0) {$\bullet$} node [below=5pt] at (11,0) {$B_{n+1}$};
  \node (14) at (12,0) {$\circledast$};
  \node (15) at (13,0) {\ $\cdots$};
  \path[-] 
	(10) edge (11)
	(14) edge (15)
	(12) edge node [above=4pt] {$i \quad j$} (13)
	(11) edge node [above=4pt] {\quad $m_1$ \;\;} (12)
	(13) edge node [above=4pt] {$m_2$ \quad} (14);
  \path[-] (12) edge [loop above] node [above=5pt] {$\ell_1-i$} (12);
  \path[-] (13) edge [loop above] node [above=5pt] {$\ell_2-j$} (13);
\end{tikzpicture} 
\end{center}
\end{enumerate}
By symmetry, similar properties hold for $\G^A_{n+1}$.

The isomorphism of edge-indexed graphs from $\G^B_{n+1}$ to $\G^B_{n}$ or $\G^{A}_{n} \bowtie \G^B_{n} $ is given by the restriction of $(n+1)$-balls to $n$-balls.
$D \mapsto \underline D$.
\end{theorem}

\begin{proof}
We show the edge-indexed graph isomorphism between $\G^B_{n+1}$ and $\G^B_{n}$ or $\G^{A}_{n} \bowtie \G^B_{n} $ by the extension of $n$-ball or the restriction of $(n+1)$-balls.

(1) If $S_n \ne C_n$, then by Lemma~\ref{lem:conf} (2),
$i (C_n, S_n) = i_B(S_{n+1}, B_{n+1})$ and $i_B (S_n, S_n) = i(B_{n+1}, B_{n+1})$. 
Also, by Lemma~\ref{lem:conf} (1), $i (B_n, S_n) = i( \overline{B_n}, B_{n+1})$ if $B_n \ne S_n, C_n$. Since there is no other class of $n$-balls adjacent to $S_n$ by Lemma~\ref{2.11KL} (2), we get
$\G^B_{n+1} \cong \G^B_{n}.$

If $S_n = C_n$, then $B_{n+1} = S_{n+1}$ and $A_{n+1} \ne C_{n+1}$.
Thus by Lemma~\ref{lem:conf} (3)
$i_B (S_n, S_n) = i_B(B_{n+1}, B_{n+1})$.
Also, by Lemma~\ref{lem:conf} (i), $i (B_n, S_n) = i( \overline{B_n}, B_{n+1})$ if $B_n \ne S_n= C_n$.
Since there is no other adjacent $n$-ball to $S_n$, we get
$\G^B_{n+1} \cong \G^B_{n}.$

(2) If $S_n \ne C_n$ and $A_{n+1} = C_{n+1}$, then 
 by Lemma~\ref{lem:conf} (2)
$i_B(S_n,S_n) = i(B_{n+1},B_{n+1}).$
By Lemma~\ref{lem:index} (2),
$$i(C_n, S_n) = i_B (S_{n+1}, A_{n+1}) + i_B (S_{n+1}, B_{n+1}),$$ 
where $i_B (S_{n+1}, A_{n+1}) >0$, $i_B (S_{n+1}, B_{n+1}) >0$ from Lemma~\ref{2.11KL} (4) and (1), respectively.
This case corresponds to
$$\G^B_{n+1} \cong \G^A_n \underset{i}  {\bowtie}  \G^B_n,$$
where 
$i =  i_B (S_{n+1}, A_{n+1}). $

(3) If $A_{n+1} = S_{n+1}$ (thus $S_n = C_n$), by Lemma~\ref{lem:index} (3), we have
\begin{align*}
i_A(S_n, S_n) &= i_B(A_{n+1}, A_{n+1}) +i_B(A_{n+1},B_{n+1}), \\
i_B(S_n, S_n) &= i(B_{n+1}, A_{n+1}) +i(B_{n+1},B_{n+1}).
\end{align*}
where 
$ i_B (A_{n+1}, B_{n+1}) >0$, $i (B_{n+1}, A_{n+1}) >0$ by Lemma~\ref{2.11KL} (1). 
This case corresponds to
$$\G^B_{n+1} \cong \G^A_n  \underset{i,j}  {\bowtie} \G^B_n,$$
where 
$i =  i_B (A_{n+1}, B_{n+1})$ and $ j =  i (B_{n+1}, A_{n+1}). $

Suppose that $A_{n+1} \ne S_{n+1}$, $S_n = C_n$ and $A_{n+1} = C_{n+1}$.
By Lemma~\ref{lem:index} (3) we have
\begin{align*}
i_A(S_n, S_n) &= i(A_{n+1}, A_{n+1}) +i(A_{n+1},B_{n+1}), \\
i_B(S_n, S_n) &= i_B(B_{n+1}, A_{n+1}) +i_B(B_{n+1},B_{n+1}).
\end{align*}
where 
$ i(A_{n+1}, B_{n+1}) >0$, $i_B (B_{n+1}, A_{n+1}) >0$ by Lemma~\ref{2.11KL} (1) and (4) respectively. 
This case corresponds to
$$\G^B_{n+1} \cong \G^A_n  \underset{i,j}  {\bowtie} \G^B_n,$$
where 
\begin{equation*}i =  i (A_{n+1}, B_{n+1}), \qquad j =  i_B (B_{n+1}, A_{n+1}). \qedhere \end{equation*}
%\seon{proof of the last assertion.}
\end{proof}

\subsection{Induction algorithm} In this section, we prove Theorem~\ref{thm:main}.
\begin{definition}
We define $n_k$ to be the sequence of integers $n_k \ge 0$ such that at least one of $A_{n_k}$ or $B_{n_k}$ is identical to $S_{n_k}$ or $C_{n_k}$. 
%\seon{This definition is strange. It does not involve $k$.}
By Lemma~\ref{lem:n0}, we have $n_k = k$ for $0 \leq k \leq \NN$. 
\end{definition}

The constant $K$ and the sequence $(n_k)$ indicate for which $n$'s either $\G^A_n$ or $\G^B_n$ becomes strictly larger than $\G^A_{n-1}, \G^B_{n-1}$.
%Using the above observations, we have the following theorem. 
\begin{theorem}\label{thm:4} 
Let $\phi$ be an acyclic Sturmian coloring.   
%There exist $\NN \in [0,\infty]$ and a sequence $(n_k)_{k=0}^{\infty}$ with the following properties. More precisely, $\NN$ is defined by part (1),(2) and $n_k$ is defined by part (3).

\begin{enumerate}
\item For $0 \le n < \NN$, we have $A_{n}=C_{n}=S_{n}$, thus for some $i, j$
$$\G^A_{n} \cong \G^A_{n-1}, \qquad \G^B_{n} \cong \G^A_{n-1} \underset{i,j}  {\bowtie}  \G^B_{n-1}.$$

\item For $n = \NN$, we have the following cases:
\begin{enumerate}
\item If $A_{n}=S_{n}$, $B_{n} =C_{n}$, then for some $i <  i'$, $j$
$$\G^A_{n} \cong \G^A_{n-1} \underset{i,j}  {\bowtie}  \G^B_{n-1},  \qquad \G^B_{n} \cong \G^A_{n-1}  \underset{i',j}  {\bowtie} \G^B_{n-1}.$$
\item If $B_{n} =S_{n}$, then $C_{n} = \overline{B_{n-1}}$ and 
$$\G^A_{n} \cong \G^A_{n-1} \underset{i,j}  {\bowtie}  \G^B_{n-1},  \qquad \G^B_{n} \cong \G^B_{n-1}.$$
\end{enumerate}

\item
For every $n > \NN$,
$$ \G^A_{n} \cong \G^A_{n-1}, \  \G^B_{n} \cong \G^B_{n-1}, \qquad \qquad \mathrm{if} \;\;n \neq n_k,$$
$$\G^A_{n} \cong \G^A_{n-1} \bowtie \G^B_{n-1}, \  \G^B_{n} \cong \G^B_{n-1} \ \text{ or } \ \G^A_{n} \cong \G^A_{n-1}, \  \G^B_{n} \cong \G^A_{n-1}  \bowtie \G^B_{n-1}, \quad \mathrm{if} \ n= n_k.$$
Here, $\bowtie$ is either $i$-concatenation or $(i,j)$-concatenation.
%\item
%The sequence of $\{\G_n\}_{n \geq 0}$ has to follow a path in the following graph. In the graph, each vertex denotes which of $A_n, B_n, C_n, S_n$ are identical. For example, $A=C$ means that three $n$-balls $A_n = C_n, B_n, S_n$ are distinct.
%
%\item Moreover, $\G_{n+1}$ is an $(i)$-sum when it goes through the vertex $A=C$ or $B=C$ and $(i,j)$-sum when it goes through the vertex $C=S$. 

%\begin{center}
%\begin{tikzpicture}[every loop/.style={}] 
%  \tikzstyle{every node}=[inner sep=3pt]
   
%  \node[draw] (0) at (0,1) {$A=C$} ;
%    \node[draw] (1) at (0,-1) {$B=C$} ;
%  \node[ draw] (2) at (3,1) {$\emptyset$};
%  \node[ draw] (3) at (6,1) {$C=S$};
%   \node[draw] (4) at (3.8,-1) {$A=S$};
%      \node[draw] (5) at (5,-1) {$B=S$};

% \draw[<<->>] (0) --(2);
%  \draw[<<->>] (1) --(2);
% \draw[->>] (5) --(2);
% \draw[->>] (2) --(3);
% \draw[->>] (4) --(2);
% \draw[<<-] (4) --(3);
%  \draw[<<-] (5) --(3);

%   \path[->>] (2) edge [loop] (2);
%\end{tikzpicture} 
%\end{center}
\end{enumerate}

\end{theorem}

\begin{proof}
(1) For $0 \le n < \NN$, by Lemma~\ref{lem:n0}, $A_n=C_n=S_n \ne B_n$.
Theorem~\ref{thm:4p} (1) and (3) implies that $\G^A_{n} \cong \G^A_{n-1} $ and $\G^B_{n} \cong \G^A_{n-1} \underset{i,j}  {\bowtie}  \G^B_{n-1}$ for some $i,j$. 

(2) Let $n = \NN$. Then $C_{n-1} = S_{n-1}$ which implies either (a) $S_n = A_n$ or (b) $S_n = B_n$.

(a) $S_n = A_n$ : we have $C_n \ne A_n = S_n$ by the definition of $\NN$. 
By Lemma~\ref{2.11KL} (4) and Lemma~\ref{lem:merged} (1) we have $B_n = C_n$.  
Thus, Theorem~\ref{thm:4p} (3) implies that 
$\G^A_{n} \cong \G^A_{n-1} \underset{i',j}  {\bowtie}  \G^B_{n-1}$ and
$\G^B_{n} \cong \G^A_{n-1} \underset{i,j}  {\bowtie}  \G^B_{n-1}$ for some $i',i,j$. 

(b) $S_n = B_n \ne A_n$: by the choice of $A_0$ and $B_0$, $n=\NN \ge 1$.
$S_n$ is adjacent to distint balls $A_n$ and $\overline{B_{n-1}}$.
By Lemma~\ref{lem:merged} (1), we have $C_n = \overline{B_{n-1}}\ne A_n$,
thus Theorem~\ref{thm:4p} (1)  and (3) implies that $\G^A_{n} \cong \G^A_{n-1} \underset{i,j}  {\bowtie}  \G^B_{n-1}$ for some $i,j$ and $\G^B_{n} \cong \G^B_{n-1}$.

(3) Assume that $n > \NN$.
If $S_{n-1} = C_{n-1}$, then we have either $A_{n} = S_{n}$ or $B_{n} = S_{n}$.
Suppose that $B_{n} \ne A_{n} = S_{n}$.
If $B_{n} = C_{n}$, then 
by Lemma~\ref{2.11KL} (4) and Lemma~\ref{lem:merged} (1), $A_n$ is adjancent to $A_n$ and $B_n$ only. By Lemma~\ref{2.11KL} (1), $A_n$ is adjacent to $A_{n-1}$, which implies that $A_{n-1} = S_{n-1} = C_{n-1}$, which contradicts $n \ge \NN +1$. 
Thus $B_{n} \ne C_{n}$ and Theorem~\ref{thm:4p} (1) and  (3) implies that $\G^A_{n} \cong \G^A_{n-1} $ and $\G^B_{n} \cong \G^A_{n-1} \underset{i,j}  {\bowtie}  \G^B_{n-1}$ for some $i,j$. 
For the case $A_{n} \ne B_{n} = S_{n}$, we apply the same argument.

If $S_{n-1} \neq C_{n-1}$, then we have $A_{n}, B_{n} \ne S_{n}$
Suppose that $B_{n} \ne  A_{n} = C_{n}$ or $A_{n} \ne  B_{n} = C_{n}$.
Then by Theorem~\ref{thm:4p} (1), (2) for some $i$,
$\G^A_{n} \cong \G^A_{n-1} $, $\G^B_{n} \cong \G^A_{n-1} \underset{i}  {\bowtie}  \G^B_{n-1}$
or
$\G^A_{n} \cong \G^A_{n-1} \underset{i}  {\bowtie}  \G^B_{n-1}$, $\G^B_{n} \cong \G^B_{n-1}.$
In case of $A_{n}, B_{n} \ne C_{n}$, by Theorem~\ref{thm:4p} (1),
$\G^A_{n} \cong \G^A_{n-1} $ and $\G^B_{n} \cong \G^B_{n-1}.$
This is the case of $n \ne n_k$.
\end{proof}

For $\alpha \in \{A, B\}$, define $\overline{\alpha}=B$ if $\alpha=A$ and vice versa. Recall from Section~\ref{ak} that $\alpha_k$ satisfies
$|V \G^{\alpha_k}_{n_k}|> |V\G^{\overline{\alpha_k}}_{n_k}|$ if $k \neq K$ and $\alpha_K = A$.
The following proposition shows how $n_k$ is related to $\alpha_k$ and $i_k, j_k$. 
It will be used in Theorem~\ref {thm:bounded1}.

\begin{proposition}\label{lem:nk}     
%If $\G^A_{n_k} \cong \G^A_{n_k-1} \bowtie  \G^B_{n_k-1}$ and $\G^B_{n_k} \cong \G^B_{n_k-1}$, then 
The special ball $S_{n_k}$ is a vertex of degree 1 in $\G^{\overline\alpha_k}_{n_k}$.
Put $m =|V\G^{\overline\alpha_k}_{n_k}|$.
For $i= 0,1, \cdots, m-1$, the vertex of $\G^{\overline\alpha_k}_{n_k}$ of distance $i$ from $S_{n_k}$, which we denote by $[i]$, is the central $n_k$-ball of $S_{n_k+i}$.
Therefore
$$ n_{k+1} = \begin{cases} n_{k} + m -1 , &\text{ if there is a $(i)$-concatenation at } n_{k+1},\\
n_{k} + m, &\text{ if there is a $(i,j)$-concatenation at } n_{k+1}.
\end{cases} $$
Moreover, none of the vertices of $\G^{\alpha_k}_{n_k-1}$ is the center of a special $n$-ball if $n_k + 1 \le n \le n_{k+1}$ for $(i)$-concatenation at $n_{k}$
and if $n_k  \le n \le n_{k+1}$ for $(i,j)$-concatenation at $n_{k}$.
\end{proposition}

\begin{proof}
We may assume that $\alpha_k = A$.  
From the definition of concatenation, it is immediate that $C_{n_k-1}$ is a vertex of degree 1 in $\G^{B}_{n_k}$. 
By Theorem~\ref{thm:4}, $S_{n_k} \ne C_{n_k}$.
By Lemma~\ref{2.11KL} (4), the only adjacent vertex of $[0]= S_{n_k}$ is $[1]$, which is equal to $C_{n_k}$. 
By the canonical projection from $\G^{B}_{n_k}$ to $\G^{B}_{n_k+1}$, the vertices $[0], [1]$ are mapped to $B_{n_k+1}, S_{n_k+1}$ respectively, thus again Lemma~\ref{2.11KL} (4) implies that $[2]$ is mapped to $C_{n_k+1}$.
Inductively, for $1 \le i \le m-1$ the vertex $[i]$ is mapped to $S_{n_k+i}$ by the canonical projection from $\G^{B}_{n_k}$ to $\G^{B}_{n_k+i}$.
By Theorem~\ref{thm:4}, we continue this procedure until $n_{k+1}$.  
\end{proof}

\begin{example}
Let $s_i, t_i$, $i = 1,2,3$ be integers satisfying $t_i \ge 1$, $s_i \ge 0$, $s_i + 2 t_i = d$ for each $i = 1,2,3$ and $s_1 \ne s_2$. Example 7 in \cite{KL1} is the following:

\begin{center}
\renewcommand{\arraystretch}{2}
\begin{tabular}{rl}
\raisebox{1.2\height}{$Y$ :} &
\raisebox{.8\height}{\begin{tikzpicture}[every loop/.style={},scale=.8]
  \tikzstyle{every node}=[inner sep=-1pt]

  \node (b0) at (-6,0) {$\cdots$} ;
  \node (b1) at (-5,0) {$\bullet$} node [above=6pt] at (-5,0) {$c$};
  \node (b3) at (-3,0) {$\bullet$} node [above=6pt] at (-3,0) {$d$};
  \node (b5) at (-1,0) {$\bullet$} node [above=6pt] at (-1,0) {$c$};
  \node (b7) at (1,0) {$\bullet$} node [above=6pt] at (1,0) {$c$};
  \node (b9) at (3,0) {$\bullet$} node [above=6pt] at (3,0) {$d$};
  \node (b11) at (5,0) {$\bullet$} node [above=6pt] at (5,0) {$c$};
  \node (b13) at (7,0) {$\bullet$} node [above=6pt] at (7,0) {$d$};
  \node (b14) at (8,0) {$\cdots$} ;

  \path[-]   
		 (b0) edge (b1)    
		 (b1) edge (b3)    
 		 (b3) edge (b5)    
 		 (b5) edge (b7)    
 		 (b7) edge (b9)    
		 (b9) edge (b11)
 		 (b11) edge (b13)    
 		 (b13) edge (b14);
\end{tikzpicture}}\\
\raisebox{1.5\height}{$X$ :} &
\begin{tikzpicture}[every loop/.style={},scale=.8]
  \tikzstyle{every node}=[inner sep=-1pt]

  \node (0) at (-6,0) {$\cdots$};
  \node (1) at (-5,0) {$\circ$};
  \node (2) at (-4,0) {$\bullet$};
  \node (3) at (-3,0) {$\circ$};
  \node (4) at (-2,0) {$\bullet$};
  \node (5) at (-1,0) {$\circ$};
  \node (6) at (0,0) {$\bullet$};
  \node (7) at (1,0) {$\circ$};
  \node (8) at (2,0) {$\bullet$};
  \node (9) at (3,0) {$\circ$};
  \node (10) at (4,0) {$\bullet$};
  \node (11) at (5,0) {$\circ$};
  \node (12) at (6,0) {$\bullet$};
  \node (13) at (7,0) {$\circ$};
  \node (14) at (8,0) {$\cdots$};

  \path[-] 
		 (0) edge node [below=2pt] {$ t_1$} (1)
		 (1) edge node [below=2pt] {$t_1 \ t_3$} (2)
		 (2) edge node [below=2pt] {$t_3 \ t_2$} (3)
		 (3) edge node [below=2pt] {$t_2 \ t_3$} (4)
		 (4) edge node [below=2pt] {$t_3 \ t_1$} (5)
		 (5) edge node [below=2pt] {$t_1 \ t_3$} (6)
		 (6)  edge node [below=2pt] {$t_3 \ t_1$} (7)
		 (7)  edge node [below=2pt] {$t_1 \ t_3$} (8)
		 (8)  edge node [below=2pt] {$t_3 \ t_2$} (9)
		 (9)  edge node [below=2pt] {$t_2 \ t_3$} (10)
		 (10)  edge node [below=2pt] {$t_3 \ t_1$} (11) 
		 (11)  edge node [below=2pt] {$t_1 \ t_3$} (12) 
		 (12)  edge node [below=2pt] {$t_3 \ t_2$} (13) 
		 (13)  edge node [below=2pt] {$t_2$} (14) 
 		 (1) edge [loop above] node [right=3pt] {$s_1$} (1)    
 		 (2) edge [loop above] node [right=3pt] {$s_3$} (2)  
 		 (3) edge [loop above] node [right=3pt] {$s_2$} (3)    
 		 (4) edge [loop above] node [right=3pt] {$s_3$} (4)  
 		 (5) edge [loop above] node [right=3pt] {$s_1$} (5)    
 		 (6) edge [loop above] node [right=3pt] {$s_3$} (6)  
 		 (7) edge [loop above] node [right=3pt] {$s_1$} (7)    
 		 (8) edge [loop above] node [right=3pt] {$s_3$} (8)  
		 (9) edge [loop above] node [right=3pt] {$s_2$} (9)
 		 (10) edge [loop above] node [right=3pt] {$s_3$} (10)  
 		 (11) edge [loop above] node [right=3pt] {$s_1$} (11)    
 		 (12) edge [loop above] node [right=3pt] {$s_3$} (12)  
 		 (13) edge [loop above] node [right=3pt] {$s_2$} (13) ;
\end{tikzpicture}
\end{tabular}
\end{center}

$$\G^A_0 :
\begin{tikzpicture}[every loop/.style={}]
  \tikzstyle{every node}=[inner sep=-1pt]
  \node (1) at (0,0) {$\circ$};
  \node (2) at (1,0) {$\bullet$};
  \path[-] 
	(1)  edge node [above=4pt] {$2t_1$ $2t_3$} (2);
  \path[-]
	(1) edge [loop left] node [above=6pt,right=3pt] {$s_1$} (1)
	(2) edge [loop right] node [above=6pt,left=3pt] {$s_3$} (2);
\end{tikzpicture}  \qquad
\G^B_0 :
\begin{tikzpicture}[every loop/.style={}]
  \tikzstyle{every node}=[inner sep=-1pt]
  \node (1) at (0,0) {$\circ$};
  \node (2) at (1,0) {$\bullet$};
  \path[-] 
	(1)  edge node [above=4pt] {$2t_2$ $2t_3$} (2);
  \path[-]
	(1) edge [loop left] node [above=6pt,right=3pt] {$s_2$} (1)
	(2) edge [loop right] node [above=6pt,left=3pt] {$s_3$} (2);
\end{tikzpicture}$$
$$\G^A_1 :
\begin{tikzpicture}[every loop/.style={}]
  \tikzstyle{every node}=[inner sep=-1pt]
  \node (1) at (0,0) {$\circ$};
  \node (2) at (1,0) {$\bullet$};
  \path[-] 
	(1)  edge node [above=4pt] {$2t_1$ $2t_3$} (2);
  \path[-]
	(1) edge [loop left] node [above=6pt,right=3pt] {$s_1$} (1)
	(2) edge [loop right] node [above=6pt,left=3pt] {$s_3$} (2);
\end{tikzpicture}  \qquad
\G^B_1 :
\begin{tikzpicture}[every loop/.style={}]
  \tikzstyle{every node}=[inner sep=-1pt]
  \node (0) at (-1,0) {$\circ$};
  \node (1) at (0,0) {$\bullet$};
  \node (2) at (1,0) {$\circ$};
  \path[-] 
	(0)  edge node [above=4pt] {$2t_1$ $t_3$} (1)
	(1)  edge node [above=4pt] {$t_3$ $2t_2$} (2);
  \path[-] 
	(0) edge [loop left] node [above=6pt,right=3pt] {$s_1$} (0)
	(1) edge [loop above] (1)
	(2) edge [loop right] node [above=6pt,left=3pt] {$s_2$} (2);
\end{tikzpicture}$$
$$\G^A_2 :
\begin{tikzpicture}[every loop/.style={}]
  \tikzstyle{every node}=[inner sep=-1pt]
  \node (0) at (0,0) {$\bullet$};
  \node (1) at (1,0) {$\circ$};
  \node (2) at (2,0) {$\bullet$};
  \node (3) at (3,0) {$\circ$};
  \path[-] 
	(0)  edge node [above=4pt] {$2t_3$ $t_1$} (1)
	(1)  edge node [above=4pt] {$t_1$ $t_3$} (2)
	(2)  edge node [above=4pt] {$t_3$ $2t_2$} (3);
  \path[-] 
	(0) edge [loop left] node [above=6pt,right=3pt] {$s_3$} (0)
	(1) edge [loop above] (1)
	(2) edge [loop above] (2)
	(3) edge [loop right] node [above=6pt,left=3pt] {$s_2$} (3);
\end{tikzpicture}  \qquad
\G^B_2 :
\begin{tikzpicture}[every loop/.style={}]
  \tikzstyle{every node}=[inner sep=-1pt]
  \node (0) at (-1,0) {$\circ$};
  \node (1) at (0,0) {$\bullet$};
  \node (2) at (1,0) {$\circ$};
  \path[-] 
	(0)  edge node [above=4pt] {$2t_1$ $t_3$} (1)
	(1)  edge node [above=4pt] {$t_3$ $2t_2$} (2);
  \path[-] 
	(0) edge [loop left] node [above=6pt,right=3pt] {$s_1$} (0)
	(1) edge [loop above] (1)
	(2) edge [loop right] node [above=6pt,left=3pt] {$s_2$} (2);
\end{tikzpicture}$$
$$\G^A_3 :
\begin{tikzpicture}[every loop/.style={}]
  \tikzstyle{every node}=[inner sep=-1pt]
  \node (0) at (0,0) {$\bullet$};
  \node (1) at (1,0) {$\circ$};
  \node (2) at (2,0) {$\bullet$};
  \node (3) at (3,0) {$\circ$};
  \path[-] 
	(0)  edge node [above=4pt] {$2t_3$ $t_1$} (1)
	(1)  edge node [above=4pt] {$t_1$ $t_3$} (2)
	(2)  edge node [above=4pt] {$t_3$ $2t_2$} (3);
  \path[-] 
	(0) edge [loop left] node [above=6pt,right=3pt] {$s_3$} (0)
	(1) edge [loop above] (1)
	(2) edge [loop above] (2)
	(3) edge [loop right] node [above=6pt,left=3pt] {$s_2$} (3);
\end{tikzpicture}  \qquad
\G^B_3 :
\begin{tikzpicture}[every loop/.style={}]
  \tikzstyle{every node}=[inner sep=-1pt]
  \node (0) at (-1,0) {$\circ$};
  \node (1) at (0,0) {$\bullet$};
  \node (2) at (1,0) {$\circ$};
  \path[-] 
	(0)  edge node [above=4pt] {$2t_1$ $t_3$} (1)
	(1)  edge node [above=4pt] {$t_3$ $2t_2$} (2);
  \path[-] 
	(0) edge [loop left] node [above=6pt,right=3pt] {$s_1$} (0)
	(1) edge [loop above] (1)
	(2) edge [loop right] node [above=6pt,left=3pt] {$s_2$} (2);
\end{tikzpicture}$$
$$\G^A_4 :
\begin{tikzpicture}[every loop/.style={}]
  \tikzstyle{every node}=[inner sep=-1pt]
  \node (0) at (0,0) {$\bullet$};
  \node (1) at (1,0) {$\circ$};
  \node (2) at (2,0) {$\bullet$};
  \node (3) at (3,0) {$\circ$};
  \path[-] 
	(0)  edge node [above=4pt] {$2t_3$ $t_1$} (1)
	(1)  edge node [above=4pt] {$t_1$ $t_3$} (2)
	(2)  edge node [above=4pt] {$t_3$ $2t_2$} (3);
  \path[-] 
	(0) edge [loop left] node [above=6pt,right=3pt] {$s_3$} (0)
	(1) edge [loop above] (1)
	(2) edge [loop above] (2)
	(3) edge [loop right] node [above=6pt,left=3pt] {$s_2$} (3);
\end{tikzpicture}  \qquad
\G^B_4 :
\begin{tikzpicture}[every loop/.style={}]
  \tikzstyle{every node}=[inner sep=-1pt]
  \node (0) at (-1,0) {$\circ$};
  \node (1) at (0,0) {$\bullet$};
  \node (2) at (1,0) {$\circ$};
  \node (3) at (2,0) {$\bullet$};
  \node (4) at (3,0) {$\circ$};
  \node (5) at (4,0) {$\bullet$};
  \path[-] 
	(0)  edge node [above=4pt] {$2t_1$ $t_3$} (1)
	(1)  edge node [above=4pt] {$t_3$ $t_2$} (2)
	(2)  edge node [above=4pt] {$t_2$ $t_3$} (3)
	(3)  edge node [above=4pt] {$t_3$ $t_1$} (4)
	(4)  edge node [above=4pt] {$t_1$ $2t_3$} (5);
  \path[-] 
	(0) edge [loop left] node [above=6pt,right=3pt] {$s_1$} (0)
	(1) edge [loop above] (1)
	(2) edge [loop above] (2)
	(3) edge [loop above] (3)
	(4) edge [loop above] (4)
	(5) edge [loop right] node [above=6pt,left=3pt] {$s_3$} (5);
\end{tikzpicture}$$
%The empty vertices in the graphs $\G^A_n$, $\G^B_n$ indicates the vertex of the special ball.
In this example, $K=0$, and moreover, we have  
$n_k = f_{k+2}-1$ where $f_k$ is the Fibonacci sequence defined by 
$f_k = f_{k-1} + f_{k-2}$, $f_1 = 1, f_2 = 1$,
i.e., $n_0=0, n_1 = 1, n_2 = 2,  n_3 =4$ and so on.
Note that
$\alpha_k = A$ if $k$ is even and $\alpha_k = B$ if $k$ is odd. 
%$$\alpha_0 = A, \ \alpha_1 = B, \ \alpha_2 = A,\ \alpha_3=B, \dots.$$
For the sequence of $\mathbf{i}_k$ we have
$v_0 = (2t_1, 2t_2,2 t_3)$ and $v_{3k+i} = (t_{2+i})$ for each $i = -1,0,1$.
%, i.e., $v_1 = (t_3),  v_2 = (t_1),  v_3 = (t_2), v_4 = (t_3), \dots$.
\end{example}

%We say that $\G^A_n$ ($\G^B_n$) has the branch $[C^1, \dots, C^m]$ if
%$C^i \in V\G^A_n$ ($V\G^B_n$ respectively) for $1 \le i \le m$.
%By Lemma~\ref{lem:KL}, each of $\G^A_n$ and $\G^B_n$ has at most two branches.

%By Lemma~\ref{2.11KL} we obtain the following lemma immediately.
%\begin{lemma}\label{lem:branch}
%(i)  If $S_n \ne A_n$ ($S_n \ne B_n$), then $\G^A_n$ ($\G^B_n$) has a branch starting from $A_n$ ($B_n$, respectively).

%(ii) If $S_n \ne C_n$, then
%$\G^A_n$ and $\G^B_n$ have a common branch which starts from $C_n$.
%\end{lemma}

\section{Inverse process for acyclic Sturmian coloring}\label{section:4}
In this section, we characterize acyclic Sturmian colorings completely by showing that the converse of Theorem~\ref{thm:4} also holds : we define admissible sequences which determine sequences of edge-indexed graphs $\F^A_k, \F^B_k$ constructed by $(i)$-concatenations or $(i,j)$-concatenations recursively, so that the appropriate direct limit $\F$ is a linear graph canonically colored by $a,b$. We show that these colorings are Sturmian.

\subsection{Admissible sequences}
Let us first define the admissible sequences of indices which we will use in concatenations.
Set $\D =\{1, 2, \dots, d-1,d\}$, %$\overline{\D} =\{0, 1, 2, \dots, d-1,d\}$, 
where $d$ is the degree of the tree. 
Let $\{ (\alpha_k, \mathbf{i}_k) \}_{k=0}^\infty$ be a sequence of pairs of $\alpha_k \in \{ A,B\}$ and $\mathbf{i}_k$ is one of $i_k \in \D$, $(i_k, j_k)\in \D^2$ or $(i_k, i'_k, j_k) \in \D^3.$
%$\mathbf{i}_k = (i_k) $ or $(i_k, j_k)$ or $(i_k,i'_k,j_k)$.
%Let $(\alpha_k)_{k=0}^\infty \in \{ A,B\}^{\mathbb N}$ and 
Put $K = \min\{ k \ge 0 : \alpha_k = A \}$ and $i_{-1} = 0$.
To define admissibility, we need an extra notation for the edge-indices at the end vertices: define the sequence $(\i^A_k, \i^B_k, \i^C_k)_{k \geq K}$ recursively by
\begin{align*}
\i^A_K = i'_{K}, \ \i^B_K = i_{K-1}, \ \i^C_K =  j_0 & &\text{ for }  i_K =0, \\
\i^A_K = i_{K} , \ \i^B_K = i'_{K}, \ \i^C_K =  j_0 & & \text{ for }  i_K >0
\end{align*}
and for $k>K$
$$\i^C_{k} = \i^{\overline\alpha_k}_{k-1}, \quad  \i^{\alpha_k}_k = i^{\alpha_k}_{k-1}, \quad \i^{\overline\alpha_k}_{k} = \i^C_{k-1}.$$

\begin{definition}[Admissible sequence]\label{def:4.1}
We call a sequence of indices $ \mathbf{i}_k$ an \emph{$\alpha_k$-admissible sequence} if it satisfies the following conditions.
% The sequence $(n_k)$ is obtained by Lemma~\ref{lem:nk}.

\begin{enumerate}
\item If $0 \le k < K$, then $\mathbf{i}_k = (i_k,j_k) \in \D^2$ and satisfies
$$
%1 \le i_0 < R, \ 1 \le j_0 \le R, \quad  
1 \le i_k < d, \qquad 1 \le j_k \le d - i_{k-1}.
$$
\item If $k = K$, then $\mathbf{i}_K = (i_K, i'_K, j_K) \in \D^3$ and satisfies
$$1 \le i_K < i'_K \le d, \qquad 1 \le j_K \le d - i_{K-1},$$
or $\mathbf{i}_K = (i_K,j_K) \in \D^2$ and satisfies
$$
1 \le i_K \le d, \qquad 1 \le j_K \le d - i_{K-1}.
$$
Furthermore $\mathbf{i}_K \in \D^3$  if $K = 0$.
%We call $(e^A, e^B, e^C)$ the three ends configuration vector.
\item If $k > K$, then $\mathbf{i}_k =i_k$ or $(i_k, j_k)$ and satisfies
\begin{align*}
&1 \le i_k < \i^C_{k-1}, & &\text{if } \ \mathbf{i}_k = i_k, \\
&1 \le i_k \le d- \i^C_{k-1}, \quad 1 \le j_k \le d- \i^C_{k-1}, & &\text{if } \ \mathbf{i}_k = (i_k,j_k).
\end{align*}
\end{enumerate}
%\begin{align*}
%e^C_{k} = e^B_{k-1}, \quad e^B_{k} = e^C_{k-1},
%& &\text{ if } \alpha_k = A, \\
%e^C_{k} = e^A_{k-1}, \quad e^A_{k} = e^C_{k-1}, 
%& &\text{ if } \alpha_k = B.
%\end{align*} 

\end{definition}

\begin{definition}[Definition of edge-index graphs $\F^A_k, \F^B_k$] \label{def:F}
To an $\alpha_k$-admissible sequence $\mathbf{i}_k$, we associate the edge-indexed graphs $\F^A_k, \F^B_k$ defined as follows.
\begin{enumerate}
\item
For $k=-1$, $\F^A_{-1}, \F^B_{-1}$ are both the graph with one vertex and one loop of index $d$.
\begin{center}
\begin{tikzpicture}[every loop/.style={}]
  \tikzstyle{every node}=[inner sep=-1pt]
  \node (1) at (0,0) {$\circ$} node [below=15pt] at (0,0) {$\F^A_{-1}$} ;
  \path[-] (1) edge [loop left] node [above=8pt,right=2pt] {$d$} (1);
\end{tikzpicture} 
\qquad
\begin{tikzpicture}[every loop/.style={}]
  \tikzstyle{every node}=[inner sep=-1pt]
  \node (1) at (0,0) {$\bullet$} node [below=15pt] at (0,0) {$\F^B_{-1}$} ;
  \path[-] (1) edge [loop left] node [above=8pt,right=2pt] {$d$} (1);
\end{tikzpicture} 
\end{center}

\item For $k=0, \cdots, K-1$, define
$$\F^A_{k} = \F^A_{k-1}, \qquad  
\F^B_{k} = \F^A_{k-1} \overset{V,V'} {\underset{\mathbf{i}_{k}}  {\bowtie} } \F^B_{k-1}.$$
Here, $V$ is the unique vertex of $\F^A_{k-1}$. The vertex $V'$ is the unique vertex of $\F^B_{k-1}$ coming from $\F^A_{k-2}$ for $k \geq 1$ and the unique vertex of $\F^B_{-1}$ for $k=0$.

\begin{center}
\begin{tikzpicture}[every loop/.style={}]
  \tikzstyle{every node}=[inner sep=1pt]
  \node (1) at (0,0) {$\circ$} node [below=5pt,right=5pt] at (0,0) {$V$} ;
  \node (1) at (0,0) {} node [below=15pt] at (0,0) {$\F^A_{k}$} ;
  \path[-] (1) edge [loop left] node [above=8pt,right=2pt] {$d$} (1);
\end{tikzpicture} 
\qquad
\begin{tikzpicture}[every loop/.style={}]
  \tikzstyle{every node}=[inner sep=-1pt]
  \node (1) at (0,0) {$\circ$} node [below=5pt] at (0,0) {$V'$} ;
  \node (2) at (1,0) {$\circ$};
  \node (3) at (2,0) {\, $\cdots$} node [below = 15 pt] at (2,0) {$\F^B_{k}$};
  \node (4) at (3,0) {$\circ$};
  \node (5) at (4,0) {$\bullet$};
  \path[-] 
	(2) edge node [above=4pt] {$\ i_{k-1}$}  (3)
	(3) edge node [above=4pt] {$\  j_1$}  (4)
	(4) edge node [above=4pt] {$i_0 \  j_0$} (5)
	(1) edge node [above=4pt] {$i_k \   j_k$} (2);
 \path[-] 
        (2) edge [loop above] (2)
        (4) edge [loop above] (4)
        (5) edge [loop right] node [above=4pt] {\ $d-j_0$} (5)
        (1) edge [loop left] node [above=4pt] {$d-i_k$ \ } (1);
\end{tikzpicture} 
\end{center}

\item For $k=K$, define
\begin{align*}
\F^A_{K} &= \F^A_{K-1} \overset{V,V'}  {\underset{i_K,j_K}  {\bowtie} } \F^B_{K-1}, \qquad
\F^B_{K}  = \F^A_{K-1} \overset{V,V'}  {\underset{i'_K,j_K}  {\bowtie} } \F^B_{K-1},&   & \text{ if } \mathbf{i}_K = (i_K, i'_K, j_K) ,  \\
\F^A_{K} &= \F^A_{K-1} \overset{V,V'}  {\underset{i_K,j_K}  {\bowtie} } \F^B_{K-1}, \qquad
\F^B_{K} = \F^B_{K-1},&  & \text{ if } \mathbf{i}_K = (i_K,j_K), 
\end{align*}
where $V, V'$ are as in part (2). 
By the \emph{common end vertices of $\F^A_K, \F^B_K$}, we mean the images of the end vertex of $\F^B_{K-1}$ different from $V'$ in $\F^A_K$, $\F^B_K$, respectively.

\item For $k > K$, define
$$
\F^{\alpha_k}_{k} = \F^{A}_{k-1}  \overset{V,V'}{\underset{\mathbf{i}_k}  {\bowtie} }\F^{B}_{k-1}, \qquad \F^{\overline \alpha_k}_{k} = \F^{\overline \alpha_k}_{k-1} $$ 
where $V, V'$ are the common end vertices of $\F^{A}_{k-1}, \F^{B}_{k-1}$.
By the \emph{common end vertices of $\F^A_k$, $\F^B_k$}, we mean the images of the non-common end vertex of $\F^{\overline \alpha_k}_{k-1}$ in $\F^A_k, \F^B_k$, respectively.
\end{enumerate}
\end{definition}

The graphs $\F^A_k$, $\F^B_k$ have canonical colorings $\varphi^A_k$, $\varphi^B_k$, resp.: the color of each vertex $y$ is $a$ or $b$, according to whether $y$ comes from $\F^A_{-1}$ or $\F^B_{-1}$, respectively.

\begin{remark}
The indices in the definition above can be interpreted as follows.
%\begin{enumerate} 
%\item In part (3), according to our definition, if $i_K=0$, then $\F^B_K=\F^B_{K-1}$, thus this definition corresponds to both of parts (a) and (b) of (2) in Theorem~\ref{thm:4}.
%\item 
For $k >K$, $d-\mathfrak{i}_k^C$ is the index of the loop at the common end vertex and $d-\mathfrak{i}_k^A, d-\mathfrak{i}_k^B$ are the indices of the loops at the non-common end vertices of $\F^A_{k}$, $\F^B_{k}$, respectively.
%\end{enumerate}
\end{remark}

\subsection{$\Psi$ is a direct limit}\label{sec4.2}
For a given sequence $\alpha_k$, a sequence $\beta_k \in \{ A, B \}$ is called \emph{$\alpha_k$-admissible} if  
\begin{equation}\label{eqn:4.2}
\beta_k = \alpha_k \ \textrm{or} \ \beta_{k-1}.
\end{equation} 
For an $\alpha_k$-admissible $\beta_k$,
we have a natural inclusions $(\F^{\beta_k}_k, \varphi^{\beta_k}_k) \hookrightarrow (\F^{\beta_{k+1}}_{k+1}, \varphi^{\beta_{k+1}}_{k+1})$ of colored graphs.
Consider the colored graph $(Y, \varphi)$ which is the direct limit determined by such inclusions:
$$ (Y, \varphi)=  \varinjlim \, (\F^{\beta_k}_k, \varphi^{\beta_k}_k).$$
The coloring $\varphi$ is the color of each vertex $y$ is $a$ or $b$, according to whether $y$ comes from $\F^A_{-1}$ or $\F^B_{-1}$, respectively. 
By construction, the universal cover of the edge-indexed graph $Y$ is a $d$-regular tree $T$, thus $(Y, \varphi)$ defines a coloring on the tree $T$, which is denoted by $(T, \varphi)$.
Similarly, the edge-indexed graph $\F^A_k, \F^B_k$ colored by $\varphi^A_k, \varphi^B_k$ have a $d$-regular colored tree as their universal cover, which we denote by $(T, \varphi^A_k), (T, \varphi^B_k)$, respectively. 

Let 
$n_k = |V\F^{\alpha_{k}}_{k}|-2$. 
Fix a vertex $t \in V \F^A_k (t \in V \F^B_k)$. For $n \le n_k,$  denote by $\N_n(t)$ the class of $n$-balls with the center a lift of $t$ in $(T, \varphi^A_k)$ ($(T, \varphi^B_k)$, respectively).

\begin{lemma}\label{lem4.4}
%Let $\F^A_{k+1} = \F^A_{k} \bowtie \F^B_{k}$. 
\begin{enumerate} 
\item For any vertices $v, w$ in $\F^A_{k-1} \bowtie \F^B_{k-1}$, we have $[\N_{n_{k}}(v)] \ne [\N_{n_{k}}(w)]$.

\item Let $v$ be a vertex of $\F^A_{k-1} \bowtie \F^B_{k-1}$ corresponding to a vertex $v'$ of $\F^A_{k-1}$ or $\F^B_{k-1}$. Then $[\N_{n_{k}}(v)]=[\N_{n_{k}}(v')]$.
\end{enumerate}
\end{lemma}

\noindent \textit{Proof.}
Denote by $v'_1, \cdots, v'_n$ and $w'_1, \cdots, w'_m$ the vertices of $\F^{A}_{k-1}$ and $\F^{B}_{k-1}$ in the linear order so that $v'_1, w'_1$ are the common end vertices with index $i^C_{k-1}$ and $v'_n, w'_m$ are the end vertices with index $\i^{A}_{k-1}, \i^{B}_{k-1}$.

If $\F^{\cdot}_k$ is $(i,j)$-concatenation, then
there are $(n+m)$ vertices in $\F^{\cdot}_k$ denoted by $v_1, \cdots, v_n, w_1, \cdots, w_m$  corresponding to vertices $v'_1, \cdots, v'_n$ of $\F^{A}_{k-1}$ and $w'_1, \cdots, w'_m$ of $\F^{B}_{k-1}$, respectively. 
If $\F^{\cdot}_k$ is $(i)$-concatenation, then 
there are $(n+m-1)$ vertices in $\F^{\cdot}_k$ denoted by $v_0, v_2, \cdots, v_n, w_2, \cdots, w_m$ where $v_0$ is a vertex of $\F^{\cdot}_{k}$ corresponding to two vertices $v'_1$ of $\F^A_{k-1}$ and $w'_1$ of $\F^B_{k-1}$.  
%Then for all $i$, $[B_{n_{k}}(v'_i)]=[B_{n_{k}}(v_i)]$ and $[B_{n_{k}}(w'_i)]=[B_{n_{k}}(w_i)]$. 

We use induction on $k$ on part (1) and part (2) at the same time.
\noindent
\begin{itemize}
\item[(a)] For $k \le K-1$, recall that $\alpha_k = B$, $n_k = k$, $n =1$, $m= k+1$ and the figures of $\F^A_k$ and $\F^B_k$ are as follows:
\begin{center}
\begin{tikzpicture}[every loop/.style={}]
  \tikzstyle{every node}=[inner sep=-1pt]
  \node (1) at (0,0) {$\circ$} node [below=5pt,right=5pt] at (0,0) {$v'_1$} ;
  \node (1) at (0,0) {} node [below=15pt] at (0,0) {$\F^A_{k} = \F^A_{k-1}$} ;
  \path[-] (1) edge [loop left] node [above=8pt,right=2pt] {$d$} (1);
\end{tikzpicture} 
\qquad
\begin{tikzpicture}[every loop/.style={}]
  \tikzstyle{every node}=[inner sep=-1pt]
  \node (1) at (0,0) {$\circ$} node [below=5pt] at (0,0) {$v_{1}$} ;
  \node (2) at (1,0) {$\circ$} node [below=5pt] at (1,0) {$w_{1}$} ;
  \node (3) at (2,0) {\, $\cdots$} node [below = 15 pt] at (2,0) {$\F^B_{k}$};
  \node (4) at (3,0) {$\circ$} node [below=5pt] at (3,0) {$w_{k}$} ;
  \node (5) at (4,0) {$\bullet$} node [below=5pt] at (4,0) {$w_{k+1}$} ;
  \path[-] 
	(2) edge node [above=4pt] {$\ i_{k-1}$}  (3)
	(3) edge node [above=4pt] {$\  j_1$}  (4)
	(4) edge node [above=4pt] {$i_0 \  j_0$} (5)
	(1) edge node [above=4pt] {$i_k \   j_k$} (2);
 \path[-] 
        (2) edge [loop above] (2)
        (4) edge [loop above] (4)
        (5) edge [loop right] node [above=4pt] {\ $d-j_0$} (5)
        (1) edge [loop left] node [above=4pt] {$d-i_k$ \ } (1);
\end{tikzpicture} 
\qquad
\begin{tikzpicture}[every loop/.style={}]
  \tikzstyle{every node}=[inner sep=-1pt]
  \node (2) at (1,0) {$\circ$} node [below=5pt] at (1,0) {$w'_{1}$} ;
  \node (3) at (2,0) {\, $\cdots$} node [below = 15pt] at (2,0) {$\F^B_{k-1}$};
  \node (4) at (3,0) {$\circ$} node [below=5pt] at (3,0) {$w'_{k}$} ;
  \node (5) at (4,0) {$\bullet$} node [below=5pt] at (4,0) {$w'_{k+1}$} ;
  \path[-] 
	(2) edge node [above=4pt] {$\ \ i_{k-1}$}  (3)
	(3) edge node [above=4pt] {$\  j_1$}  (4)
	(4) edge node [above=4pt] {$i_0 \  j_0$} (5);
%	(1) edge node [above=4pt] {$i_k \   j_k$} (2);
 \path[-] 
%        (2) edge [loop above] (2)
        (4) edge [loop above] (4)
        (5) edge [loop right] node [above=4pt] {\ $d-j_0$} (5)
        (2) edge [loop left] node [above=4pt] {$d-i_{k-1}$\ \qquad } (2);
\end{tikzpicture} 
\end{center}
Obviously, there is exactly one class of $n_k$-balls in $(T, \phi^{\overline{\alpha_k}}_k)$, namely the ball colored by $a$ (colored white in the figure)
only. There are $n_k+2=k+2$ distinct classes of $n_k$-balls in $(T, \phi^{\alpha_k}_k)$, namely 
the balls  $[\N_{n_k}(v_1)]$ and $[\N_{n_k}(w_i)]$ for $i=1, \cdots, k+1$. They are distinct classes since the vertex in $\N_{n_k}(w_i)$ closest to the center $v_i$ and colored by $b$ (colored black in the figure) is of distance $n_k+2-i$ for $i=1, \cdots, k+1$ from the center.
For part (2), we have $[\N_{n_k}(v_1)]$ = $[\N_{n_k}(v'_1)]$ since all the vertices are colored by $a$ in both balls. 
Let $w'_i$ be the vertices of $\F^B_{k-1}$. Then for each $1 \le i \le n_k+1$ we have $[\N_{n_k}(w_i)]$ = $[\N_{n_k}(w'_{i})]$ since the number of neighboring $n_k -1$ colored balls are the same by the induction argument.

\item[(b)] If $k = K$ and $\mathbf{i}_K = (i_K, j_K)$, then $\F^A_k = \F^A_{k-1} \underset{i_K,j_K}{\bowtie} \F^B_{k-1}$.
An argument similar to (1) shows that $[\N_{n_{K}}(w_i)]=[\N_{n_{K}}(w'_i)]$ for $1 \le i \le m + K+1$ and $[\N_{n_{K}}(v_1)]=[\N_{n_{K}}(v'_1)]$.

\item[(c)] If $k = K$ and $\mathbf{i}_K = (i_K, i'_K, j_K)$, then 
$\F^A_k = \F^A_{k-1} \underset{i_K,j_K}{\bowtie} \F^B_{k-1}$, $\F^B_k = \F^A_{k-1}\underset{i'_K,j_K}{\bowtie} \F^B_{k-1}$.
Let us denote by $v^A_1$, $w^A_i$ and $v^B_1$, $w^B_i$ the vertices of $\F^A_K$ and $\F^B_K$ respectively.
Then $[\N_{n_{K}}(w^A_i)]=[\N_{n_{K}}(w^B_i)]=[\N_{n_{K}}(w'_i)]$ for all $1 \le i \le m = K+1$ and $[\N_{n_{K}}(v^A_1)]=[\N_{n_{K}}(v^B_1)]=[\N_{n_{K}}(v'_1)]$.

%Without loss of ambiguity, let $\alpha_{k}=A$. 

%Let $m_k= |\F^{\overline{\alpha_k}}_k|$, which is equal to  $n_{k+1}-n_k$ for $(i,j)$-concatenation and $n_{k+1}-n_k +1$ for $(i)$-concatenation.

\item[(d)] Now let us prove the case $k \ge K+1$ assuming induction hypothesis for up to $k-1$. 
Suppose $\alpha_{k} = A$. 
%The case $\alpha_{k+1} = B$ is similar.
%Denote by $v_1, \cdots, v_{n_k+2}$ the vertices of $\F^A_{k+1}$ in the linear order. 
\begin{center}
\begin{tikzpicture}[every loop/.style={}]
  \tikzstyle{every node}=[inner sep=-.5pt]
  \node (1) at (0,0) {$\bullet$} node [below=5pt] at (0,0) {$v'_1$} ;
  \node (2) at (1,0) {$\bullet$} node [below=5pt] at (1,0) {$v'_2$} ;
  \node (3) at (2,0) {\, $\cdots$} node [below = 18 pt] at (2,0) {$\F^A_{k-1}$};
  \node (4) at (3,0) {$\bullet$} node [below=5pt] at (3,0) {$v'_{n}$} ;
  \path[-] 
	(1) edge node [above=4pt] {}  (2)
	(2) edge node [above=4pt] {}  (3)
	(3) edge node [above=4pt] {}  (4);
%	(1) edge node [above=4pt] {} (2);
 \path[-] 
        (1) edge [loop above] (1)
        (2) edge [loop above] (2)
        (4) edge [loop above] (4);
\end{tikzpicture} 
\qquad
\begin{tikzpicture}[every loop/.style={}]
  \tikzstyle{every node}=[inner sep=-.5pt]
  \node (0) at (0,0) {$\bullet$} node [below=5pt] at (0,0) {$v_n$};
  \node (1) at (1,0) {\, $\cdots$} node [below=5pt] at (1,0) {} ;
  \node (2) at (2,0) {$\bullet$} node [below=5pt] at (2,0) {$v_2$};
  \node (3) at (3,0) {$\bullet$} node [below=5pt] at (3,0) {$v_1$};
  \node (4) at (4,0) {$\bullet$} node [below=5pt] at (4,0) {$w_1$};
  \node [below=18pt] at (2.5,0) {$\F^A_{k} = \F^A_{k-1} \bowtie \F^B_{k-1}$};
  \node (5) at (5,0) {\, $\cdots$} node [below=5pt] at (5,0) {} ;
  \node (6) at (6,0) {$\bullet$} node [below=5pt] at (6,0) {$w_m$} ;
  \path[-] 
	(0) edge node [above=4pt] {} (1)
	(1) edge node [above=4pt] {} (2)
	(2) edge node [above=4pt] {} (3)
	(3) edge node [above=4pt] {} (4)
	(4) edge node [above=4pt] {} (5)
	(5) edge node [above=4pt] {} (6);
 \path[-] 
        (0) edge [loop above] (0)
        (2) edge [loop above] (2)
        (3) edge [loop above] (3)
        (4) edge [loop above] (4)
        (6) edge [loop above] (6);
\end{tikzpicture} 
\qquad
\begin{tikzpicture}[every loop/.style={}]
  \tikzstyle{every node}=[inner sep=-.5pt]
  \node (1) at (0,0) {$\bullet$} node [below=5pt] at (0,0) {$w'_1$} ;
%  \node (2) at (1,0) {$\bullet$} node [below=5pt] at (1,0) {$w_2$} ;
  \node (3) at (1,0) {\, $\cdots$} node [below = 18 pt] at (1,0) {$\F^B_{k} = \F^B_{k-1}$};
  \node (4) at (2,0) {$\bullet$} node [below=5pt] at (2,0) {$w'_{m}$} ;
  \path[-] 
	(1) edge node [above=4pt] {}  (3)
	(3) edge node [above=4pt] {}  (4);
%	(1) edge node [above=4pt] {} (2);
 \path[-] 
        (1) edge [loop above] (1)
%        (2) edge [loop above] (2)
        (4) edge [loop above] (4);
\end{tikzpicture} 
\end{center}
%By induction hypothesis on part (1) of step $k$, the class of $n_k$-balls $[\N_{n_k}(v'_i)]$ ($i=1, \cdots, n_k+2$) arounds vertices of $\F^A_k$ are all distinct classes in $(T, \phi^A_k)$, thus in $(T, \phi^A_{k+1})$ as well.
%By induction hypothesis on part (2) of step $k$, for the vertex $w'_i \;\; (i=1, \cdots, m_k)$ of $\F^B_k$, the $n_k$-ball around $w_i$ satisfies $[\N_{n_k}(w'_i)]=[\N_{n_k}(v'_i)]$ in $(T, \phi^B_k)$, thus $[\N_{n_k}(w_i)]=[\N_{n_k}(v_i)]$ in $(T, \phi^A_{k+1})$ as well.
If $\alpha_{k-1} = A$, then $\F^A_{k-1} = \F^A_{k-2} \bowtie \F^B_{k-2}$ and $n \ge m$. 
%Suppose that $n >0$.
Moreover the vertices $v'_1, \dots v'_m$ corresponds to $w'_1, \dots, w'_m$ by the definition of the common end vertices. 
By the induction argument, 
$[\N_{n_{k-1}}(v'_i)]=[\N_{n_{k-1}}(w'_i)]$ for all $1 \le i \le m$. 
Since $[\N_{n_{k-1}}(v'_i)] \ne [\N_{n_{k-1}}(v'_j)]$ for $i \ne j$ and $[\N_{n_{k-1}}(w'_i)] \ne [\N_{n_{k-1}}(w'_j)]$ for $i \ne j$, to show part (1) it remains to show that 
$[\N_{n_{k}}(v'_i)] \neq [\N_{n_{k}}(w'_i)]$ for all $1 \le i \le m$ for $(i,j)$-concatenation and for all $2 \le i \le m$ for $(i)$-concatenation.

Note that $d(v'_i, v'_{m+1}) \leq m = n_{k}-n_{k-1}$ for $1 \leq i \leq m$ in the case of $(i,j)$-concatenation. 
Similarly, $d(v'_i, v'_{m+1}) \leq m -1 = n_{k}-n_{k-1}$, and for $2 \leq i \leq m$ in the case of $(i)$-concatenation.
It follows that
$B_{n_{k}}(v'_i)$ contains $B_{n_{k-1}}(v'_{m+1})$. 
In contrast, $B_{n_{k}}(w'_i)$ does not contain $B_{n_{k-1}}(v'_{m+1})$ since all $n_{k-1}$-balls in $\F^B_{k-1}$ are of the class $[B_{n_{k-1}}(w'_i)] = [B_{n_{k-1}}(v'_i)]$, $1 \le i \le m$.
Therefore, $[\N_{n_{k}}(w'_i)] \neq [\N_{n_{k}}(v'_i)]$ in $(T, \phi^A_{k})$, for $i \geq 1$ for $(i,j)$-concatenation and for $i \geq 2$ in the case of $(i)$-concatenation. 
Thus part (1) for $k$ follows.

If $n=m$, then $k = K+1$ and $n = m = K+2$.
In this case, $\mathbf{i}_K = (i_K, i'_K, j_K)$, i.e.,
$\F^A_K = \F^A_{K-1} \underset{i_K,j_K}{\bowtie} \F^B_{K-1}$, $\F^B_K = \F^A_{K-1}\underset{i'_K,j_K}{\bowtie} \F^B_{K-1}$.
Note that $[\N_K (v'_i)] = [\N_K (w'_i)]$ for all $1 \le i \le n=m$.
Since $i_K < i'_K$, the number $[\N_K (v'_{n})]$ adjacent to $v'_{n}$ and $w'_n$ are different. Thus $[\N_{K+1} (w'_{n})] \ne [\N_{K+1} (v'_{n})]$.
By the similar argument, we have $[\N_{n_{K+1}} (w'_{i})] \ne [\N_{n_{K+1}} (v'_{i})]$ for each $1 \le i \le n$.

For part (2), we consider $\N_{n_{k+1}}(v_i)$ as a colored $m = (n_{k+1}-n_k)$-ball by coloring of $\N_{n_k} (t)$ on each vertex. 
Then $v_i$ and $w_i$ are colored by the same color $[\N_{n_k} (v_i)] = [\N_{n_k} (w_i)]$.

Note that for $1 < i \leq m_k -1$, since they both come from the same $\F_{k-1}$, we have
$$i(w_i, w_j) = i(v_i, v_j)$$ for $j=i, i-1, i+1.$ For $i=1$, 
$$i(w_1, w_1) + i(w_1, v_1) =i(v_1, w_1) + i(v_1, v_1).$$
In other words, the edge-indexed graph colored by $n_k$-balls in $\F^A_{k+1}$ with vertices of distance $m_k$ from the vertex $w_i$ is isomorphic to the edge-indexed graph colored by $n_k$-balls in $\F^B_{k+1}$ with vertices of distance $m_k$ from the vertex $w'_i$.

Therefore $[\N_m (v_i)] = [\N_m (v'_i)]$ and $[\N_m (w_i)] = [\N_m (w'_i)]$. 
The proof for the case $\alpha_k = B$ is similar.

\end{itemize}

\qed

\begin{proposition}
Two admissible sequences $\beta_k, \beta'_k$ are eventually equal if and only if they have the same direct limit
$\varinjlim \, (\F^{\beta_k}_k, \varphi^{\beta_k}_k) =  \varinjlim \, (\F^{\beta'_k}_k, \varphi^{\beta'_k}_k)$.
\end{proposition}

\begin{proof}
%We first claim that  if $t$ is a vertex of $\varinjlim \, \F^{\beta_k}_k$ and $t$ is an end vertex of $\F^{\beta_k}_k$ for all large $k$, then
%$\beta_k = \alpha_k = A$ or $B$ for all large $k$.
%Since the common end vertices of $\F^A_k, \F^B_k$ become the joining vertices of the concatenation, $t$ is not common end vertex $\F^{\beta_k}_k$, which is followed by $\alpha_{k+1} = \beta_k = \beta_{k+1}$.
It is clear that if $\beta_k, \beta'_k$ are eventual equal, the direct limits are the same.
Suppose that $\beta_k \ne \beta'_k$ for infinitely many $k$.

Let $t, t'$ be the vertices of $\varinjlim \, \F^{\beta_k}_k,  \varinjlim \, \F^{\beta'_k}_k$.
Then there exist $k_1, k_2$ such that $t \in V\F^{\beta_k}_k$ for $k \ge k_1$ and $t' \in V\F^{\beta'_k}_k$ for $k \ge k_2$.
%If $t$ is an end vertex of $t \in V\F^{\beta_k}_k$ for all large $k$, then $\alpha_k = \beta_k$ is stabilized, say $A$.
%Thus, $\beta'_k = B$ for all large $k$ which is follows that  
Choose $k \ge \max( k_1, k_2)$ as $\beta_k \ne \beta'_k$.
Then $t \in V\F^{\beta_{k}}_{k}$, $t' \in V\F^{\beta'_{k}}_{k}$, thus $t$, $t'$ are different vertices in $V\F^{\alpha_{k+1}}_{k+1}$ unless they are the common ends of $\F^{\beta_{k}}_{k}$, $\F^{\beta'_{k}}_{k}$.
By Lemma~\ref{lem4.4} (1),  we have $[\N_{n_{k+1}} (t)] \ne [\N_{n_{k+1}} (t')]$.

If $t$, $t'$ are the common end vertices of $\F^{\beta_{k}}_{k}$, $\F^{\beta'_{k}}_{k}$, then they are the joining vertex of the concatenation.
Therefore, one of $t$, $t'$ is not end vertex of $\F^{\beta_{k+\ell}}_{k+\ell}$ or $\F^{\beta'_{k+\ell}}_{k+\ell}$ for all $\ell \ge 1$.
Choose another $k' > k$ such that $\beta_{k'} \ne \beta'_{k'}$ and apply the same argument.
\end{proof}

\begin{lemma}\label{lem4.6}
For any vertex $v$ in $\F^A_{k + \ell}$ or $\F^B_{k + \ell}$ with $\ell \ge 0$, 
there exists $v' \in V\F^{\alpha_k}_k$ such that $[ \N_{n_k} (v)] = [ \N_{n_k} (v')] $.
\end{lemma}

\begin{proof}
We may assume that $v$ is a vertex of $\F^A_{k + \ell}$.
If $\alpha_{k+i} = B$ for all $0 \le i \le \ell$, then $\F^A_{k + \ell} = \F^A_{k-1}$, thus $v$ is also a vertex of $\F^A_{k-1}$.
By Lemma~\ref{lem4.4} (2) there exists $v' \in V\F^{\alpha_k}_k$ such that $[ \N_{n_k} (v)] = [ \N_{n_k} (v')] $.

Let $m$ be the largest integer with $k \le m \le k+ \ell$ satisfying $\alpha_{m} = A$.
Then $\F^A_{k + \ell} = \F^A_{m}$,  thus $v$ is also a vertex of $\F^A_{m}$.
By Lemma~\ref{lem4.4} (2), there exists $v_{m-1} \in V\F^{\alpha_{m-1}}_{m-1}$ such that $[ \N_{n_m} (v)] = [ \N_{n_m} (v_{m-1})] $.
Inductively, we can choose $v_{i} \in V\F^{\alpha_{i}}_{i}$ such that $[ \N_{n_{i+1}} (v_{i+1})] = [ \N_{n_{i+1}} (v_{i})]$ for $k \le i \le m-1$.
Therefore, $[ \N_{n_{k+1}} (v_{k})] = [ \N_{n_{k+1}} (v_{k+1})] = \dots = [ \N_{n_{k+1}} (v_{m-1})] = [ \N_{n_{k+1}} (v)]$.
\end{proof}

\begin{proposition}\label{Prop4.7} 
Assume that $\beta_k = \alpha_k$ for infinitely many $k$.
Then the direct limit $(Y, \varphi)$ is a Sturmian coloring.
\end{proposition}

\begin{proof}
By Lemma~\ref{lem4.6} for $k$ such that $\beta_k=\alpha_k$, we have $\mathbf{B}_{n_k}(\varphi) \subset  \mathbf{B}_{n_k}(\varphi^{\alpha_k}_k)$.
On the other hand, $\mathbf{B}_{n_k}(\varphi) \supset  \mathbf{B}_{n_k}(\varphi^{\alpha_k}_k)$ is implied by Lemma~\ref{lem4.4} (2).
Using Lemma~\ref{lem4.4} (1), we deduce that 
$$|\mathbf{B}_{n_k}(\varphi)| = |\mathbf{B}_{n_k}(\varphi^{\alpha_k}_k)|  = |V\F^{\alpha_k}_k| = n_k +2.$$
Since there are infinitely many such $k$'s, $(Y,\phi)$ is Sturmian.
\end{proof}

We remark that if $\alpha_k$ is not stabilized i.e. there are infinitely many $k$ such that $\alpha_{k+1}\neq\alpha_k$, then any admissible sequence $\beta_k$ satisfies Proposition~\ref{Prop4.7}, thus the direct limit is a Sturmian coloring.

\subsection{Sturmian coloring and direct limits}
In this section we show that $\Psi \circ \Phi = \mathrm{Id}$. Let us first define $\Phi$. 
For a given Sturmian coloring $(T,\phi)$, we need to define $(\alpha_k, \mathbf{i}_k, [\beta_k])$. 
The sequence $\alpha_k$ and $\G^A_n, \G^B_n$ are defined in Section~\ref{section:wocycle}. 
They determine $\mathbf{i}_k$ as follows. 

\begin{definition}[Definition of $\Phi$]\label{def:4.8} 
Let $(X, \phi)$ be an acyclic Sturmian coloring.
\begin{enumerate}
\item
Define $\mathbf{i}_k$ by
$$\mathbf{i}_{k} = \begin{cases} 
(i,j) , &\text{for \; case\; (1)},\\
(i,i',j) , &\text{for \;  case\; (2)(a)},\\
(i,j) , &\text{for \;  case\; (2)(b)},\\
(i)\; \text{or} \;(i,j) , &\text{for \;  case\; (3)},
\end{cases} $$
where the cases are as in Theorem~\ref{thm:4}.

\item For a fixed vertex $t \in T$, define $\beta_k = \beta_k(t)$ as follows:
\begin{enumerate}
\item[(i)] if $[\B_{n_{k}}(t)] \in V\G^{\overline{\alpha_{k}}}_{n_{k}}$ then set $\beta_{k-1} =\overline {\alpha_{k}}$.
\item[(ii)] if $[\B_{n_{k}}(t)] \notin %V\G^{{\alpha_{k}}}_{n_{k}} -
V\G^{\overline{\alpha_{k}}}_{n_{k}}$ then set $\beta_{k-1} = \alpha_{k}$.
\end{enumerate}
\end{enumerate}

\end{definition}
One can easily check that $\mathbf{i}_k$ is $\alpha_k$-admissible.

%\begin{lemma}\label{lem:4new}
%For any $k \ge 0$, we have
%$$ [\B_{n_{k}}(t)] \in V\G^{\beta_{k}}_{n_{k}} \ \text{ if } \beta_{k} =\overline {\alpha_{k+1}} \quad \text{   and } \quad
%[\B_{n_{k}}(t)] \in V\G^{\beta_{k}}_{n_{k}} -V\G^{\overline{\beta_{k}}}_{n_{k}} \ \text{ if } \beta_{k} = \alpha_{k+1}.$$
%\end{lemma}
%
%\begin{proof}
%By Theorem~\ref{thm:4} (3),  
%$$\G^{{\alpha_{k}}}_{n_{k}} \cong \G^{A}_{n_{k-1}} \bowtie \G^{B}_{n_{k-1}}, \qquad \G^{\overline{\alpha_{k}}}_{n_{k}} \cong \G^{\overline{\alpha_{k}}}_{n_{k-1}},$$ 
%where the isomorphism is given by the restriction. 
%If $\beta_{k} =\overline {\alpha_{k+1}}$, then by Definition~\ref{def:4.8} (2) (i),
%$$[\B_{n_{k+1}}(t)] \in V\G^{\overline{\alpha_{k+1}}}_{n_{k+1}}, \quad  \text{ thus } \quad
%[\B_{n_{k}}(t)] \in V\G^{\overline{\alpha_{k+1}}}_{n_k}=  V\G^{\beta_{k}}_{n_k}.$$
%For $\beta_{k} = \alpha_{k+1}$, we have from Definition~\ref{def:4.8} (2) (ii),
%$$[\B_{n_{k+1}}(t)] \in V\G^{{\alpha_{k+1}}}_{n_{k+1}} -V\G^{\overline{\alpha_{k+1}}}_{n_{k+1}},$$
% which implies
%\begin{equation*}
%[\B_{n_{k}}(t)] \in V\G^{A}_{n_k} \cup V\G^{B}_{n_k},  \qquad 
%[\B_{n_{k}}(t)] \notin \G^{\overline{\alpha_{k+1}}}_{n_k}=  \G^{\overline{\beta_{k}}}_{n_k}. \qedhere
%\end{equation*}
%\end{proof}

\begin{lemma}\label{lem:4new}
For any $k \ge 0$, we have
$ [\B_{n_{k}}(t)] \in V\G^{\beta_{k}(t)}_{n_{k}}.$
\end{lemma}

\begin{proof}
By Theorem~\ref{thm:4} (3), 
%$$\G^{{\alpha_{k+1}}}_{n_{k+1}} \cong \G^{A}_{n_{k}} \bowtie \G^{B}_{n_{k}}, \qquad 
$\G^{\overline{\alpha_{k+1}}}_{n_{k+1}} \cong \G^{\overline{\alpha_{k+1}}}_{n_{k}},$ 
where the isomorphism is given by the restriction and extension, i.e. 
$[\B_{n_{k+1}}(t)] \in V\G^{\overline{\alpha_{k+1}}}_{n_{k+1}}$
if and only if $[\B_{n_{k}}(t)] \in V\G^{\overline{\alpha_{k+1}}}_{n_{k}}.$
%Therefore, the proof is immediate by Definition~\ref{def:4.8} (2)

If $\beta_{k} \ne \alpha_{k+1}$, then by Definition~\ref{def:4.8} (2) (i),
$[\B_{n_{k+1}}(t)] \in V\G^{\overline{\alpha_{k+1}}}_{n_{k+1}},$ thus  
$[\B_{n_{k}}(t)] \in V\G^{\beta_{k}}_{n_k}.$

If $\beta_{k} = \alpha_{k+1}$, then by Definition~\ref{def:4.8} (2) (ii),
$[\B_{n_{k+1}}(t)] \notin V\G^{\overline{\alpha_{k+1}}}_{n_{k+1}},$ thus $[\B_{n_{k}}(t)] \notin V\G^{\overline{\alpha_{k+1}}}_{n_k},$
which implies that 
\begin{equation*}
[\B_{n_{k}}(t)] \in V\G^{\alpha_{k+1}}_{n_k} = V\G^{\beta_{k}}_{n_k}. 
\qedhere
\end{equation*}
\end{proof}

\begin{lemma}\label{lem:4.5} 
The sequence $\beta_k$ defined above is $\alpha_k$-admissible, i.e. it satisfies $\beta_k = \alpha_k$ or $\beta_{k-1}$.
\end{lemma} 

\begin{proof} 
Suppose that $\beta_k \neq \beta_{k-1} = \alpha_k$. 
By Definition~\ref{def:4.8} (2) (ii), we have 
\begin{equation*}\label{eq:4.2}
[\B_{n_{k}}(t)] \notin V\G^{\overline{\alpha_{k}}}_{n_{k}}. 
\end{equation*}
However, by Lemma~\ref{lem:4new}, we have
$
[\B_{n_{k}}(t)] \in V\G^{\beta_{k}}_{n_{k}} = V\G^{\overline{\alpha_{k}}}_{n_{k}},
$
which is a contradiction.
\end{proof}

\begin{remark} 
If $\phi$ is unbounded, then  $\beta_k= \alpha_k$ for arbitrary large $k$. If $\phi$ is bounded, then there exists $k_0$ such that $\beta_k = \alpha_k$ for all $k \geq k_0$.
\end{remark}
%Note that if the vertex $t$ is not an end vertex of $\G^{A}_{n_{k'}}, \G^{B}_{n_{k'}}$, then $\beta_k$ is uniquely determined for $k \geq k'$. If $t$ is an end vertex of both $\G^{A}_{n_k-1}\cong\G^{A}_{n_{k-1}}$ and $\G^B_{n_k-1}\cong\G^{B}_{n_{k-1}}$, then it belongs to exactly one of $\G^A_{n_k}$ and $\G^B_{n_k}$, since it is the common end vertex $C_{n_k-1}$ which is the joining vertex for the level $n_k$ by Theorem~\ref{thm:4p}. 
%In this case, we choose $\beta_k = \beta_{k+1}$. 
%Thus the direct limit $\varinjlim \G^{\beta(t)}_{n_k}$ does not depend on the choice of $t$.

\begin{lemma}\label{lem:8}
Fix $t\in VT$ and let $\beta_k=\beta_k(t)$. Let $t'$ be the vertex adjacent to $t$. 
Then there exists $k_0$ such that $\beta_k(t') = \beta_k (t)$ for all $k \geq k_0$. 
 \end{lemma}

%$t' \in VX$, there exists $k_0$ such that for all $k \geq k_0$, $\N_{n_k}(t'') \in V\G^{\beta_k}_{n_k},$ for every $t'' \in VX$ between $t$ and $t'$ in the quotient graph $X$ associated to $(T, \phi)$.
% \end{lemma}
 
\begin{proof} 
It is enough to show that if there is some $k > K$ such that 
$$[\B_{n_{k}}(t)] \in V\G^{\overline{\alpha_{k}}}_{n_{k}} \quad \text{ and } \quad  [\B_{n_{k}}(t')] \notin V\G^{\overline{\alpha_{k}}}_{n_{k}}, $$ 
then for $\beta_l = \beta_l(t')$ for all $l >k$.

Let $k$ be such an integer. Without loss of generality, let us assume that $\alpha_k = A$.  Then by the definition of $\G_n^A$, $\G_n^B$,
we have $[\B_{n_{k}}(t)] = S_{n_{k}}$ and $[\B_{n_{k}}(t')] = A_{n_k}$.
Since $A_{n_k+1}$ is adjacent to $A_{n_k}$ and not $B_{n_k}$, we have  $[\B_{n_{k}+1}(t)] = A_{n_{k}+1}$.
Let $t'' \in VT$ be an adjacent vertex of $t$ such that $[\B_{n_{k}}(t'')] = B_{n_{k}} = C_{n_{k}}$.
From $|V\G^B_{n_k}| > 1$, we check $B_{n_{k}} \ne S_{n_k}$.
Therefore $t$ has two adjacent vertices $t', t''$ with distinct $n_k$-balls. 

It follows that for any $n \ge n_k +1$, $[\B_n(t)]$ always adjacent to two different $n$-balls in both $\G^A_n$ and $\G^B_n$, whenever $[\B_n(t)]$ is a vertex of them.
We conclude that that $[\B_n(t)]$ is not an end vertex of $\G_n^A$ nor $\G_n^B$ for $n \ge n_k +1$. In other words, $\beta_l(t) = \beta_l(t')$ for all $l > k$.
%We first claim that there exists $k_0$ with $[\B_{n_k} (t')] \in V \G_{n_k}^{\beta_k}$ for all $k \geq k_0$. 
%Suppose that there is no such $k_0$, i.e., $[\B_{n_k}(t')] \notin V\G_{n_k}^{\beta_k}$ for arbitrary large $k=k_i \to \infty$. 
\end{proof}
 
Repeating the above lemma for vertices between any pair of vertices $t,t'$, the sequence $\beta_k(t')$ is eventually equal to $\beta_k(t)$, thus $\Phi$ is well-defined.

\begin{remark} 
The sequence $(\alpha_k)$ has a role corresponding to the slope for the irrational rotation associated to a Sturmian word (or the ratio of alphabets $a$ and $b$ appearing in the word). The freedom coming from the intercept (starting point of the irrational rotation) of Sturmian words are replaced by a sequence $\beta_k$ satisfying Lemma~\ref{lem:4.5}.
\end{remark}

Define $(Y, \varphi)$ to be the direct limit $\varinjlim \F^{\beta_k}_k$ in Section~\ref{sec4.2} with $\F^A_k=\G^A_{n_k}$, $\F^B_k=\G^B_{n_k}$.

\begin{theorem}\label{thm:6} 
We have $\Psi \circ \Phi =1$, i.e. for given Sturmian coloring, let $(X, \phi)$ be the quotient graph defined in the introduction. 
Then the direct limit $ (Y, \varphi)$ is equal to $(X, \phi)$.

%We also have $\Phi \circ \Psi = Id$, i.e if $\mathbf{i}_k$ is an $a_k$-admissible sequence and $\beta_k$ is an $a_k$-admissible sequence, then the sequence induced from the direct limit $\varinjlim \F^{\beta_k}_k$ %$(Y, \varphi)$ 
%is $(\alpha_k, \mathbf{i}_k, [\beta_k])$.
\end{theorem}

\begin{proof} 
By Lemmas \ref{lem:4new} and \ref{lem:8}, for any vertex $t \in VX$, there exists $k(t)$ such that $\N_{n_k}(t) \in V\G_{n_k}^{\beta_k}$ for $k \geq k(t)$. Moreover, for $k > k(t)$, the $n_k$-ball centered at $t$ is not an end vertex of $\G^{\beta_k}_{n_k}$, thus the injection preserves the adjacencent $n_k$-balls. Thus there is a injection from $VX$ to the set of vertices of $\varinjlim G^{\beta_k}_{n_k}(t)$
%where $\varinjlim \N_{n_k}(t)$ is the point in the direct limit corresponding to the sequence $\N_{n_k}(t)$.
preserving the edge indices. This map is clearly surjective, since any vertex of $\varinjlim G^{\beta_k}_{n_k}(t)$ determines a sequence of equivalence classes of $n$-balls, which determines a vertex of $X$.

%\kdh{
%As for the second assertion, 
%by Definition~\ref{def:4.8} the direct limit coloring $\varinjlim \F^{\beta_k}_k$ gives $\alpha_k, \mathbf{i}_k$. The eventual equality of $\beta_k$ is from Lemma~\ref{lem:8}. I am not sure why 4.12 gives us eventual equality.
%}
\end{proof}

\section{Sturmian colorings of bounded type}\label{sec:5}

\subsection{Cyclic Sturmian colorings}\label{sec:4}

In this subsection, we investigate cyclic Sturmian colorings. The main result in this section is Proposition~\ref{prop:circular}. 

\begin{lemma}\label{lem:cyclic} Suppose that $\G_n$ has a cycle. Then
\begin{enumerate}
\item the special ball $S_n$ is in the cycle.

\item the special ball $S_n$ is adjacent to $A_n, B_n, C_n$ only, apart from $S_n$ itself.
\end{enumerate}
\end{lemma}

\begin{proof}
(1) Suppose that $\G_n$ has a vertex $v$ which is not in the cycle.
Since $\G_n$ is connected, there exists a vertex $w$ in the cycle connected to $v$.
Since $S_n$ is the unique vertex of $\G_n$ with degree 3 by Lemma~\ref{newlem} (2), it follows that $w=S_n$ and $S_n$ is in the cycle. 

(2) Suppose that there exists $D \neq S_n$ distinct from $A_n, B_n, C_n$ and adjacent to $S_n$. 
By Lemma~\ref{newlem} (1), $S_n \neq C_n$ and by Lemma~\ref{lem:merged} (1),  $D$ is of degree 1.
Since $S_n$ is adjacent to $C_n$ by Lemma~\ref{2.11KL} (4),
 the classes of distance 1 from $S_n$ in $A_{n+1}, B_{n+1}$ are $\{C_n , S_n, D\}$ (with possibly distinct indices). 
Using Lemma~\ref{2.11KL} (4) again, $C_n$ and $D$ belong to the cycle, contradicting that $D$ is of degree 1.
\end{proof}

\begin{lemma}\label{lem:graphs} 
If $\G_n$ has a cycle not containing $C_n$, then $\G_{n+\ell}$ has a cycle containing $C_{n+\ell}$ for some $\ell \ge 1$.
\end{lemma}

\begin{proof} 
By Lemma~\ref{lem:cyclic} (2), $S_n$ is adjacent to $A_n, B_n, C_n, S_n$ only.
Since $S_n$ is the unique vertex of degree at least 3 and it is in the cycle, and since $C_n$ is adjacent to $S_n$, 
the graph $\G_n$ is the union of a cycle containing $S_n, A_n, B_n$ and a line segment containing $S_n, C_n$. 

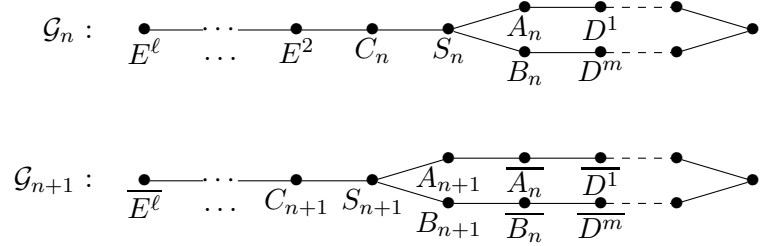
\medskip

\begin{center}
\begin{tikzpicture}[every loop/.style={}]
  \tikzstyle{every node}=[inner sep=-1pt]
  
  \node (-1) at (-1,2) {$\bullet$} node [left=20pt] at (-1,2) {$\G_n$ :} node [below=4pt] at (-1,2) {$E^\ell$};
  \node (0) at (0,2) {$\cdots$} node [below=10pt] at (0,2) {$\ldots$};
  \node (1) at (1,2) {$\bullet$} node [below=4pt] at (1,2) {$E^2$};
  \node (2) at (2,2) {$\bullet$} node [below=4pt] at (2,2) {$C_n$};
  \node (3) at (3,2) {$\bullet$} node [below=4pt] at (3,2) {$S_n$};
  \node (4) at (4,2.3) {$\bullet$} node [below=4pt] at (4,2.3) {$A_n$};
  \node (5) at (4,1.7) {$\bullet$} node [below=4pt] at (4,1.7) {$B_n$};
  \node (6) at (5,2.3) {$\bullet$} node [below=4pt] at (5,2.3) {$D^1$};
  \node (7) at (5,1.7) {$\bullet$} node [below=4pt] at (5,1.7) {$D^m$};
  \node (8) at (6,2.3) {$\bullet$};
  \node (9) at (6,1.7) {$\bullet$};
  \node (10) at (7,2) {$\bullet$};

  \path[-] 
  	(-1)  edge (0)
	(0)  edge (1)
	(1)  edge (2)
	(2)  edge (3)
	(3)  edge (4)
	(3)  edge (5)
        (4)  edge (6)
        (5)  edge (7)
     (8)  edge (10)
     (9)  edge (10);
   \draw[dashed] (6) -- (8);
   \draw[dashed] (7) -- (9);
  
  \node (19) at (-1,0) {$\bullet$} node [left=20pt] at (-1,0) {$\G_{n+1}$ :} node [below=4pt] at (-1,0) {$\overline{E^\ell}$};
  \node (20) at (0,0) {$\cdots$} node [below=10pt] at (0,0) {$\ldots$};
  \node (21) at (1,0) {$\bullet$} node [below=4pt] at (1,0) {$C_{n+1}$};
  \node (22) at (2,0) {$\bullet$} node [below=4pt] at (2,0) {$S_{n+1}$};
  \node (23) at (3,.3) {$\bullet$} node [below=4pt] at (3,.3) {$A_{n+1}$};
  \node (24) at (3,-.3) {$\bullet$} node [below=4pt] at (3,-.3) {$B_{n+1}$};
  \node (25) at (4,.3) {$\bullet$} node [below=4pt] at (4,.3) {$\overline{A_n}$};
  \node (26) at (4,-.3) {$\bullet$} node [below=4pt] at (4,-.3) {$\overline{B_n}$};
  \node (27) at (5,.3) {$\bullet$} node [below=4pt] at (5,.3) {$\overline{D^1}$};
  \node (28) at (5,-.3) {$\bullet$} node [below=4pt] at (5,-.3) {$\overline{D^m}$};
  \node (29) at (6,.3) {$\bullet$};
  \node (30) at (6,-.3) {$\bullet$};
  \node (31) at (7,0) {$\bullet$};
%  \node (20) at (0,0) {$\bullet$} node [below=4pt] at (0,0) {$C_{n+1}$} node [left=20pt] at (0,0) {$\G_{n+1}$ :};
%  \node (21) at (1,0) {$\bullet$} node [below=4pt] at (1,0) {$S_{n+1}$};
% \node (22) at (2,.3) {$\bullet$} node [below=4pt] at (2,.3) {$A_{n+1}$};
%  \node (23) at (2, -.3) {$\bullet$} node [below=4pt] at (2,-.3) {$B_{n+1}$};
%  \node (24) at (3,.3) {$\bullet$};
%  \node (25) at (3,-.3) {$\bullet$};
%  \node (26) at (4,.3) {$\bullet$};
%  \node (27) at (4,-.3) {$\bullet$};
%  \node (28) at (5,.0) {$\bullet$};

  \path[-] 
	(19)  edge (20)
	(20)  edge (21)
	(21)  edge (22)
	(22)  edge (23)
	(22)  edge (24)
	(23)  edge (25)
	(24)  edge (26)
           (25)  edge (27)
           (26)  edge (28)
           (29) edge (31)
           (30) edge (31) ;
	\draw[dashed] (27) -- (29);
	\draw[dashed] (28) -- (30);
\end{tikzpicture}
\end{center}
Denote the cycle by $[S_n A_n D^1 D^2 \dots D^m B_n S_n]$ and the line segment by $[S_n C_n E^2 \cdots E^\ell]$. 
Since  $(n+1)$-ball extension of $S_n$ which is adjacent to $A_n$ (resp. $B_n$) is $A_{n+1}$ (resp. $B_{n+1}$), it follows that
$[S_{n+1} A_{n+1}\overline{A_n} \overline{D^1}  \dots \overline{D^m} \overline{B_n} S_{n+1}]$ is a cycle in $\G_{n+1}$.

If there are no vertices $E^i$, $C_n$ is the unique vertex not in the cycle in $\G_n$, thus all the vertices in $\G_{n+1}$ belong to the cycle.

If $\ell \geq 2$, then $\overline{E^2}$ is adjacent to every $\overline{C_n} = S_{n+1}$, which implies that $\overline{E^2} = C_{n+1}$.
Only the path $[\overline{E^2} \overline{E^3} \dots \overline{E^\ell}]$ in $\G_{n+1}$ is not in the cycle.
Repeating this procedure $j$ times, it follows that only the path $[\overline{E^{j+1}}^{j} \overline{E^{j+2}}^{j} \cdots \overline{E^\ell}^{j}]$ is not in the cycle. (Here we denote the extension of an $n$-ball $E$ to the $(n+j)$-ball by $\overline{E}^{j}$.)  Thus $\G_{n+\ell}$ contains all the vertices, in particular $C_{n+\ell}.$
\end{proof}

\begin{lemma}\label{lem:graphs2} 
If $\G_n$ has a cycle containing $C_n = S_n$, then $\G_{n+1}$ has a cycle containing $C_{n+1} \ne S_{n+1}$.
\end{lemma}

\begin{proof}
Suppose that $\G_{n}$ has a cycle which contains $C_{n} = S_{n}$. By definition of $C_n$ and $S_n$, it follows that either $A_{n+1}= S_{n+1}$ or $B_{n+1} = S_{n+1}$. 
As $C_n = S_n$, by Lemma~\ref{newlem} (1), the cycle in $\G_n$ is of the form 
$[S_n A_n C^1 \dots C^mB_n S_n]$.
Since we have by Lemma~\ref{2.11KL} (3), either $B_{n+1}$ is adjacent to $S_{n+1} = A_{n+1}$
or $A_{n+1}$ is adjacent to $S_{n+1} = B_{n+1}$.
Therefore, we have a cycle either
$[S_{n+1}  (= A_{n+1} ) \overline{A_n} \overline{C^1} \dots \overline{C^m} \overline{B_n} B_{n+1} S_{n+1}  (= A_{n+1} )]$ or $[S_{n+1}  (= B_{n+1} ) A_{n+1} \overline{A_n} \overline{C^1} \dots \overline{C^m} \overline{B_n} S_{n+1} (= B_{n+1} )]$.
Since $S_{n+1}$ is adjacent to $A_{n+1}$, $B_{n+1}$ or $C_{n+1}$, either $\overline{A_n}= C_{n+1}$ or $\overline{B_n}= C_{n+1}$. 
Therefore $\G_{n+1}$ has a cycle which contains $C_{n+1} \ne S_{n+1}$.
\end{proof}

\begin{proposition}\label{prop:circular} 
Suppose that for some $n$, $\G_n$ has a cycle. Let $m$ be the smallest integer such that $S_m \ne C_m$ are in the cycle of $\G_m$. 
\begin{enumerate}
\item
The coloring $\phi$ is of bounded type. 
\item 
Then the quotient graph $X_\phi$ of $(T,\phi)$ is isomorphic to $\varinjlim \G^B_n$ (resp. $X = \varinjlim \G^A_n$).
All the graphs $\G^A_{n}$ (or $\G^B_n$)  are isomorphic for $n \ge m$. 
\item By part (1), the quotient graph $X_\phi$ is an infinite ray by Theorem~\ref{thm:1.1}. As in Theorem~\ref{thm:1.1}, let $x_i$ be the $(i-1)$-th vertex from the left. The ball $\B_{m+i}(x_i)$ is special for all $i\geq0$.
\end{enumerate}
\end{proposition}

\begin{proof}
By Lemmas~\ref{lem:graphs} and \ref{lem:graphs2}, such $m$ exists.
%we may assume that $\G_n$ has a cycle containing $C_n \neq S_n$ for sufficiently large $n$.

Let $n \ge m$.
Denote the cycle in $\G_n$ by $[S_n C_n C^1 \dots C^k S_n]$,
where $C_n, C^1, \dots, C^k$ are distinct.  
It follows from Lemma~\ref{lem:cyclic} that $C^k$ is $A_n$ or $B_n$, say $C^k = A_n$. Thus the cycle is the underlying graph of $\G^A_n$.
Choose an $n$-ball colored by $S_n$ and adjacent to $A_n$. Its extension to $(n+1)$-ball is $A_{n+1}$ which is adjacent to $S_{n+1} = \overline{C_n}$ and $\overline{ A_n} = \overline{C^k}$.
Therefore, 
$[A_{n+1} S_{n+1}\overline{ C^1} \dots \overline{A_n} A_{n+1}]$
is the cycle in $\G_{n+1}$ and $\overline{C^1} = C_{n+1}$.  
By the the same argument for $\G^A_n$, we deduce that   
$\G^A_{n+1}$ is the edge indexed graph with underlying graph the cycle.   
By Lemma~\ref{lem:conf}, we have $\G^A_{n+1} \cong \G^A_n$.

By repeating this procedure, it follows that $\G^A_{m+i} \cong \G^A_m$ and $B_{m+i}$ is not a vertex of $\G^A_{m+i}$ for any $i \ge 1$.
Let $x_i$ be a vertex of $T$ such that $\B_{m+i} (x_i) = B_{m+i}$ for $i \ge 1$.
Then $\B_{m+j} (x_i) \in V\G^B_{m+j} - V\G^A_{m+j}$ for any $j \ge i$ which implies that $\B_{m+j} (x_i)$ is not special for any $j \ge i$. 
Since the type set at $x_i$ is finite, $\phi$ is of bounded type. 
It follows that each vertex $x$ in $T$ satisfies $\B_{m+i} (x) = B_{m+i}$ for some $i \ge 1$.
Moreover, the number of neighboring vertices is determined by the edge index of $\G^B_n$ and
the quotient graph $X = \varinjlim \G^B_n$.

We now prove that $\G^A_{m} \neq \G^A_{m-1}$.
If $\G_{m-1}$ does not have a cycle, then $\G^A_{m-1}$ is not isomorphic to $\G^A_{m}$ which is cyclic.
If $\G_{m-1}$ has a cycle which does not contain $C_{m-1}$, then  
by the proof of Lemma~\ref{lem:graphs}, the cycle in $\G^A_m$ is smaller than the cycle of $\G_{m-1}$.
If $\G_{m-1}$ has a cycle containing $S_{m-1} = C_{m-1}$, then by the proof of Lemma~\ref{lem:graphs2}, the cycle in $\G^A_m$ is smaller than the cycle of $\G_{m-1}$.

As for part (3), it is easy to see that the integer $m$ in Theorem~\ref{thm:1.1} and the $m$ here coincide.

\end{proof}

\begin{example}
An example of a cyclic coloring: \\

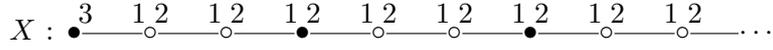
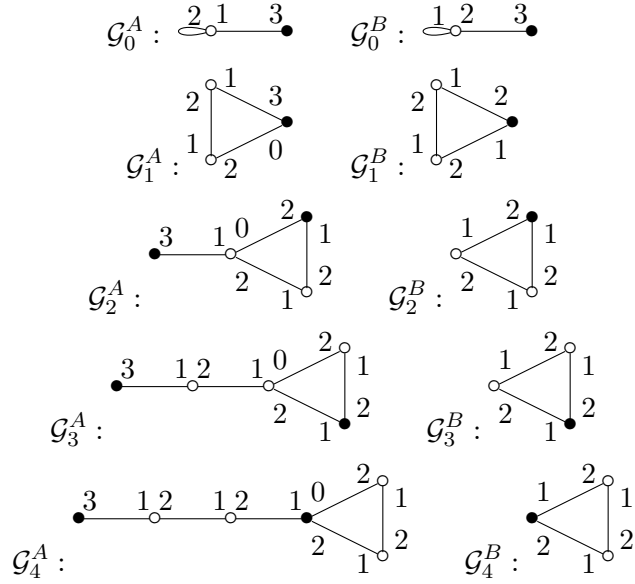
\begin{center}$X$ :
\begin{tikzpicture}[every loop/.style={}]
  \tikzstyle{every node}=[inner sep=0pt]
  \node (0) {$\bullet$};
  \node (1) at (1,0) {$\circ$};
  \node (2) at (2,0) {$\circ$};
  \node (3) at (3,0) {$\bullet$};
  \node (4) at (4,0) {$\circ$};
  \node (5) at (5,0) {$\circ$};
  \node (6) at (6,0) {$\bullet$};
  \node (7) at (7,0) {$\circ$};
  \node (8) at (8,0) {$\circ$};
  \node (9) at (9,0) {$\cdots$};

  \path[-] (0) edge node [above=4pt] {3 \quad 1} (1)
		 (1) edge node [above=4pt] {2 \quad 1} (2)
		 (2) edge node [above=4pt] {2 \quad 1} (3)
		 (3) edge node [above=4pt] {2 \quad 1} (4)
		 (4) edge node [above=4pt] {2 \quad 1} (5)
		 (5) edge node [above=4pt] {2 \quad 1} (6)
		 (6) edge node [above=4pt] {2 \quad 1} (7)
		 (7) edge node [above=4pt] {2 \quad 1} (8)
		 (8) edge node [above=4pt] {2 \quad \ } (9);
\end{tikzpicture}
\end{center}

$$\G^A_0 :
\begin{tikzpicture}[every loop/.style={}]
  \tikzstyle{every node}=[inner sep=-1pt]
  \node (1) at (0,0) {$\circ$};
  \node (2) at (1,0) {$\bullet$};
  \path[-] 
	(1)  edge node [above=4pt] {1 \quad 3} (2);
  \path[-] (1) edge [loop left] node [above=6pt,right=4pt] {2} (1);
\end{tikzpicture}  \qquad
\G^B_0 :
\begin{tikzpicture}[every loop/.style={}]
  \tikzstyle{every node}=[inner sep=-1pt]
  \node (1) at (0,0) {$\circ$};
  \node (2) at (1,0) {$\bullet$};
  \path[-] 
	(1)  edge node [above=4pt] {2 \quad 3} (2);
  \path[-] (1) edge [loop left] node [above=6pt,right=4pt] {1} (1);
\end{tikzpicture}$$
$$\G^A_1 :
\begin{tikzpicture}[every loop/.style={}]
  \tikzstyle{every node}=[inner sep=-1pt]
  \node (1) at (0,.5) {$\circ$};
  \node (2) at (0,-.5) {$\circ$};
  \node (3) at (1,0) {$\bullet$};
  \path[-] 
	(1) edge node [above=8pt,left=5pt] {2} node [below=8pt,left=5pt] {1} (2)
	(1) edge node [above=10pt,left=5pt] {1} node [above=3pt,right=8pt] {3} (3)
	(2) edge node [below=10pt,left=5pt] {2} node [below=3pt,right=8pt] {0} (3);
\end{tikzpicture}  \qquad
\G^B_1 :
\begin{tikzpicture}[every loop/.style={}]
  \tikzstyle{every node}=[inner sep=-1pt]
  \node (1) at (0,.5) {$\circ$};
  \node (2) at (0,-.5) {$\circ$};
  \node (3) at (1,0) {$\bullet$};
  \path[-] 
	(1) edge node [above=8pt,left=5pt] {2} node [below=8pt,left=5pt] {1} (2)
	(1) edge node [above=10pt,left=5pt] {1} node [above=3pt,right=8pt] {2} (3)
	(2) edge node [below=10pt,left=5pt] {2} node [below=3pt,right=8pt] {1} (3);
\end{tikzpicture}$$
$$\G^A_2 :
\begin{tikzpicture}[every loop/.style={}]
  \tikzstyle{every node}=[inner sep=-1pt]
  \node (0) at (-1,0) {$\bullet$};
  \node (1) at (0,0) {$\circ$};
  \node (2) at (1,.5) {$\bullet$};
  \node (3) at (1,-.5) {$\circ$};
  \path[-] 
	(0) edge node [above=4pt] {3 \quad 1} (1)
	(1) edge node [above=10pt,right=5pt] {2} node [above=3pt,left=8pt] {0} (2)
	(1) edge node [below=10pt,right=5pt] {1} node [below=3pt,left=8pt] {2} (3)
	(2) edge node [above=8pt,right=5pt] {1} node [below=8pt,right=5pt] {2} (3);
\end{tikzpicture}  \qquad
\G^B_2 :
\begin{tikzpicture}[every loop/.style={}]
  \tikzstyle{every node}=[inner sep=-1pt]
  \node (1) at (0,0) {$\circ$};
  \node (2) at (1,.5) {$\bullet$};
  \node (3) at (1,-.5) {$\circ$};
  \path[-] 
	(1) edge node [above=10pt,right=5pt] {2} node [above=3pt,left=8pt] {1} (2)
	(1) edge node [below=10pt,right=5pt] {1} node [below=3pt,left=8pt] {2} (3)
	(2) edge node [above=8pt,right=5pt] {1} node [below=8pt,right=5pt] {2} (3);
\end{tikzpicture}$$
$$\G^A_3 :
\begin{tikzpicture}[every loop/.style={}]
  \tikzstyle{every node}=[inner sep=-1pt]
  \node (-1) at (-2,0) {$\bullet$};
  \node (0) at (-1,0) {$\circ$};
  \node (1) at (0,0) {$\circ$};
  \node (2) at (1,.5) {$\circ$};
  \node (3) at (1,-.5) {$\bullet$};
  \path[-] 
	(-1) edge node [above=4pt] {3 \quad 1} (0)
	(0) edge node [above=4pt] {2 \quad 1} (1)
	(1) edge node [above=10pt,right=5pt] {2} node [above=3pt,left=8pt] {0} (2)
	(1) edge node [below=10pt,right=5pt] {1} node [below=3pt,left=8pt] {2} (3)
	(2) edge node [above=8pt,right=5pt] {1} node [below=8pt,right=5pt] {2} (3);
\end{tikzpicture}  \qquad
\G^B_3 :
\begin{tikzpicture}[every loop/.style={}]
  \tikzstyle{every node}=[inner sep=-1pt]
  \node (1) at (0,0) {$\circ$};
  \node (2) at (1,.5) {$\circ$};
  \node (3) at (1,-.5) {$\bullet$};
  \path[-] 
	(1) edge node [above=10pt,right=5pt] {2} node [above=3pt,left=8pt] {1} (2)
	(1) edge node [below=10pt,right=5pt] {1} node [below=3pt,left=8pt] {2} (3)
	(2) edge node [above=8pt,right=5pt] {1} node [below=8pt,right=5pt] {2} (3);
\end{tikzpicture}$$
$$\G^A_4 :
\begin{tikzpicture}[every loop/.style={}]
  \tikzstyle{every node}=[inner sep=-1pt]
  \node (-2) at (-3,0) {$\bullet$};
  \node (-1) at (-2,0) {$\circ$};
  \node (0) at (-1,0) {$\circ$};
  \node (1) at (0,0) {$\bullet$};
  \node (2) at (1,.5) {$\circ$};
  \node (3) at (1,-.5) {$\circ$};
  \path[-] 
	(-2) edge node [above=4pt] {3 \quad 1} (-1)
	(-1) edge node [above=4pt] {2 \quad 1} (0)
	(0) edge node [above=4pt] {2 \quad 1} (1)
	(1) edge node [above=10pt,right=5pt] {2} node [above=3pt,left=8pt] {0} (2)
	(1) edge node [below=10pt,right=5pt] {1} node [below=3pt,left=8pt] {2} (3)
	(2) edge node [above=8pt,right=5pt] {1} node [below=8pt,right=5pt] {2} (3);
\end{tikzpicture}  \qquad
\G^B_4 :
\begin{tikzpicture}[every loop/.style={}]
  \tikzstyle{every node}=[inner sep=-1pt]
  \node (1) at (0,0) {$\bullet$};
  \node (2) at (1,.5) {$\circ$};
  \node (3) at (1,-.5) {$\circ$};
  \path[-] 
	(1) edge node [above=10pt,right=5pt] {2} node [above=3pt,left=8pt] {1} (2)
	(1) edge node [below=10pt,right=5pt] {1} node [below=3pt,left=8pt] {2} (3)
	(2) edge node [above=8pt,right=5pt] {1} node [below=8pt,right=5pt] {2} (3);
\end{tikzpicture}$$

\end{example}

\begin{example}
An example with an $n$-ball $C$ adjacent to $S_n$ and which is not one of $S_n$, $A_n$, $B_n$, $C_n$. 
Note that this case can happen only for bounded type Sturmian colorings. \\
 
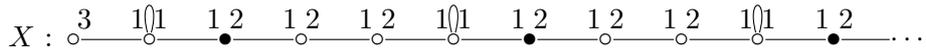
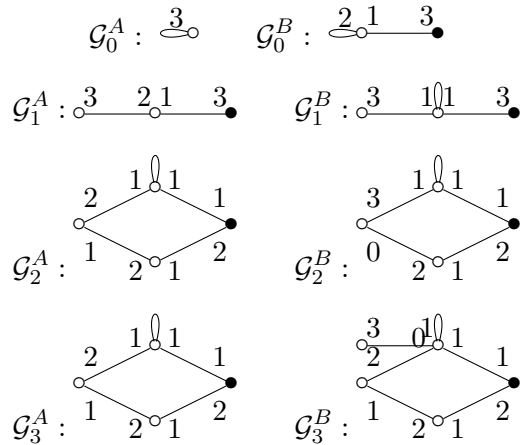
\begin{center}$X$ : 
\begin{tikzpicture}[every loop/.style={}]
  \tikzstyle{every node}=[inner sep=0pt]
  \node (0) {$\circ$};
  \node (1) at (1,0) {$\circ$};
  \node (2) at (2,0) {$\bullet$};
  \node (3) at (3,0) {$\circ$};
  \node (4) at (4,0) {$\circ$};
  \node (5) at (5,0) {$\circ$};
  \node (6) at (6,0) {$\bullet$};
  \node (7) at (7,0) {$\circ$};
  \node (8) at (8,0) {$\circ$};
  \node (9) at (9,0) {$\circ$};
  \node (10) at (10,0) {$\bullet$};
  \node (11) at (11,0) {$\cdots$};

  \path[-] (0) edge node [above=4pt] {3 \quad 1} (1)
		 (1) edge node [above=4pt] {1 \quad 1} (2)
		 (2) edge node [above=4pt] {2 \quad 1} (3)
		 (3) edge node [above=4pt] {2 \quad 1} (4)
		 (4) edge node [above=4pt] {2 \quad 1} (5)
		 (5) edge node [above=4pt] {1 \quad 1} (6)
		 (6) edge node [above=4pt] {2 \quad 1} (7)
		 (7) edge node [above=4pt] {2 \quad 1} (8)
		 (8) edge node [above=4pt] {2 \quad 1} (9)
		 (9) edge node [above=4pt] {1 \quad 1} (10)
		 (10) edge node [above=4pt] {2 \quad \ } (11);
  \path[-] (1) edge [loop above] (1)
		 (5) edge [loop above] (5)
	 	 (9) edge [loop above] (9);
\end{tikzpicture}
\end{center}

$$\G^A_0 :
\begin{tikzpicture}[every loop/.style={}]
  \tikzstyle{every node}=[inner sep=-1pt]
  \node (1) at (0,0) {$\circ$};
  \path[-] (1) edge [loop left] node [above=6pt,right=4pt] {3} (1);
\end{tikzpicture}  \qquad
\G^B_0 :
\begin{tikzpicture}[every loop/.style={}]
  \tikzstyle{every node}=[inner sep=-1pt]
  \node (1) at (0,0) {$\circ$};
  \node (2) at (1,0) {$\bullet$};
  \path[-] 
	(1)  edge node [above=4pt] {1 \quad 3} (2);
  \path[-] (1) edge [loop left] node [above=6pt,right=4pt] {2} (1);
\end{tikzpicture}$$
$$\G^A_1 :
\begin{tikzpicture}[every loop/.style={}]
  \tikzstyle{every node}=[inner sep=-1pt]
  \node (0) at (0,0) {$\circ$};
  \node (1) at (1,0) {$\circ$};
  \node (2) at (2,0) {$\bullet$};
  \path[-] 
	(0)  edge node [above=4pt] {3 \quad 2} (1)
	(1)  edge node [above=4pt] {1 \quad 3} (2);
\end{tikzpicture}  \qquad
\G^B_1 :
\begin{tikzpicture}[every loop/.style={}]
  \tikzstyle{every node}=[inner sep=-1pt]
  \node (0) at (0,0) {$\circ$};
  \node (1) at (1,0) {$\circ$};
  \node (2) at (2,0) {$\bullet$};
  \path[-] 
	(1) edge [loop above] (1)
	(0)  edge node [above=4pt] {3 \quad 1} (1)
	(1)  edge node [above=4pt] {1 \quad 3} (2);
\end{tikzpicture}$$
$$\G^A_2 :
\begin{tikzpicture}[every loop/.style={}]
  \tikzstyle{every node}=[inner sep=-1pt]
  \node (0) at (0,0) {$\circ$};
  \node (1) at (1,.5) {$\circ$};
  \node (2) at (2,0) {$\bullet$};
  \node (3) at (1,-.5) {$\circ$};
  \path[-] 
	(1) edge [loop above] (1)
	(0) edge node [above=10pt,right=5pt] {1} node [above=3pt,left=8pt] {2} (1)
	(0) edge node [below=10pt,right=5pt] {2} node [below=3pt,left=8pt] {1} (3)
	(1) edge node [above=10pt,left=5pt] {1} node [above=3pt,right=8pt] {1} (2)
	(3) edge node [below=10pt,left=5pt] {1} node [below=3pt,right=8pt] {2} (2);
\end{tikzpicture}  \qquad
\G^B_2 :
\begin{tikzpicture}[every loop/.style={}]
  \tikzstyle{every node}=[inner sep=-1pt]
  \node (0) at (0,0) {$\circ$};
  \node (1) at (1,.5) {$\circ$};
  \node (2) at (2,0) {$\bullet$};
  \node (3) at (1,-.5) {$\circ$};
  \path[-] 
	(1) edge [loop above] (1)
	(0) edge node [above=10pt,right=5pt] {1} node [above=3pt,left=8pt] {3} (1)
	(0) edge node [below=10pt,right=5pt] {2} node [below=3pt,left=8pt] {0} (3)
	(1) edge node [above=10pt,left=5pt] {1} node [above=3pt,right=8pt] {1} (2)
	(3) edge node [below=10pt,left=5pt] {1} node [below=3pt,right=8pt] {2} (2);
\end{tikzpicture}$$
$$\G^A_3 :
\begin{tikzpicture}[every loop/.style={}]
  \tikzstyle{every node}=[inner sep=-1pt]
  \node (0) at (0,0) {$\circ$};
  \node (1) at (1,.5) {$\circ$};
  \node (2) at (2,0) {$\bullet$};
  \node (3) at (1,-.5) {$\circ$};
  \path[-] 
	(1) edge [loop above] (1)
	(0) edge node [above=10pt,right=5pt] {1} node [above=3pt,left=8pt] {2} (1)
	(0) edge node [below=10pt,right=5pt] {2} node [below=3pt,left=8pt] {1} (3)
	(1) edge node [above=10pt,left=5pt] {1} node [above=3pt,right=8pt] {1} (2)
	(3) edge node [below=10pt,left=5pt] {1} node [below=3pt,right=8pt] {2} (2);
\end{tikzpicture}  \qquad
\G^B_3 :
\begin{tikzpicture}[every loop/.style={}]
  \tikzstyle{every node}=[inner sep=-1pt]
  \node (0) at (0,.5) {$\circ$};
  \node (1) at (1,.5) {$\circ$};
  \node (2) at (2,0) {$\bullet$};
  \node (3) at (1,-.5) {$\circ$};
  \node (4) at (0,0) {$\circ$};
  \path[-] 
	(1) edge [loop above] (1)
	(0) edge node [above=4] {3 \quad 1} (1)
	(1) edge node [above=10pt,left=5pt] {1} node [above=3pt,right=8pt] {1} (2)
	(3) edge node [below=10pt,left=5pt] {1} node [below=3pt,right=8pt] {2} (2)
	(3) edge node [below=10pt,right=5pt] {2} node [below=3pt,left=8pt] {1} (4)
	(4) edge node [above=10pt,right=5pt] {0} node [above=3pt,left=8pt] {2} (1);
\end{tikzpicture}$$
Let $n=1$. Then any special 1-ball is adjacent to the 1-ball with center $b= \bullet$ which is not the central $1$-ball of special $j$- balls, $j= 0,1,2$.
\end{example}

\subsection{Sturmian colorings of bounded type} \label{sec:6}
In this subsection, we show that Proposition~\ref{prop:circular} is a more general phenomenon: the second statement of part (2) is a characterization of a Sturmian coloring of a bounded type.
Let $\phi$ be a Sturmian coloring of bounded type, i.e. 
the type set of each vertex is finite. 
By Theorem~\ref{thm:1.1}, we know that the quotient graph $X$ is an infinite ray. Denote the vertices of $X$ from the left by $(x_i)_{i \geq 0}$ as in Figure~\ref{fig:6}.

\begin{theorem}\label{thm:bounded1}    
Let $\phi$ be a Sturmian coloring.

\begin{enumerate}
\item The coloring $\phi$ is of bounded type if and only if either all the $\G^A_n$'s or all the $\G^B_n$'s are isomorphic for sufficiently large $n$.
%for some $n$, either $\G^A_{n+\ell}$ for $\ell \ge 0$ are all isomorphic to each other or $\G^B_{n+\ell}$ for all $\ell \ge 0$  are all isomorphic to each other.
\item Moreover, if $\G^A_{n}$ (resp. $\G^B_{n}$) are all isomorphic for sufficiently large $n$, then the quotient graph $X_\phi = \varinjlim \G^B_n$ (resp. $X = \varinjlim \G^A_n$).
\item Let $m$ be the smallest integer such that $\G^A_{n+1} \cong \G^A_n$ for all $n \geq m$. 
Then $\B_{m+i}(x_i)$ is special for all $i\geq0$. In particular, $\B_m(x_0)$ is the $m$-special ball.
\end{enumerate}
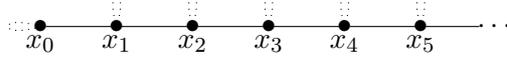
\begin{figure}
\begin{tikzpicture}[every loop/.style={}]
  \tikzstyle{every node}=[inner sep=-1pt]
  \node (5) at (-1,0) {$\bullet$} node [below=4pt] at (-1,0) {$x_0$};
  \node (6) at  (0,0) {$\bullet$} node [below=4pt] at (0,0) {$x_1$};
  \node (7) at  (1,0) {$\bullet$} node [below=4pt] at (1,0) {$x_2$};
  \node (8) at  (2,0) {$\bullet$} node [below=4pt] at (2,0) {$x_3$};
  \node (9) at  (3,0) {$\bullet$} node [below=4pt] at (3,0) {$x_4$};
  \node (10) at (4,0) {$\bullet$} node [below=4pt] at (4,0) {$x_5$};
  \node (11) at (5,0) {$\cdots$} ;

\tikzstyle{every loop}=   [-, shorten >=.5pt]
  \path[-] 
	(5)  edge  (6)
	(6)  edge (7)
	(7)  edge (8)
	(8)  edge (9)
	(9)  edge  (10)
	(10) edge (11);
  \path[dotted] 
		 (5) edge [loop left] (5)
		 (6) edge [loop above] (6)
		 (7) edge [loop above] (7)
		 (8) edge [loop above] (8)
		 (9) edge [loop above] (9)
		 (10) edge [loop above] (10);  
\end{tikzpicture}
\caption{Quotient graph of bounded type Sturmian coloring}
\label{fig:6}
\end{figure}

\end{theorem}

\begin{proof}   
We showed all the statements in Proposition~\ref{prop:circular} for cyclic Sturmian colorings.

%The $j$-th vertex from the unique end of $\G^B_{n+i}$ for $1 \le j \le i$ is the extension of $B_{n+j}$ which is the central $n+j$-ball of $S_{n+j-1}$. 
%\seon{strange... the vertices in $\G^B_{n+i}$ are $n+i$-balls!}
Now we prove part (1) for acyclic Sturmian colorings.
 If $\phi$ is of bounded type, then 
there exists a vertex $t \in V\T$ and some integer $m$ such that $t$ is not the center of the special $n$-ball for all $n > m$. 
For any $n_k > m+1$, 
\begin{equation}\label{eqn:star} 
\N_{n_k}(t) \in V \G^{\alpha_k}_{n_k} - V \G^{\overline{\alpha_k}}_{n_k} 
\end{equation} 
by the first statement of Lemma~\ref{lem:nk}.
We claim that $\alpha_{k+l} = \alpha_k$ for all $l>0$. Indeed, otherwise, for the minimal $l$ such that $\alpha_{k+l} =\overline{ \alpha_{k+l-1}}$,
$\N_{n_{k+l}-1}(t)$ is a vertex in $\G^{\alpha_{k+l}}_{n_{k+l}-1}$ thus $\N_{n_{k+l-1}}(t)$ is a vertex of $\G^{\overline{\alpha_{k+l-1}}}_{n_{k+l-1}}$, which is a contradiction to (\ref{eqn:star}).

Conversely, if $\G^B_n$ are all isomorphic for sufficiently large $n$, then by Proposition~\ref{lem:nk}, for any $\N_n(t) \in V\G^A_n - V \G^B_n$, $\N_m(t)$ is not special for all $m \geq n$. Thus $\phi$ is of bounded type.

For part (2), suppose that  there exists $m$ such that $\G^A_{m+\ell}$ is isomorphic for all $\ell \ge 0$.
Choose $t',t'' \in VX$ of distance larger than $|V\G^A_m|$. By Lemma~\ref{lem:8}, for $n_k >m$, the $n_k$-balls around vertices between $t',t''$ are in $\G^{\beta_k}_{n_k}$, thus $\beta_k =B$ for $n_k >m$.

Part (3) for acyclic colorings follows from Proposition~\ref{lem:nk}.
\end{proof}

\section*{Acknowledgement} We thank the anonymous referee for valuable comments.
We would also like to thank the hospitality of KIAS, of which both authors are associate members and where part of this work was done. 
The first author was supported by the National Research Foundation of Korea (NRF-2015R1A2A2A01007090) and the second author is supported by Samsung Science and Technology Foundation under Project No. SSTF-BA1601-03.


\begin{thebibliography}{19}

\bibitem{AR}  P. Arnoux and G. Rauzy, \emph{Repr\'esentation g\'eom\'etrique des suites de complexit\'e $2n+1$,} Bull. Soc. Math. France, \textbf{119} (1991), 199-215.
\bibitem{ALN} N. Avni, S. Lim and E. Nevo, \emph{On commensurator growth}, Israel J. of Math. \textbf{188} (2012), 259--279.

%\bibitem{BK} H. Bass and R. Kulkarni, \emph{Uniform tree lattices}, J. Amer. Math. Soc. \textbf{3} (1990), no. 4, 843--902.

\bibitem{BL} H. Bass and A. Lubotzky, \emph{Tree lattices}, Progress in Math., \textbf{176}, Birkhauser, 2000.

\bibitem{BM} M. Burger and S. Mozes, \emph{$CAT(-1)$-spaces, divergence groups and their commensurators}, J. Amer. Math. Soc. \textbf{9} (1996), 57--93.

\bibitem{Ca} J. Cassaigne, \emph{On a conjecture of J. Shallit}, Automata, languages and programming (Bologna, 1997), 693--704, Lecture Notes in Comput. Sci., \textbf{1256}, Springer, Berlin, 1997. 

%\bibitem{DP} X. Droubay and G. Pirillo, \emph{Palindromes and Sturmian words}, Theoret. Comput. Sci. \textbf{223} (1999) 73--85.

%\bibitem{Fe} Ferenczi, \emph{Complexity for finite factors of infinite sequences}, WORDS (Rouen, 1997). 
%Theoret. Comput. Sci. \textbf{218} (1999), no. 1, 177--195. 

\bibitem{Fo} N. Pytheas Fogg, \emph{Substitutions in Dynamics, Arithmetics and Combinatorics}, Edited by V. Berth�, S. Ferenczi, C. Mauduit and A. Siegel. Lecture Notes in Mathematics, \textbf{1794}. Springer-Verlag, Berlin, 2002. 

\bibitem{HM}  G.A. Hedlund and M. Morse, \emph{Symbolic dynamics II: Sturmian trajectories}, Amer. J. Math. \textbf{62} (1940) 1--42.

\bibitem{KL1} D.H. Kim and S. Lim, \emph{Subword complexity and Sturmian colorings of trees}, Ergod. Th. Dynam. Sys. Vol \textbf{35}, no. 2 (2015) 461--481.

\bibitem{KL2} D.H. Kim and S. Lim, \emph{Hyperbolic tessellation and colorings of trees}, Abstract and Applied Analysis, vol. 2013, ID 706496 (2013).

\bibitem{Loth}
M. Lothaire, Algebraic combinatorics on words. Encyclopedia of Mathematics and its Applications, 90.  Cambridge University Press, Cambridge, 2002.

\bibitem{LMZ} A. Lubotzky and S. Mozes, Zimmer, \emph{Superrigidity for the commensurability group of tree lattices.} Comment. Math. Helv. 69 (1994), no. 4, 523--548. 

%\bibitem{Paulin}
%F. Paulin, \emph{Groupe modulaire, fractions continues et approximation diophantienne en caract\'eristique $p$}, Proceedings of the Conference on Geometric and Combinatorial Group Theory, Part II (Haifa, 2000). 
%Geom. Dedicata \textbf{95} (2002), 65--85. 

%\bibitem{PR}
%Platonov and Rapinchuk, \emph{Algebraic groups and number theory}. Translated from the 1991 Russian original by Rachel Rowen,  Pure and Applied Mathematics, \textbf{139}. Academic Press, Boston, MA, 1994.

\bibitem{R82}
G. Rauzy, \emph{Suite \`a termes dans un alphabet fini,} In S\'eminaire de Th\'eorie des Nombres de Bordeaux, pp. 25.01-25.06, 1982-1983.

\bibitem{R85}
G. Rauzy, \emph{Mots infinis en arithmetiques,} M. Nivat and D. Perrin
(Eds) Automata on Infinite Word No. 192 in Lecture Notes Comp. Sci. pp165-171 (1985), Springer Verlag.


%\bibitem{Sa} L. Sadum, \emph{Topology of tiling spaces}, University lecture series, vol. \textbf{46}, AMS (2008).

\bibitem{Se} C. Series, \emph{The geometry of Markoff numbers}, Math. Intell. \textbf{7} (1958), 20--29.

\bibitem{Serre} J.-P. Serre, \emph{Trees}, Translated from the French by John Stillwell, Springer-Verlag, 1980.

\end{thebibliography}
\end{document}